\documentclass[11pt]{amsart}
\usepackage{mathrsfs, amsmath, amssymb, graphics,booktabs,diagbox,color,textcomp,picinpar,wrapfig,cutwin,longtable}
\usepackage{geometry}
\usepackage{qbordermatrix}

\definecolor{listinggray}{gray}{0.9}
\definecolor{lbcolor}{rgb}{0.9,0.9,0.9}
\usepackage{floatrow}
\newfloatcommand{capbtabbox}{table}[][\FBwidth]

\newtheorem{theorem}{Theorem}[section]
\newtheorem{lemma}[theorem]{Lemma}
\newtheorem{proposition}[theorem]{Proposition}
\newtheorem{question}[theorem]{Question}

\theoremstyle{definition}
\newtheorem{definition}[theorem]{Definition}
\newtheorem{example}[theorem]{Example}

\theoremstyle{remark}
\newtheorem{remark}[theorem]{Remark}

\numberwithin{equation}{section}

\begin{document}

\title{Geometrically  bounding 3-manifold, volume and Betti number}

\author{Jiming Ma}
\address{School of Mathematical Sciences \\Fudan University\\
	Shanghai 200433, China} \email{majiming@fudan.edu.cn}

\author{Fangting Zheng}

\address{School of Mathematical Sciences, Fudan University, Shanghai 200433,
	China  }

\email{fzheng13@fudan.edu.cn}



\date{Nov. 19,  2017}

\thanks{Jiming Ma was partially supported by  NSFC  11371094 and 11771088.}

\begin{abstract}

It is well known that an arbitrary closed orientable $3$-manifold can be realized as the unique boundary of a compact orientable  $4$-manifold, that is, any closed orientable $3$-manifold is  cobordant to zero. In this paper, we consider the   geometric cobordism problem: a hyperbolic $3$-manifold is geometrically bounding if it is  the only boundary of a totally geodesic hyperbolic 4-manifold.  However, there are very rare geometrically bounding closed hyperbolic 3-manifolds according to the previous research \cite{longR:2000,MN:1992}. Let $v \approx 4.3062\ldots$ be the volume of the regular right-angled hyperbolic dodecahedron in $\mathbb{H}^{3}$, for each $n \in \mathbb{Z}_{+}$ and each odd integer $k$ in $[1,5n+3]$, we construct a closed hyperbolic 3-manifold $M$ with $\beta^1(M)=k$ and $vol(M)=16nv$ that bounds a totally geodesic hyperbolic 4-manifold. The proof uses small cover theory over a sequence of linearly-glued  dodecahedra and some results of Kolpakov-Martelli-Tschantz \cite{KMT:2015}.
\end{abstract}

\maketitle

%% The correct journal style for \specialsection is all uppercase; a known bug
%% in amsart.cls prevents this, so input must be uppercase until it is fixed.
%\specialsection*{This is a Special Section Head}
%\specialsection*{THIS IS A SPECIAL SECTION HEAD}
%This is an example of a special section head%
%%%%%%%%%%%%%%%%%%%%%%%%%%%%%%%%%%%%%%%%%%%%%%%%%%%%%%%%%%%%%%%%%%%%%%%%
%\footnote{Here is an example of a footnote. Notice that this footnote
%text is running on so that it can stand as an example of how a footnote
%with separate paragraphs should be written.
%\par
%%%%%%%%%%%%%%%%%%%%%%%%%%%%%%%%%%%%%%%%%%%%%%%%%%%%%%%%%%%%%%%%%%%%%%%%

%---------------------------------------------------------------------1����
\section{\textbf{Introduction}}

\subsection{\textbf{Geometrically  bounding 3-Manifolds}}

It is an open question what kind of closed $n$-manifolds can bound  $(n+1)$-manifolds. Farrell and Zdravkovska once conjectured that every almost flat $n$-manifold bounds a $(n+1)$-manifold,  for example, see \cite{fz:1983, df:2016}.

There is a well-known result given by Rohlin in 1951 that $\Omega_3 =0$, which means every closed orientable 3-manifold $M$ bounds a compact orientable 4-manifold (for example, see Corollary 2.5 of \cite {Saveliev:1977}). Farrell and Zdravkovska once conjectured \cite{fz:1983} that every flat $n$-manifold $M$ is the cusp section of a one-cusped hyperbolic $(n+1)$-manifold. However, Long-Reid \cite{longR:2000} gave a negative answer to this stronger conjecture by showing that if $M$ is the cusp section of a one-cusped hyperbolic $4n$-manifold, then the $\eta$-invariant of $M$ must be an integer. 

Long-Reid also further studied what kind of 3-manifolds bound geometrically \cite{longR:2000}. If a hyperbolic $n$-manifold $M$ is the unique totally geodesic boundary of a hyperbolic $(n+1)$-manifold $N$, then we say $M$ \emph{bounds geometrically} or $M$ is a \emph{geometrically bounding} hyperbolic $n$-manifold. See also Ratcliffe-Tschantz \cite{RT:1998} for cosmological motivations of studying geometrically bounding hyperbolic 3-manifolds. 
Geometrically bounding 3-manifolds are difficult to seek since there are very few examples of hyperbolic 4-manifolds. Moreover, Long-Reid showed \cite{longR:2000} that if a hyperbolic closed 3-manifold $M$ is geometrically bounding, then the $\eta$-invariant $\eta(M) \in  \mathbb{Z}$. By Theorem 1.3 of Meyerhoff-Neumann \cite{MN:1992}, the set of $\eta$-invariants of all hyperbolic 3-manifolds is dense in $\mathbb{R}$. So, to some extend, geometrically bounding 3-manifolds are very rare in the set of hyperbolic 3-manifolds. As far as we know,  the following question is wide open:

\begin{question}
 For a  hyperbolic closed $3$-manifold $M$ with $\eta$-invariant $\eta(M) \in  \mathbb{Z}$, is there a hyperbolic totally geodesic $4$-manifold $N$ such that  $\partial N = M$?
\end{question}

It is well known from Jorgensen-Thurston's Dehn surgery  theory \cite{Thurston:1978} that, given  $x \in \mathbb{R}_{+}$,  there are only finitely many (maybe no) hyperbolic 3-manifolds with the given volume $x$. The best result in this direction is given by Millichap \cite{ Millichap:2015}. More precisely, if considering the function
\begin{center}
$f(x)$=sup$\{n|$ there are $n$ different hyperbolic 3-manifolds  with the same volume $v \leq x\}$,
\end{center}
then Jorgensen-Thurston theory implies $f(x)$ is finite and Millichap \cite{ Millichap:2015} showed $f(x)$  grows at least factorially.

In this paper, we consider how many geometrically bounding 3-manifolds are there with the same volume. That is, we are focusing on the function
\begin{center}
$f_{b}(x)$=sup$ \{n|$   there are $n$ different geometrically bounding  3-manifolds  with the same volume $v \leq x\}$.
\end{center}

Building on \cite{KMT:2015} and small cover theory, we have:

\begin{theorem} \label{theorem: bounding} Let $v \approx 4.3062\ldots$ be the volume of the regular right-angled hyperbolic dodecahedron in $\mathbb{H}^{3}$, for each $n \in \mathbb{Z}_{+}$ and each odd integer $k$ in $[1,5n+3]$, there is a closed hyperbolic $3$-manifold $M$ with $\beta^1(M)=k$ and $vol(M)=16nv$ that bounds a totally geodesic hyperbolic 4-manifold.
\end{theorem}
By this theorem,  for each $n \in \mathbb{Z}_{+}$,there are at least $\lceil\frac{5n+3}{2}\rceil$ many geometrically bounding 3-manifolds of the same volume $16nv$. Namely $f_{b}(x)$ defined above grows at least linearly. And we believe the growth curve ought to be much steeper.

See Ratcliffe-Tschantz \cite{RT:2000}  for  counting questions on the number of totally geodesic hyperbolic 4-manifolds with the same 3-manifold boundary $M$, and  Slavich \cite {Slavich:2015, Slavich:2015b} for other  topics on geometrically bounding 3-manifolds. See also the recent paper \cite{KRS:2017}  by Kolpakov-Reid-Slavich for geodesically embedding questions about hyperbolic manifolds and the relation between the geodesically embedding and geometrically bounding is subtle.

\subsection{\textbf{Small covers}}
Small covers, or Coxeter orbifolds, were studied by Davis and Januszkiewicz in \cite{dj:1991}, see also \cite{Vesnin:1987}.  They are a class of  $n$-manifolds  which  admit  locally standard $\mathbb{Z}_2^{n}$-actions, such that the orbit spaces are  $n$-dimensional simple polytopes. The algebraic and topological properties of a small cover are closely related to the combinatorics of the orbit polytope  and the coloring on the boundary of that polytope.  For example, the $mod~2$ Betti number $\beta_{i}^{(2)}$ of a small cover $M$ over the polytope $L$ agrees with $h_{i}$, where $h=(h_{0}, h_{1},\ldots, h_{n})$ is the $h$-vector of the polytope $L$ \cite{dj:1991}.

Those manifolds admitting locally standard $\mathbb{Z}_2^{k}$-actions would form a class wider  than small covers. We will say more about this topic in Section 2, and here we just give the definition for the convenience of stating Theorem \ref{theorem: betti}.

\begin{definition} Let $L$ be an $n$-dimensional simple polytope, $\mathcal{F}$ be the set of co-dimensional one faces of $L$. Such faces are called as facets. A \emph{$\mathbb{Z}_2^k$-coloring} is a map $\lambda:\mathcal{F} \longrightarrow \mathbb{Z}_2^k$ satisfying $\lambda(F_1)$, $\lambda(F_2),\ldots,\lambda(F_n)$ generate a subgroup of $\mathbb{Z}_2^k$ which is isomorphic to $\mathbb{Z}_2^n$,
 when the facets $F_1$, $F_2, \cdots, F_n$  are sharing a common vertex.
 \end{definition}

Conversely, through Proposition 1.7 of \cite{dj:1991}, from a $\mathbb{Z}_2^k$-coloring $\lambda$ and a principal $\mathbb{Z}_2^k$-bundle over an $n$-dimensional simple polytope $L$, we can get a unique closed $n$-manifold $M$. In particular, we can use $2^k$ copies of $L$, namely $L\times\mathbb{Z}_2^k$, to construct a quotient space $M(L, \lambda)$ under the following equivalent relation:
\begin{equation}
(p,g_{1}) \sim(q,g_{2}) \Leftrightarrow
\begin{cases} p=q~~and~~g_{1}=g_{2} ,  \hspace{0.8cm} \text{if $p\in {\rm Int~}L$}\\
 p=q ~~and~~ g_{1}g^{-1}_{2} \in G_{f}, ~~\text{if $p\in \partial P$}.
\end{cases}
\end{equation}
Here for a face $f$ of the simple polytope $L$,  $G_{f}$ is the subgroup generated by $\lambda(F_{i_1}), \lambda(F_{i_2}), \ldots ,\lambda(F_{i_k})$, where $f=F_{i_1}\cap F_{i_2}\cap \ldots \cap F_{i_k}$ is the unique face that contains $p$ as an interior point and $F_{i_j}\in\mathcal{F}$.
It is easy to see $M(L, \lambda)$  is a closed $n$-manifold.

%In this paper, we only consider closed 3-manifold associated to a $(\mathbb{Z}_2)^3$-coloring and the trivial principal $(\mathbb{Z}_2)^3$-bundle over $N$.

A simple example is that if we color the four co-dimensional one faces of a tetrahedron by $e_{1}$, $e_{2}$, $e_{3}$ and  $e_{1}+e_{2}+e_{3}$ respectively, where  $e_{1}$, $e_{2}$ and $e_{3}$ are the standard basis vectors of $\mathbb{Z}_2^3$. Then from the above construction, we can get the  closed orientable 3-manifold $\mathbb{RP}^{3}$. It should be noted that a tetrahedron admits a unique right-angled spherical structure, and these spherical structures on copies of the tetrahedron are glued together to build up the unique  spherical structure on $\mathbb{RP}^{3}$.  This point of view appeals in this paper.

In the following of this section, we suppose $P$ to be the regular right-angled hyperbolic dodecahedron in $\mathbb{H}^{3}$ with twelve 2-dimensional facets. We let $nP$ be the linearly-gluing of $n$ copies of  $P$.  It is obvious that $nP$ has $12$ pentagonal facets and $5n-5$ hexagonal facets. See Section 2.5 for more details.

\begin{definition} From a $\mathbb{Z}_2^3$-coloring $\lambda$ on the polytope $nP$, we obtain a natural $\mathbb{Z}_2^4$-coloring $\delta$ on  $nP$ by following manners: Supposing $\{e_1,~ e_2,~e_3,~e_4\}$ is the standard basis of $\mathbb{Z}_2^4$. For each facet $F$ of $nP$, if $\lambda (F)=\Sigma^{3}_{i=1} x_{i} e_{i}$, $x_{i}= 1~~ or~~0 $,  then we take  $\delta (F)=\Sigma^{4}_{i=1} x_{i} e_{i}$, where $x_{4}=1+\Sigma^{3}_{i=1} x_{i}$ \emph{mod} 2. A $\mathbb{Z}_2^3$-coloring $\lambda$ is called \emph{non-orientable} when  the 3-manifold $M(nP, \lambda)$ is non-orientable. Furthermore, if the 3-manifold $M(nP, \lambda)$ is non-orientable, then its natural  $\mathbb{Z}_2^4$-coloring $\delta$ is called an \emph{admissible extension} of $\lambda$ or a \emph{natural $\mathbb{Z}_2^4$-coloring associated to $\lambda$} (say a \emph{natural $\mathbb{Z}_2^4$-extension} of $\lambda$ for short). It can be shown that $M(nP, \delta)$ is the orientable double cover of $M(nP, \lambda)$ when $\lambda$ is non-orientable.
	\end{definition}
The following is our main technical theorem.
\begin{theorem} \label{theorem: betti} For each $n \in \mathbb{Z}_{+}$ and each odd integer $k \in [1,5n+3]$, there is a non-orientable $\mathbb{Z}_2^3$-coloring $\lambda$ on the polytope $nP$, such that the first Betti number of the orientable $3$-manifold $M(nP, \delta)$ is $k$, where $\delta$  is the natural  $\mathbb{Z}_2^4$-extension of $\lambda$.
\end{theorem}

As one orientable 3-manifold $M$ may double cover many non-orientable 3-manifolds, we must show the orientable 3-manifolds under consideration are not homeomorphic in order to prove Theorem \ref{theorem: bounding}. And here the first  Betti number is the classification index we adopt to determine the lower bound.

Now on one hand, based on Theorem \ref{theorem: betti}, for a given $n \in \mathbb{Z}_{+}$ and an odd integer $k\in [1, 5n+3]$, we can construct an orientable 3-manifold $M(nP,\delta)$ whose first  Betti number is exactly $k$. Moreover, we believe that the inverse side  is true as well. That means the  first Betti numbers of $M(nP,\delta)$ are definitely the odd integers in $[1,5n+3]$, where $\delta$ is the natural $\mathbb{Z}_2^4$-extension of $\lambda$ and $\lambda$ is among all possible non-orientable $\mathbb{Z}_2^3$-colorings over the polytope $nP$. The ``only if" part is only checked by programming so far and we haven't proof it precisely yet. So we here just list them together as a question:
\begin{question} \label{question: betti} Is $k \in \mathbb{Z}$ the first  Betti number of $M(nP, \delta)$, where $\delta$ is the natural $\mathbb{Z}_2^4$-coloring associated to a non-orientable $\mathbb{Z}_2^3$-coloring $\lambda$ on $nP$, if and only if $k$ is odd in $[1,5n+3]$?

\end{question}
%------------------����֤����Ҫ�ټ���
\noindent\textbf{Proof of Theorem \ref{theorem: bounding}}\quad For a non-orientable $\mathbb{Z}_2^3$-coloring $\lambda$ on the polytope $nP$, there is a natural $\mathbb{Z}_2^4$-extension $\delta$ on $nP$. Both $M(nP, \delta)$ and $M(nP, \lambda)$ are 3-manifolds and $M(nP, \delta)$ is the orientable double cover of  $M(nP, \lambda)$. See Section 3, in particular,  Proposition \ref{remark:1} for more details.

Now there are two methods to show $M(nP, \delta)$ is geometrically bounding: firstly we may use Proposition 2.9 in  \cite{KMT:2015} to extend  the $\mathbb{Z}_2^4$-coloring $\delta$ on the 3-dimensional polytope $nP$ to a $\mathbb{Z}_2^5$-coloring $\varepsilon$  on the 4-dimensional polytope  $nE$. Here $nE$ is a 4-dimensional polytope obtained by linearly-gluing $n$  copies of the hyperbolic right-angled 120-cell $E$. Then $M(nE, \varepsilon)$ is an orientable hyperbolic 4-manifold in which $M(nP, \lambda)$ can be embedded. Secondly since $M(nP, \delta)$ is the orientable double cover of
$M(nP, \lambda)$, we can thus apply Corollary 8  of  \cite{Mar:2015} directly because $M(nP, \delta)$ admits  a fixed-point free orientation-reversing involution. Therefore $M(nE, \varepsilon)-M(nP,\lambda)$ is a totally geodesic hyperbolic 4-manifold with boundary $M(nP,\delta)$.

Now from Theorem \ref{theorem: betti}, Theorem \ref{theorem: bounding} follows.  $\square$
%------------------------------����ȫ

%Our paper is organized as followings:

\textbf{Outline of the paper}: In Section 2, we give some preliminaries on algebraic theory  of small covers. In Section 3, we show Lemma \ref{lemma: key}, which is the key of our paper. In Section 4 and Section 5, we prove  Theorem \ref{theorem: betti} for $n$ is even and odd respectively.

%-------------------------------------------------�ڶ���
\section{\textbf{Preliminaries}}
%-------------------------------------------------2.1
\subsection{Polyhedral product}

Let $K$ be an abstract simplicial complex with ground set $[m]:=\{1,2,\ldots,m\}$, so we have $\emptyset \in K \subset  2^{[m]}$. If $\sigma \in K$, then for all $\tau \subset \sigma$ we have $\tau \in K$. We associate $m$ pairs of topological spaces, $(X_i,A_i)_{i=1}^m$, to $K$.  Then
the corresponding \emph{polyhedral product} $(\underline{X},\underline{A})^K$is defined as 
$$(\underline{X},\underline{A})^K:=
\displaystyle\bigcup_{\sigma\in K}D(\sigma),$$
where $D(\sigma)=\prod_{i=1}^mY_i=\left\{
\begin{array}{rcl}X_i &\mbox{if}& i\in\sigma, \\
 A_i &\mbox{if}&otherwise.
 \end{array}\right.$

\noindent That is to say $(\underline{X},\underline{A})^K=\{x\in \prod_{i=1}^m X_i \arrowvert  \sigma_x \in K\}$, when $\sigma_x=\{i \arrowvert  x_i \in X_i \backslash A_i\}$ are defined for every $x=(x_i)_{i=1}^m\in \prod_{i=1}^m X_i $. If $(X_i, A_i)_{i\in[m]}$ are the same pair $(X,A)$, then $(\underline{X},\underline{A})^K$ is abbreviated  as $(X, A)^K$. Specially, $(D^1, S^0)^K$ is defined as a \emph{real moment-angle complex}, denoted by $\mathbb{R}\mathcal{Z}_K$. 

For example, let $K$ to be the 1-skeleton of the 2-simplex, namely an abstract simplicial complex $2^{[3]} \backslash \{1,~2,~3\}$, then we have
${(D^1,S^0)}^{K}=\mathbb{R}\mathcal{Z}_{K}=D^1\times D^1\times S^0\bigcup S^0\times D^1\times D^1\bigcup D^1\times S^0\times D^1=\partial(D^1\times D^1\times D^1)\cong S^2.$ 

%It is natural to question that under what conditions the monment-angle complex can inherit a differential structure.

By Davis \cite{Davis:2008} and L. Cai \cite{2licai:2013}, $\mathbb{R}\mathcal{Z}_K$ is a topological $n$-manifold if and only if  $K$  is a generalized homology $(n-1)$-sphere  $K$  with $\pi_1(\arrowvert K\arrowvert)=0$ when $n \neq 1, 2$.
Specially, assuming $L$ to be an $n$-dimensional simple polytope and $K$ is the dual of the boundary of $L$, then $\mathbb{R}\mathcal{Z}_K$ is definitely to be a topological $n$-manifold. Then for simple polytope, there is an equivalent but more practical way in describing the moment-angle manifold by using the language of coloring and conducting the re-construction procedure.

%----------------------------------------------------------------2.2
\subsection{Re-construction procedure} For an $n$-dimensional simple polytope $L$, let $\mathcal{F}(L)=\{F_1,F_2,\ldots,F_m\}$ be  the set of co-dimensional one faces of $L$. Taking $\{e_1, e_2,\ldots, e_m\}$ to be a basis in $\mathbb{Z}_2^m$. Then we define a \emph{$\mathbb{Z}_2^m$-coloring characteristic function} $$\lambda:\mathcal F(L)=\{F_1,F_2,\dots,F_m\}\longrightarrow \mathbb{Z}_2^m$$ by mapping $F_i$ to $e_i$. $\lambda$ is also named as a \emph{$\mathbb{Z}_2^m$-coloring} for short. Because the images of these facets are the basis elements, it is naturally satisfied that $\lambda(F_1),\lambda(F_2),\dots,\lambda(F_n)$ generate a  subgroup  of $\mathbb{Z}_2^m$, which is isomorphic to $\mathbb{Z}_2^n$, when the facets $F_{1}, F_{2}, \ldots, F_{n}$ share a common vertex.

Then we can construct $M(L,\lambda):=L\times \mathbb{Z}_2^m/\sim $ by the following equivalent relation:

\begin{center}
$(x,g_1)\sim (y,g_2){\Longleftrightarrow }
\left\{
\begin{array}{rcl}x=y ~~ and  ~~~ g_1=g_2  \hspace{0.7cm} &\mbox~~~~~{{\rm if}}& x\in\ {\rm Int~}L, \\
x=y ~~ and ~~~ g_1^{-1}g_2 \in G_f&\mbox ~~~~~{{\rm if}}& x \in \partial L,
\end{array}\right.
$
\end{center}
where $f=F_{i_1}\cap\dots\cap F_{i_{n-k}}$ is the unique co-dimensional $(n-k)$-face  that contains $x$ as an interior point, and $G_f$ is the subgroup generated by $\lambda(F_{i_1}), \lambda(F_{i_2}), \ldots, \lambda(F_{i_{n-k}})$. It is easy to proof that $M(L,\lambda)$ is exactly the real moment-angle complex $\mathbb{R}\mathcal{Z}_K$ over its dual $K=(\partial L)^{*}$. Hence we also denote the manifold by $M(K, \lambda)$. We use Example \ref{example:2} to  illustrate the homeomorphism between the two spaces according to the two definitions respectively. And by replacing the facet set $\mathcal{F}(L)$ with the vertice set of $K$, the characteristic function $\lambda$ can also be seen as being defined on the simplicial complex $K$.

\begin{example}\label{example:2} Defining a $\mathbb{Z}_2^3$-coloring characteristic function $\lambda$ on the 2-simplex as show in Figure \ref{figure:2example1}, namely the characteristic function is
\begin{center}
	$\lambda:\{\{a,b\},\{b,c\},\{a,c\}\}\rightarrow\{(1,0,0), (0,1,0), (0, 0, 1)\}$,
	
	$(a,b)\mapsto (1,0,0)$;	$(b,c)\mapsto (0,1,0)$;	$(a,c)\mapsto (0, 0, 1)$,
\end{center}
where $(1,0,0)=e_1$, $(0,1,0)=e_2$ and $(0,0,1)=e_3$ are the standard basis vectors of $\mathbb{Z}_{2}^{3}$.

\begin{figure}[H]
	\scalebox{0.5}[0.5]{\includegraphics {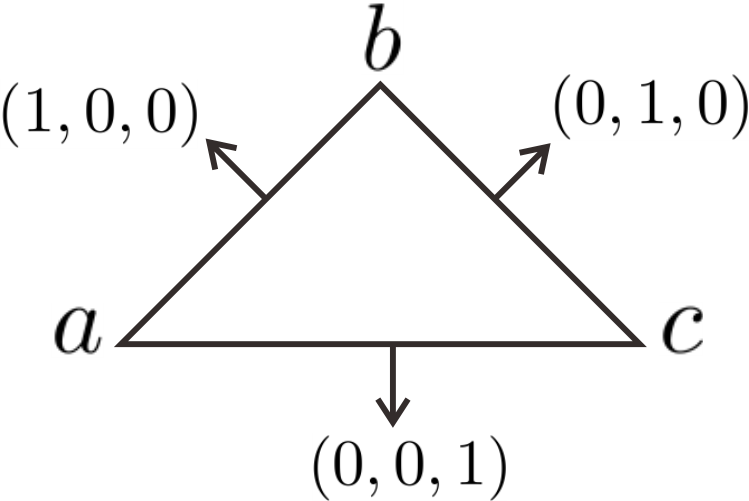}}
	\caption{Characteristic function of Example \ref{example:2}.} \label{figure:2example1}
\end{figure}

Now we will have eight polytopes, namely $\triangle^2\times \mathbb{Z}_2^3$, as shown in Figure \ref{figure: 2example2}.
\begin{figure}[H]
	\scalebox{0.8}[0.8]{\includegraphics {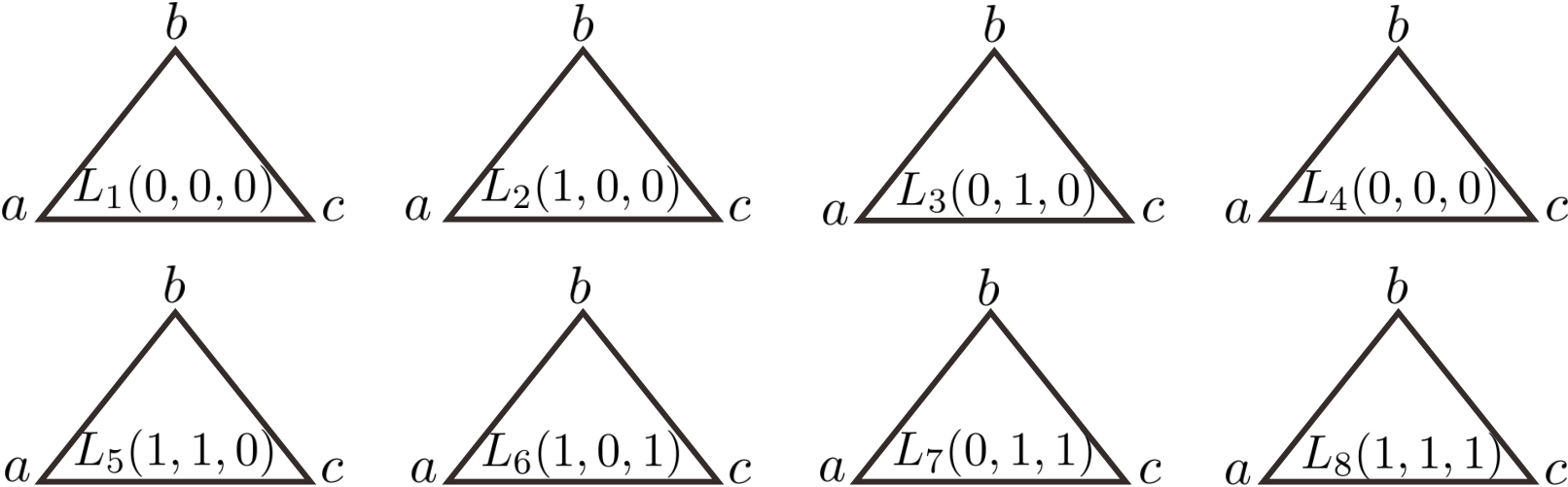}}
	\caption{$\triangle^2\times \mathbb{Z}_3^2$ of Example \ref{example:2}.}\label{figure: 2example2}
\end{figure}

If $p\in(a,b)$, $(p,g_1)\sim (q,g_2)${$\Longleftrightarrow $}
$\left\{
\begin{array}{rcl} p=q ~~{\rm (same\ location)}\\
g_1-g_2\in \{(1,0,0)\} \ \ \
\end{array}\right.
$, then gluing $L_1=\triangle \times (0,0,0)$ and $L_2=\triangle \times (1,0,0)$, $L_3=\triangle \times (0,1,0)$ and $L_5=\triangle \times (1,1,0)$, $L_4=\triangle \times (0,0,0)$ and $L_6=\triangle \times (1,0,1)$, $L_7=\triangle \times (0,1,1)$ and $L_8=\triangle \times (1,1,1)$ together respectively along the edge $(a,b)$ as shown in Figure  \ref{figure: 2example3}.
\begin{figure}[H]
	\scalebox{0.8}[0.8]{\includegraphics {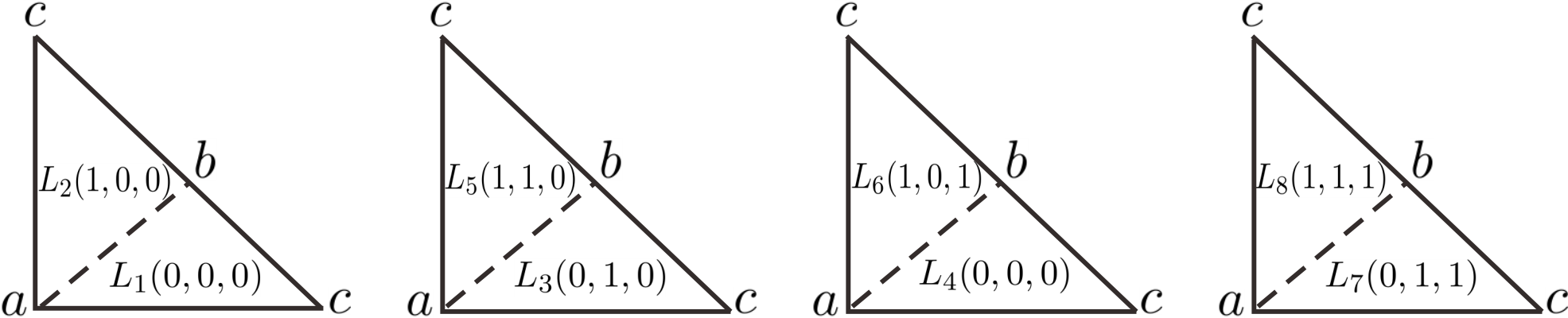}}
	\caption{The first gluing  of $L_i$.}\label{figure: 2example3}
\end{figure}

Similarly, if $p\in(b,c)$, $(p,g_1)\sim (q,g_2)${$\Longleftrightarrow $}
$\left\{
\begin{array}{rcl} p=q~~{\rm (same\ location)}\\
g_1-g_2\in \{(0,1,0)\} \ \ \
\end{array},\right.$
 gluing $L_1$ and $L_3$, $L_2$ and $L_5$, $L_4$ and $L_7$, $L_6$ and $L_8$ together respectively along the edge $(b,c)$ as shown in Figure \ref{figure: 2example4}.
\begin{figure}[H]
	\scalebox{0.8}[0.8]{\includegraphics {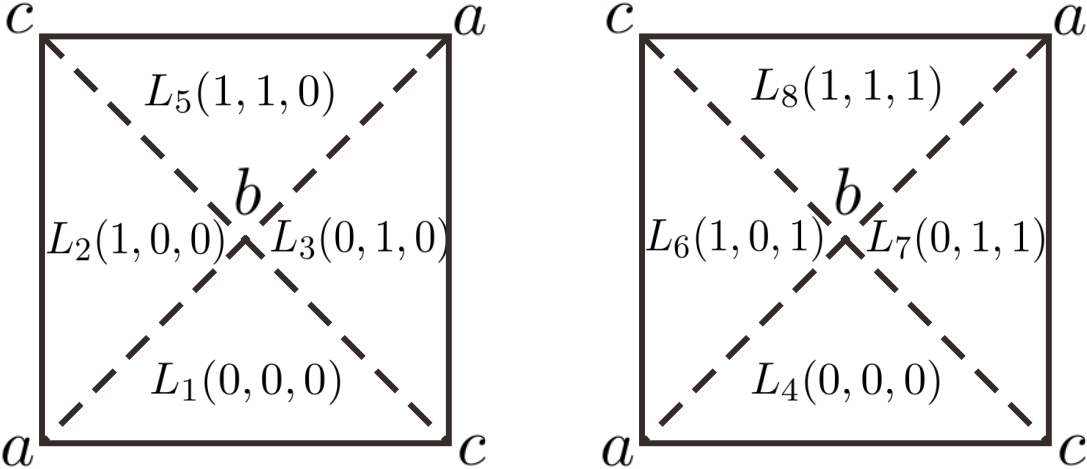}}		
	\caption{The second gluing of $L_i$.}\label{figure: 2example4}
\end{figure}

Finally, if $p\in(a,c)$, $(p,g_1)\sim (q,g_2)${$\Longleftrightarrow $}
$\left\{
\begin{array}{rcl} p=q~~{\rm (same\ location)}\\
g_1-g_2\in \{(0,0,1)\}  \ \ \
\end{array}\right.
$, $L_1$ and $L_4$, $L_2$ and $L_6$, $L_3$ and $L_7$, $L_5$ and $L_8$ are glued along the edge $(a,c)$ as shown in Figure \ref{figure: 2example5}. Now we additionally adopt a coordinate for a more precise description. Here  $L_4$, $L_6$, $L_2$, $L_1$ are the four frontier ones that lie in the quadrants clockwise with  non-negative $x$-coordinates. And $L_8$, $L_7$, $L_5$, $L_3$ are the four back polytopes lying clockwise in the quadrants with non-positive $x$-coordinates .
\begin{figure}[H]
	\scalebox{0.9}[0.9]{\includegraphics {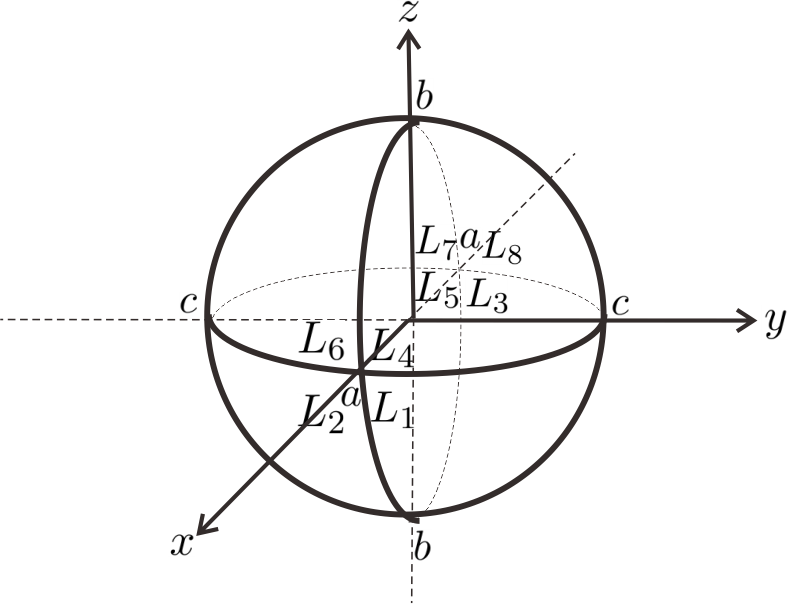}}
	\caption{The third gluing of $L_i$.}\label{figure: 2example5}
\end{figure}
Thus we have $M(\Delta^2,\lambda)\approx \mathbb{S}^2\approx\mathbb{R}\mathcal{Z}_K$, where $K$ is the 1-skeleton of 2-simplex.
\end{example}
%---------------------------------------------------------2.3
\subsection{Buchstaber invariant and real toric manifolds} Let $L$ be an $n$-dimensional simple polytope. From the construction of the  real moment-angle complex $\mathbb{R}\mathcal{Z}_K$,  we can easily see that there is a natural $\mathbb{Z}_2^m$-action over $\mathbb{R}\mathcal{Z}_K$, where $m$ is the cardinal number of the facet set $\mathcal{F}(L)$. The maximal rank among subgroups of $\mathbb{Z}_2^m$ that acts freely on $\mathbb{R}\mathcal{Z}_K$ is called the \emph{Buchstaber invariant} and denoted by $S_\mathbb{R}(L)$. Now we use $H^r$ to represent a subgroup of $\mathbb{Z}_2^m$ that acts on $\mathbb{R}\mathcal{Z}_K$ freely, where $r$ is the rank of $H^r$ satisfying $0\leq r\leq S_\mathbb{R}(L)$. Then we can have a smooth closed manifold  $\mathbb{R}\mathcal{Z}_K/{H^r}$ by quotient. Those smooth closed manifolds obtained by quotienting free $\mathbb{Z}_2^r$-actions from the real moment-angle manifolds are called \emph{real toric manifolds}. If $S_\mathbb{R}(L)=m-n$, $\mathbb{R}\mathcal{Z}_{K}/H^{(m-n)}$ is named as a \emph{small cover}. If $L$ is a 3-dimensional simple polytope, by the Four Color Theorem, $S_\mathbb{R}(L)=m-3$. Namely small covers can always be realized over any
3-dimensional simple polytope.

\begin{table}[H]
	%\caption{\label{tab:test}ʾ������}
	\begin{tabular}{lcl}
		\toprule
		%\midrule
		$\mathbb{Z}_2^m/H^r\curvearrowright$& $\mathbb{R}\mathcal{Z}_K/H^r$ & \\
		 &$\downarrow$ &real toric manifolds ($0\leq r\leq S_\mathbb{R}(L)$)\\
		&$L$& \\
		\bottomrule
	\end{tabular}
\end{table}

Then we have a short exact sequence and can define the \emph{$\mathbb{Z}_2^{m-r}$-coloring characteristic function} $\lambda_{H^r}=q\circ\lambda\circ i_d^{-1}
$ through a commutative diagram as shown below, where $q$ is the quotient map and $i_d$ is  the identity.

\begin{table}[H]
	%\caption{\label{tab:test}ʾ������}
	\begin{tabular}{ccccccccc}
	\toprule
	&&&& $\mathcal{F}(L)$ & $\xrightarrow{~~i_d~~}$ &  $\mathcal{F}(L)$ &&\\	
	&&&& $\downarrow$ $\lambda$ & $\curvearrowright$ & $\downarrow$ $\lambda_{H^r}$ &&\\
 0 & $\longrightarrow$ & $H^r$ & $\hookrightarrow$ & $\mathbb{Z}_2^m$ & $\xrightarrow{~~q~~}$ & $ \mathbb{Z}_2^{m-r}$ & $\longrightarrow$ &0\\
 	\bottomrule
	\end{tabular}
\end{table}

 The free action requirement ensures that the non-singularity condition always holds at every vertex. That means for each vertex $v=F_{i_1}\cap  F_{i_2}\cap\dots\cap F_{i_n}$, Span$\{\lambda_H (F_{i_1}), \lambda_H (F_{i_2}), \cdots,\lambda_H(F_{i_n})\}=\mathbb{Z}_2^{n}$. The re-construction procedure can be applied parallelly to all real toric manifolds. Thus there is a one-to-one  correspondence between  real toric manifolds $\{\mathbb{R} \mathcal{Z}_K/H^r~ \vert~ H^r < \mathbb{Z}_2^m ~ {\rm and~acts ~freely~on~} ~ \mathbb{R} \mathcal{Z}_K\}$ and the set of pairs of polytopes and characteristic functions $\{(L, \lambda_{H^r})\}$. We always use $M(L,\lambda_{H^r})$ to denote the corresponding real toric manifold.

%---------------------------------------------------------------------2.4
\subsection{Algebraic topology of $\mathbb{R}\mathcal{Z}_K/H^r$}
In \cite{dj:1991}, Davis and Januszkiewicz have proven that the $\mathbb{Z}_2$-coefficient cohomology groups of a small cover depend only on the polytope and its characteristic function. In 2013, Li Cai gave  a method to calculate the $\mathbb{Z}$-coefficient cohomology groups of $\mathbb{R}\mathcal{Z}_K$ \cite{2licai:2013}. Based on the work of Cai and Suciu-Trevisanon's result on rational homology groups of real toric manifolds  \cite {Suciu:13, SuciuT:12}, Choi and Park then gave a formula of the cohomology groups of real toric manifolds \cite{ChoiP:2013}, which can also be viewed as a combinatorial version of Hochster Theorem \cite{Hochster:1977}.

Let $K$ be a simplicial complex on $[m]$. We have  a bijective map $\varphi:\mathbb{Z}_2^m\longrightarrow 2^{[m]}$ such that the $i$-th entry of $v\in \mathbb{Z}_2^m$ is nonzero if and only if $i\in \varphi(v)$, where $2^{[m]}$ denotes the power set of $[m]$. Let $\lambda$ be a $\mathbb{Z}_2^n$-coloring characteristic function, then  the binary matrix $\Lambda_{(n\times m)}=(\lambda(F_1), \lambda(F_2), \ldots, \lambda(F_m))$ is called as \emph{characteristic matrix}. We denote $row\Lambda$ to be the $\mathbb{Z}_2$-space generated by the $n$ rows of $\Lambda$, namely the \emph{row space} of the  characteristic matrix $\Lambda$. Then we have the following theorem:

\begin{theorem}  \label{theorem: ChoiP} \emph{(Choi-Park \cite{ChoiP:2013})} For a coefficient ring $G$, 
	$$H^p(M(K,\lambda);G)\cong\underset{\varphi^{-1}(\omega) \in row \Lambda, ~ \omega\subseteq[m]}\oplus\widetilde{H}^{p-1}(K_\omega;G), $$
	where $K_{\omega}$ is the full sub-complex of $K$ by restricting to $\omega \subset [m]$. In particular,  $$\beta^i(M(K,\lambda);G)=\underset{\omega \in row \Lambda}\sum{\widetilde\beta^{i-1}}(K_\omega;G), $$
where $\beta^{-1}(K_\omega)=\left\{
\begin{array}{rcl}1 &\mbox{{\rm if}}& K_\omega=\emptyset, \\
0 &\mbox{{\rm if}}&otherwise
\end{array}\right.$ is defined particularly. 
\end{theorem}

Here every full-subcomplex $K_\omega$ is represented by a vector $\varphi^{-1}(\omega)$, which is actually an element of the row space  $row\Lambda$. Such kind of vector is called the \emph{representative} of $K_\omega$.

By means of Theorem \ref{theorem: ChoiP}, we can calculate the Betti numbers of $M(K, \lambda)$ only through the combinatorial information of the colored polytope and the row space of its characteristic matrix. The following is a simple example.

\begin{example}\label{example:3} Calculating the Betti numbers of the Klein bottle $S=M(L,\lambda)$.

The left side of Figure \ref{figure: 3example1}  is a colored  2-dimensional square $L$  and the right side  is the corresponding coloring on the dual of its boundary  $K=(\partial L)^*$.
\begin{figure}[H]
	\scalebox{0.7}[0.7]{\includegraphics {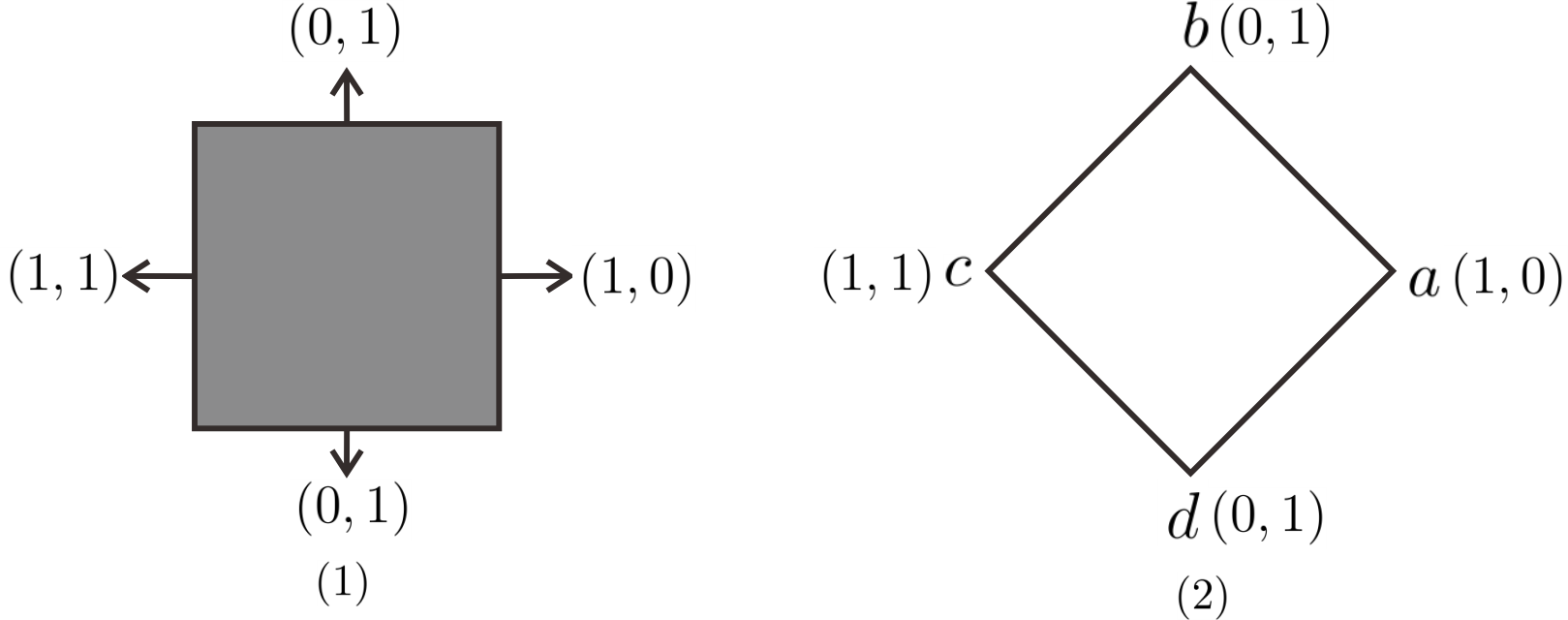}}	
	\caption{Colored square for Example \ref{example:3}.}\label{figure: 3example1}
\end{figure}

Then the row space is  $Row \Lambda=\langle(1, 0, 1, 0), (0, 1, 1, 1)\rangle=\{(0,0,0,0), (1,0,1,1), (0,1,0,1), (1,1,1,0)\}$.

For $\omega_1=(0, 0, 0, 0)$,  $K_{\omega_1}=\emptyset$. By definition $\beta^{-1}(K_{\omega_1})=1$. Thus $\widetilde\beta^{0}(K_{\omega_1})=\widetilde\beta^{0}(\emptyset)=0$, $\widetilde\beta^{1}(K_{\omega_1})=\beta^{1}(\emptyset)=0$.

For $\omega_2=(1, 0, 1, 1)$, then $K_{\omega_2}$ is as shown in Figure \ref{example:klein} (1). So $\widetilde\beta^{0}(K_{\omega_2})=0$, $\widetilde\beta^{1}(K_{\omega_2})=0$.

For $\omega_3=(0, 1, 0, 1)$, then $K_{\omega_3}$ is as shown in  Figure \ref{example:klein} (2). So $\widetilde\beta^{0}(K_{\omega_3})=1$, $\widetilde\beta^{1}(K_{\omega_3})=0$.

For $\omega_4=(1, 1, 1, 0)$, then $K_{\omega_4}$ is as shown in  Figure \ref{example:klein} (3). So $\widetilde\beta^{0}(K_{\omega_4})=0$, $\widetilde\beta^{1}(K_{\omega_4})=0$.

\begin{figure}[H]
	\scalebox{0.7}[0.7]{\includegraphics {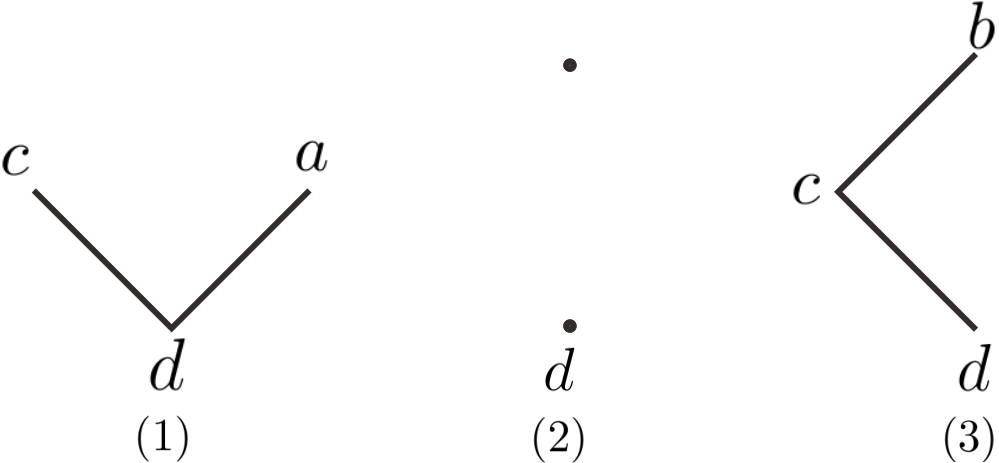}} 
	\caption{$\beta_1$ of a Klein bottle.}\label{example:klein}
\end{figure}

Then from Theorem \ref{theorem: ChoiP}, we have:
 $$\beta^0(S)=\beta^{-1}(K_{\omega_1})=\beta^{-1}(\emptyset)=1;$$
 $$\beta^1(S)=\widetilde\beta^{0}(K_{\omega_1})+\widetilde\beta^{0}(K_{\omega_2})+\widetilde\beta^{0}(K_{\omega_3})+\widetilde\beta^{0}(K_{\omega_4})=0+0+1+0=1;$$
 $$\beta^2(S)=\widetilde\beta^{1}(K_{\omega_1})+\widetilde\beta^{1}(K_{\omega_2})+\widetilde\beta^{1}(K_{\omega_3})+\widetilde\beta^{1}(K_{\omega_4})=0+0+0+0=0.$$

 That coincides with the well-known result of rational homology groups of  the Klein bottle.

\end{example}
%----------------------------------------------------------------------------2.5
\subsection{Object polytopes $nP$ }

In the following, we always assume $P$ to be the regular right-angled hyperbolic dodecahedron in $\mathbb{H}^{3}$ with twelve 2-dimensional facets. Using $nP$, specially $1P=P$, to denote the linearly-glued $n$ copies of  $P$ as shown in Figure \ref{figure: 4example}, and $nK$, specially $1K=K$, is the dual of the boundary of $nP$. So for each polytope $nP$, $\ n\geqslant2$, there are $(n+3)$ layers of facets of $nP$:
both the first and the last layer are pentagons; 
both the second and the $(n+2)$-th layers consist of five pentagons;
each layer from the third to the $(n+1)$-th consists of five hexagons. There is no hexagonal layer in $1P$, and the polytope  $nP$ has $(5n+7)$ facets in total. All the polytopes $nP$, $n\in \mathbb{Z}_+$, are right-angled  hyperbolic polytopes. Notations $nP$ and $nK$ make sense in the rest of this paper unless other statements.
\begin{figure}[H]
	\scalebox{0.58}[0.58]{\includegraphics {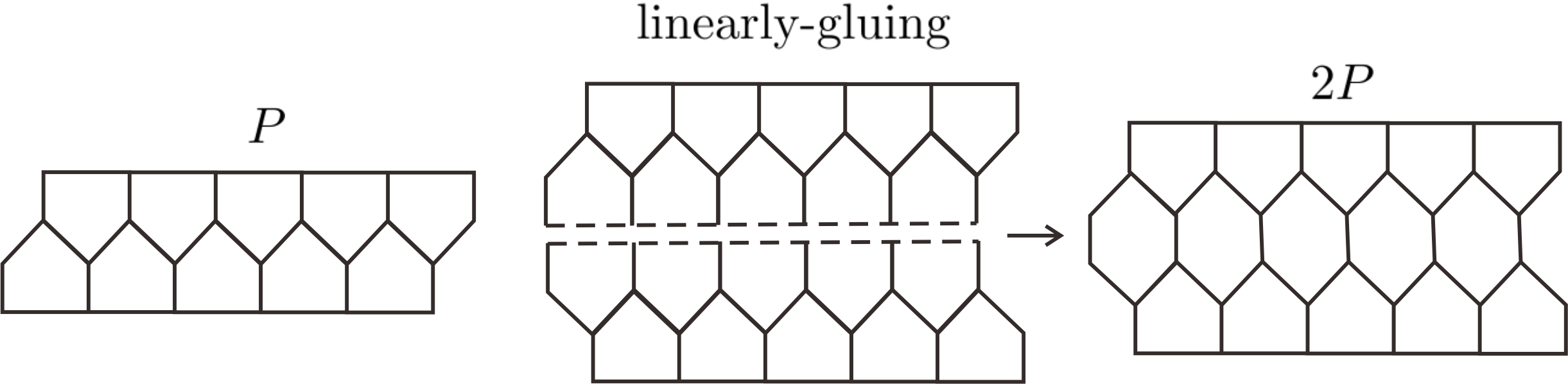}}
	\caption{Building up the polytope $2P$ by linearly-gluing.}\label{figure: 4example}
\end{figure}

\begin{definition}For a simplicial complex $N$, denoting $X(N)=(a_{ij})_{m\times m}$ the \emph{adjacent matrix} of $N$,  where $m=\vert\{{\rm vertices} \ {\rm of}\  N\}\vert$. That means
\begin{center}
$a_{ij}=\left\{
\begin{array} {rcl}1 &\mbox{
 if vertex $i$ and $j$ are connected by an 1-dimensional  simplex in $N$} \\ 0 & \mbox{otherwise}
\end{array}\right.$.
\end{center}
\end{definition}
If $N$ is the dual of a simple polytope $L$, namely $N=(\partial L)^*$,  there is a one-to-one correspondence between facets of $L$ and vertices of $N$. Thus we have $\vert \mathcal{F}(L)\vert=\vert\{{\rm vertices} \ {\rm of}\  N\}\vert=m$. Thus the matrix is also called as the \emph{adjacent matrix} of simple polytope $L$ and can be denoted by either $X(L)$ or $X(N)$. For $nP$, $n\in \mathbb{Z}_+$, is always a flag simple polytope, then all the intersecting information about its facets is included in the adjacent matrix. 

In order to get more disciplined adjacent matrices $X(nP)$ as $n$ increases, we order the facets of the polytope $nP$ by the following manners: the  first and the  last layer are ordered as $1$ and $5n+7$ respectively; the facets between are labeled layer by layer. For an even layer, we start from the middle and then order the rest by doubly siding (left-right). While for an odd layer we adopt a right-left doubly siding. 

We illustrate all these descriptions on $5P$ as shown in Figure \ref{figure: 5example}, where the double sidings of even and odd layers are displayed by the arrow-lines on the second and third layers respectively.
\begin{figure}[H]
	\scalebox{0.8}[0.8]{\includegraphics {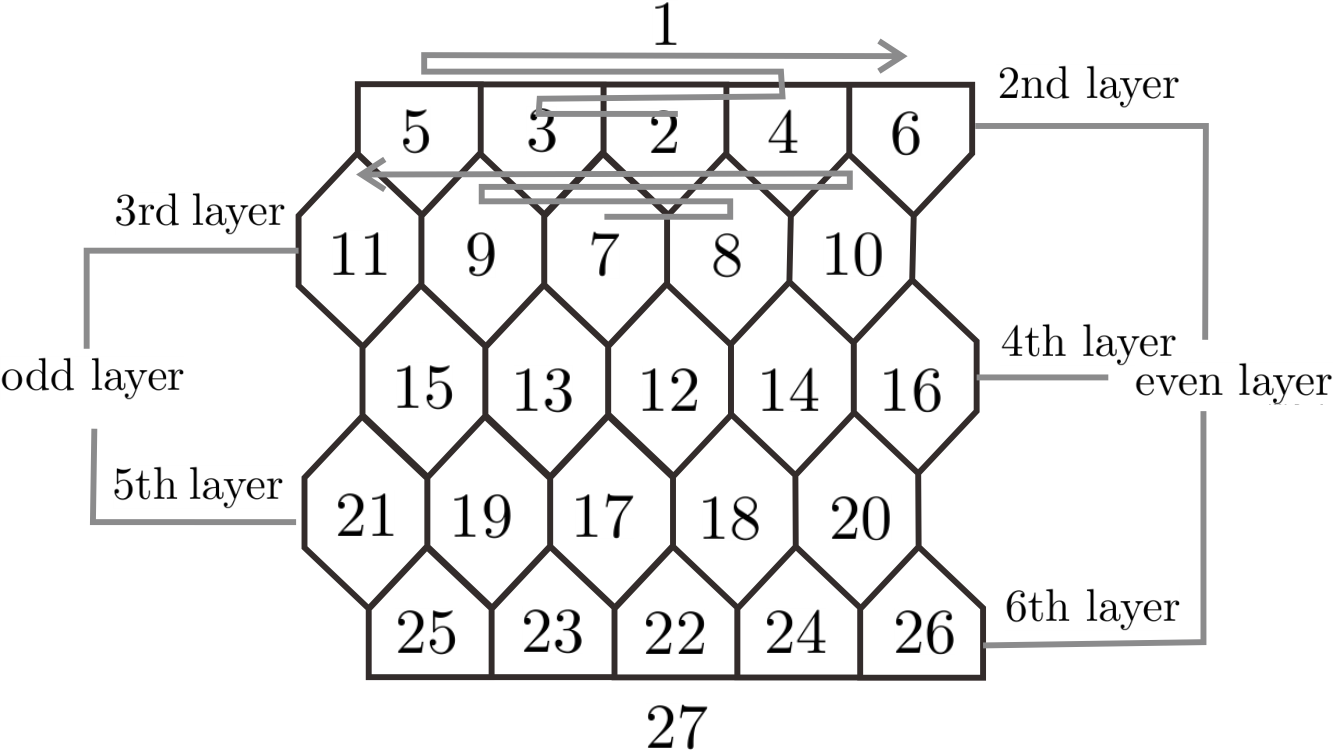}}
	\caption{Facet ordering of the polytope $5P$.}\label{figure: 5example}
\end{figure}

Using this ordering manner,  we achieve more unified increasing patterns of the adjacent matrices and display some as shown in Figure \ref{figure:adjacent matrix} (the omitted entries are all zeros):

\begin{figure}[H]
	\scalebox{0.77}[0.77]{\includegraphics {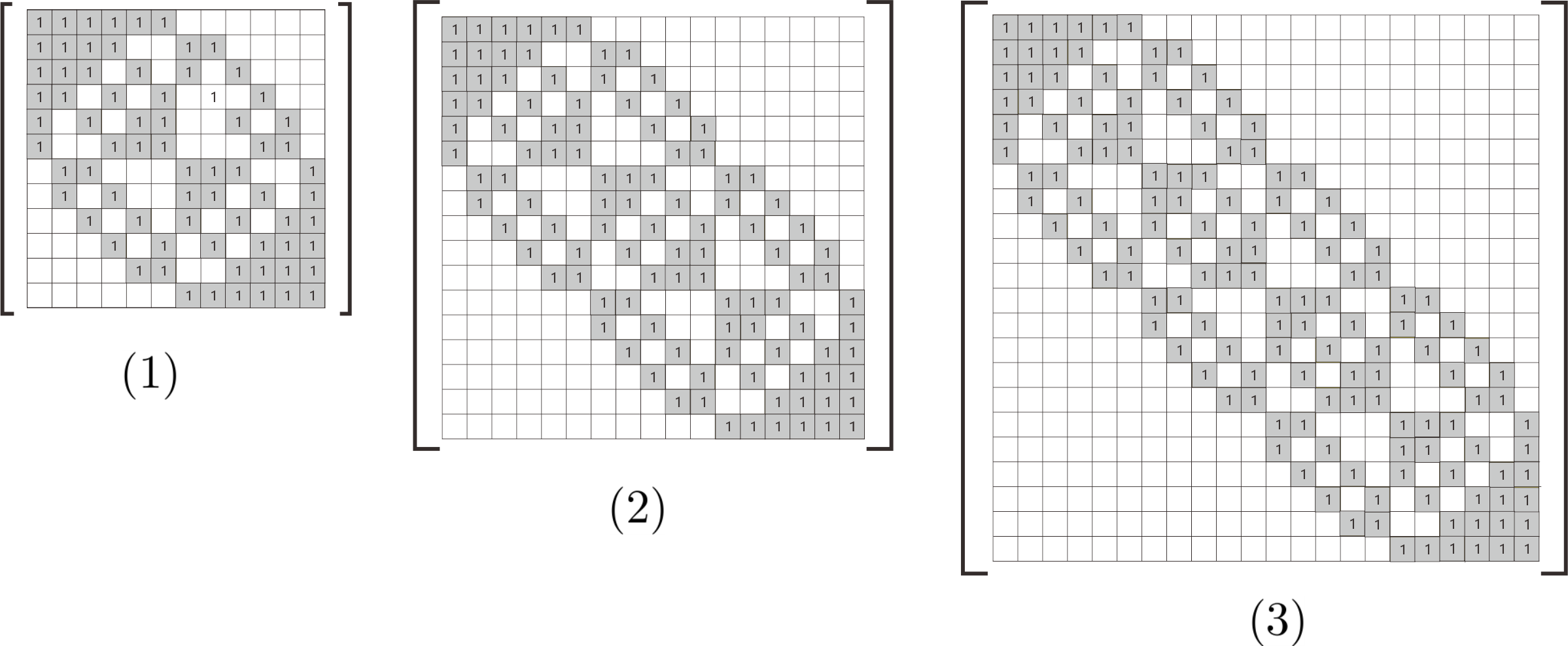}} 
	\caption{Adjacent matrices of the polytopes $P$, $2P$ and $3P$.}\label{figure:adjacent matrix}
	%\ref{figure:adjacent matrix} Adjacent Matrix
\end{figure}

\subsection{classification}
All the real toric manifolds over the same polytope $L$ are $G$-manifolds of $L$ \cite{dj:1991}. Two $G$-manifolds $M_1$ and $M_2$ are  \emph{$\theta$-equivariant homeomorphic} if there exists a homeomorphism  $f: M_1 \rightarrow M_2$ and an automorphism $\theta$ of $G$, such that $f(gx)=\theta (g)f(x)$ for every $g\in G$ and
$x \in M_1$. Those two $G$-manifolds are then said to be \emph{DJ-equivalent over $L$}.
For an $n$-dimensional simple polytope $L$%, let $\mathbb{A}=\mathbb{A}(L)$ be the symmetry group of $L$ and $G=GL_n(\mathbb{Z}_{2})$ be the general linear group of degree $n$ over $\mathbb{Z}_{2}$. Two small covers $M(L,\lambda)$ and $M(L,\mu)$ are DJ-equivalent over $L$ if $\lambda=g \circ \mu \circ a$ for some $a \in \mathbb{A}$ and $g \in GL_n(\mathbb{Z}_{2})$ \cite{Scott:02}. Such description about  DJ-equivalent can be  applied to all the real toric manifolds rather than only small covers.

We fix the colorings  of  three facets of the polytope $nL$, which are adjacent to a fixed  vertex, to be $e_1, e_2$ and  $e_3$, the standard basis of $\mathbb{Z}_2^3$. Therefore the general linear group action $GL_3(\mathbb{Z}_2)$ have been moduled out. All elements of one class of $\{M(L,\lambda)\}/GL_3(\mathbb{Z}_2)$ are said to be $GL_3(\mathbb{Z}_2)$-equivalent, where $\{M(L,\lambda)\}$ is the set of all possible small covers over $L$. For example, assuming $L$ is a square, and we order the facets as shown in (1) of Figure \ref{figure:0}. Then there are totally 18 small covers that recovered from 18 colored $L$. The image of these characteristic functions on $\mathcal{F}(L)$ is arranged in a vector, $(\lambda(F_1),~\lambda(F_2),~\lambda(F_3),~\lambda(F_4)$), as shown in Table \ref{table:0}. 
	\begin{table}[H]
		\begin{tabular}{|c|c|c|}
			\hline
		($e_1$, $e_2$, $e_1$, $e_2$)&($e_1$, $e_2$, $e_1$, $e_1+e_2$)&($e_1+e_2$, $e_1$, $e_2$, $e_1$) \\
		
		($e_2$, $e_1$, $e_2$, $e_1$)&($e_2$, $e_1$ $e_2$, $e_1+e_2$)&($e_1$, $e_2$, $e_1+e_2$, $e_2$)\\
		
		($e_1+e_2$, $e_1$, $e_1+e_2$, $e_1$)&($e_1+e_2$, $e_1$, $e_1+e_2$, $e_2$)&($e_2$, $e_1$, $e_1+e_2$, $e_1$)\\

		($e_1$, $e_1+e_2$, $e_1$, $e_1+e_2$)&($e_1$, $e_1+e_2$, $e_1$, $e_2$)&($e_1+e_2$, $e_2$, $e_1$, $e_2$)\\

		($e_2$, $e_1+e_2$, $e_2$, $e_1+e_2$)&($e_2$, $e_1+e_2$, $e_2$, $e_1$)&($e_2$, $e_1+e_2$, $e_1$, $e_1+e_2$)\\

		($e_1+e_2$, $e_2$, $e_1+e_2$, $e_2$)&($e_1+e_2$, $e_2$, $e_1+e_2$, $e_1$)&($e_1$, $e_1+e_2$, $e_2$, $e_1+e_2$)\\
		\hline
		\end{tabular}
		\caption{ All the small covers over square. }
		\label{table:0}
	\end{table}
	After fixing the colors of the first two adjacent facets to be $e_1$ and $e_2$, we get three $GL_2(\mathbb{Z}_2)$-equivalent classes in total. Three representatives picked from the each of the classes are illustrated below. The characteristic functions in (2), (3) and (4) of Figure \ref{figure:0} are denoted by $\lambda_1$, $\lambda_2$ and $\lambda_3$ respectively.
	
\begin{figure}[H]
	\scalebox{0.5}[0.5]{\includegraphics {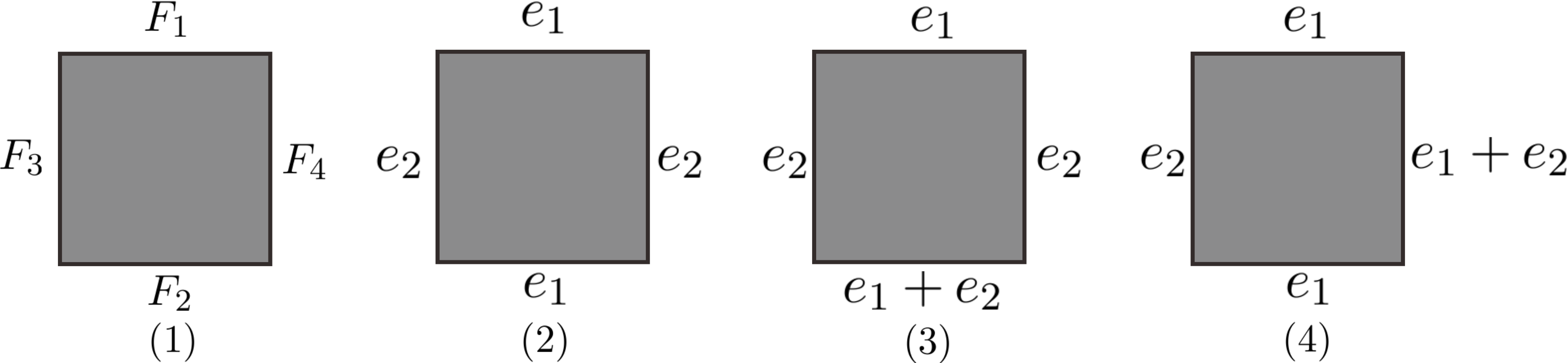}}
	\caption{Three DJ-equivalent representatives.}\label{figure:0}
\end{figure}
	
The inclusion relations of different classifications upon toric manifolds that corresponds  a $\mathbb{Z}_2^{m-r}$-coloring and a certain polytope $L$, where $m=\vert \mathcal{F}(L)\vert$  and $r$ is the rank of the subgroup that being quotiented freely, is depicted in Figure \ref{figure:00}.
\begin{figure}[H]
	\scalebox{0.5}[0.5]{\includegraphics {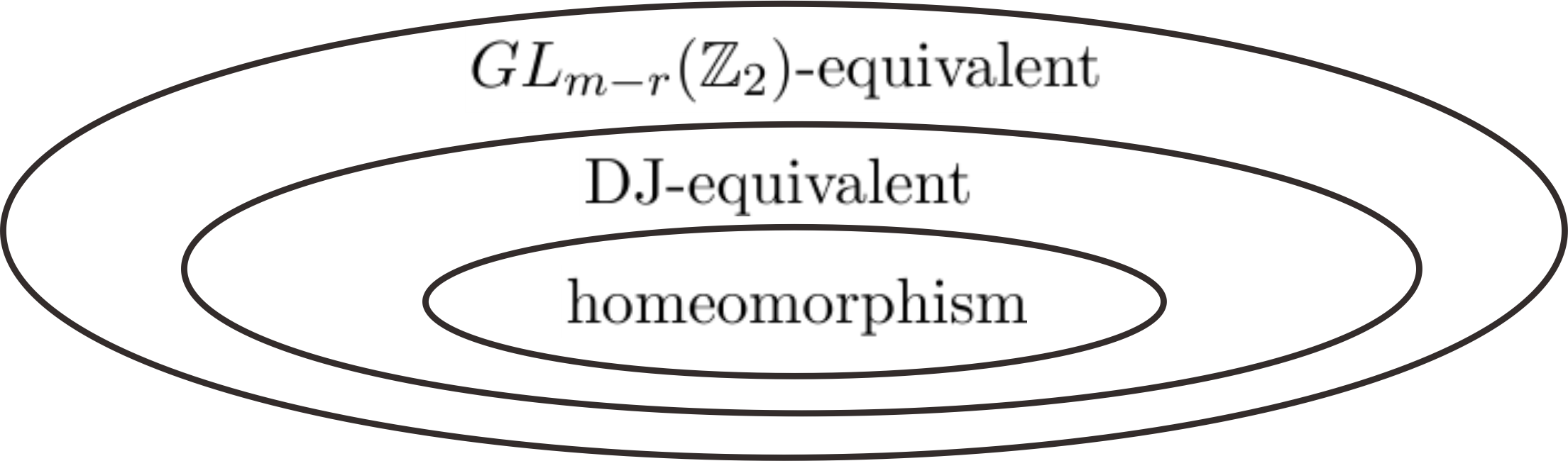}}
	\caption{Different classifying relations.}\label{figure:00}
\end{figure}

The DJ-equivalent set might sometimes coincide with the homeomorphism set. For example homeomorphic class and DJ-equivalent class of small covers over the hyperbolic right-angled dodecahedron $P$ mean the same, see \cite{Scott:02}. %We now firstly obtain the $GL_3(\mathbb{Z}_2)$ classes and then calculated the Betti numbers of those manifolds. By recursive algorithm programming, we have calculated that the numbers of the $GL_3(\mathbb{Z}_2)$-equivalent classes over polytopes  $1P$, $2P$ and $3P$ are 2165, 392170 and 57507285 respectively.

\subsection{Useful forms }
To keep the data concise, we encode every $\mathbb{Z}_2^*$-coloring vector by an integer through binary. For example in the $\mathbb{Z}_2^3$-coloring case, we can use $1$, $2$, $3$,  $4$, $5$, $6$ and  $7$ to represent the seven colorings $(1,0,0)$, $(0,1,0)$,  $(1,1,0)$, $(0,0,1)$ $(1,0,1)$, $(0,1,1)$ and  $(1,1,1)$ respectively. Then a characteristic matrix can also be viewed as a characteristic vector. For example, the characteristic matrix of the $\mathbb{Z}_2^3$-coloring characteristic function in  Example \ref{example:2} is:
\begin{center}
$(\lambda(\mathcal{F}_1),\lambda(\mathcal{F}_2),\lambda(\mathcal{F}_3))=(\lambda(a,b),\lambda(b,c),\lambda(a,c))=
\begin{pmatrix}
1&0 & 0\\
0&1 & 0\\
0&0 & 1\\
\end{pmatrix}$.
\end{center}

\noindent Then the  corresponding characteristic vector $C$ is $(1, 2, 4)$, and the row space  $row \Lambda=\langle (1, 0, 0), (0, 1, 0), (0, 0, 1)\rangle$. By the way, the characteristic function $\lambda$, characteristic matrix $\Lambda$ and the characteristic vector $C$ can be inferred from each other easily. Here the characteristic vector $C$ is the most  concise representing form among them.

%---------------------------------------------------------------------------��������
\section{\textbf{A key lemma}}

The purpose of this section is to prove Lemma \ref{lemma: key}, which is the key in the proof of Theorem  \ref{theorem: betti}.

\begin{definition}
Starting from a $\mathbb{Z}_2^3$-coloring $\lambda$ on the polytope $nP$, we can extend it to  $(2^m-1)$ many $\mathbb{Z}_2^4$-colorings on $nP$ by adding a non-zero fourth row to the $3 \times m$ characteristic matrix $\Lambda$ of $\lambda$ as shown below
\begin{center}
	$
	\begin{pmatrix}
	1&0 & 0 &...&*\\
	0&1 & 0 &...&*\\
	0&0 & 1 &...&*\\
	*&* & * &...&*\\
	\end{pmatrix},
	$
\end{center}
where $m=5n+7$ and $*\in \{0,1\}$. Those characteristic functions are called the \emph{extensions} of $\lambda$ and they naturally satisfy the non-singularity condition,
\end{definition}

%\begin{definition} a $(\mathbb{Z}_2)^s$-coloring $\lambda$ on $nP$ is unsinkable if  $rank(\Lambda) =s$.
%\end{definition}

%We recall the definition of  admissible coloring in Section 1.

\begin{definition} A $\mathbb{Z}_2^3$-coloring $\lambda$ on the polytope  $nP$ is \emph{admissible} if there is a $\mathbb{Z}_2^4$-coloring  extension of $\lambda$, denoted by $\delta$, satisfying $M(nP,\lambda)$ is non-orientable and  $M(nP,\delta)$ is orientable.
\end{definition}

%---------------------------------------------------------�е���������֤�� ƽ�а汾 ����
H. Nakayama and Y. Nishimura discussed the orientability of a  small cover in \cite{NakayamaN:05} and below is the main theorem.

\begin{theorem} \label{theorem: NakayamaN} \emph{(H. Nakayama-Y. Nishimura \cite{NakayamaN:05})} For a simple $n$-dimensional polytope $L$, and for a basis $\{e_1,...,e_n\}$ of $\mathbb{Z}_2^n$, a homomorphism $\epsilon:\mathbb{Z}_2^n\rightarrow\mathbb{Z}_2=\{0,1\}$ is defined by $\epsilon(e_i)=1$, for each $i=1, ... , n$. A small cover $M(L,\lambda)$ is orientable if and only if there exists a basis $\{e_1,...,e_n\}$ of $\mathbb{Z}_2^n$ such that the image of $\epsilon\lambda$ is $\{1\}$.
\end{theorem}

In particular, a small cover $M(L, \lambda)$ over a 3-dimensional simple polytope $L$ is orientable if and only if it is colored by $e_1$,  $e_2$, $e_3$ and $e_1 +e_2+e_3$ up to $GL_{3}(\mathbb{Z}_{2})$-action.

Theorem \ref{theorem: NakayamaN} can actually be adjusted to meet all real toric manifolds instead of only small covers by merely parallel generalization. We can also use Theorem \ref{theorem: ChoiP} with rational coefficient  to re-check this claim. Because the $n$-th Betti number of a real toric manifold $M(L,\delta)$ is 1 if and only if there is an element in the row space of $\Delta$ with all the entries are 1. Here $\Delta$ is the characteristic matrix of $\delta$, which is a $\mathbb{Z}_2^4$-coloring extension of $\lambda$. When $M(L,\lambda)$ is non-orientable, then the only possible row of $row\Delta$ with all entries are $1$ is the row contributed by summing up all four rows of $\Delta$. That is to say the sum of every column of the characteristic matrix $\Delta$ is certainly $1 ~~mod~~ 2$.

So by Theorem \ref{theorem: NakayamaN} and paragraph above, we can have:
\begin{remark}\label{remark:2}
A $\mathbb{Z}_2^3$-coloring $\lambda$ over a simple 3-dimensional polytopy $L$ is admissible if and only if $M(L,\lambda)$ is non-orientable.
\end{remark}

Furthermore, we can obtain the following proposition based on previous discussions and some facts about fundamental group of a double cover:

\begin{proposition}\label{remark:1}
For every non-orientable $\mathbb{Z}_2^3$-coloring $\lambda$ over the $3$-dimensional polytope $nP$, there is a unique $\mathbb{Z}_2^4$-coloring extension $\delta$ such that the $3$-manifold $M(nP,\delta)$ is the orientable double cover of the non-orientable $3$-manifold $M(nP,\lambda)$.
\end{proposition}
\begin{proof}
By Theorem \ref{theorem: NakayamaN}, every non-orientable $\mathbb{Z}_2^3$-coloring ${\lambda}$ over the $3$-dimensional simple polytope $nP$ has a unique  $\mathbb{Z}_2^4$-coloring extension $\delta$ such that  $M(nP,\delta)$ is orientable. The characteristic matrix of $\delta$ is the only one, among all $\mathbb{Z}_2^4$-extensions of $\lambda$, satisfying that the sum of every column is $1 ~~mod~~ 2$.

Denoted by $W(nP)$ the Coxeter group of $nP$, then there is a short exact sequence:
$$1\longrightarrow W(nP)\xrightarrow{~l~} \mathbb{Z}_2^m\xrightarrow{~q_1~} \mathbb{Z}_2^4\xrightarrow{~q_2~} \mathbb{Z}_2^3 \longrightarrow 1.$$
We have $\pi_1(M(nP,\delta))=ker\delta$, $\pi_1(M(nP,\lambda))=ker\lambda$, where $l$ is the abelization and  $q_1$, $q_2$ are the natural  quotient maps, see \cite{dj:1991}. Denoting $\lambda^0:\mathcal F(L)=\{F_1,F_2,\dots,F_m\}\longrightarrow \mathbb{Z}_2^m$ a map that maps each $F_i$ to $e_i$. Then $\lambda=q_1\circ\lambda^0$, $\delta=q_2\circ q_1\circ\lambda^0$.

Thus $ker\lambda < ker\delta$ and $M(nP,\delta)$ is the orientable double cover of $M(nP,\lambda)$.
\end{proof}
%---------------------------------------------------------������������

The $\mathbb{Z}_2^4$-coloring $\delta$ on the polytope $nP$ in Proposition \ref{remark:1} is called an \emph{admissible extension} of $\lambda$ or a \emph{natural $\mathbb{Z}_2^4$-coloring associated to $\lambda$} (say a \emph{natural $\mathbb{Z}_2^4$-extension} of $\lambda$ for short).

From this proposition, on one hand, the existence and the uniqueness of the admissible extension make sense. On the other hand the orientability of a real toric manifold can be detected easily by the characteristic matrix. In particular, $(1,1,0)$, $(1,0,1)$ and $(0,1,1)$ in $\mathbb{Z}_2^3$, which are the binary form of 3, 5 and 6, are the only three elements whose item sum are $0$ mod 2. So a characteristic vector corresponding to a non-orientable $\mathbb{Z}_2^3$-coloring, should contain at least one of  3, 5 and  6. %Then by programming computation we select 2155, 392130 and 57507085 out of the 2165, 392170 and 57507285 $GL_3(\mathbb{Z}_2)$-equivalent classes over the polytopes  $P$, $2P$ and $3P$ as discussed in Section 2.5 that are  non-orientable manifolds. Those characteristic functions naturally have admissible extensions. 
%
%\begin{example} \label{example:7}
%	Let $\lambda$ and $\delta$ be the $\mathbb{Z}_2^3$-coloring and $\mathbb{Z}_2^4$-coloring over a 2-simplex respectively as shown in (1) and (2) of Figure \ref{figure:unique_extending}. Namely:
%	\begin{center}
% 		$\lambda:\{\{a,b\},\{b,c\},\{a,c\}\}\rightarrow\{(1,0),(0,1),(1,1)\}, $
%		
%		$(a,b)\mapsto (1,0),(b,c)\mapsto (0,1),(a,c)\mapsto (1,1); 
%		$
%			
%		$\delta:\{\{a,b\},\{b,c\},\{a,c\}\}\rightarrow\{(1,0,0),(0,1,0),(0,0,1)\},$
%		
%		$(a,b)\mapsto (1,0,0),(b,c)\mapsto (0,1,0),(a,c)\mapsto (0,0,1).$
%	\end{center}
%	
%
%\begin{figure}[H]
%	\scalebox{0.58}[0.58]{\includegraphics {071.png}} 
% \caption{Existence and  uniqueness of admissible extension.}
 %\label{figure:unique_extending}
%\end{figure}
%	
%Moreover $M(L,\lambda)\cong \mathbb{RP}^2$ is non-orientable and $M(L,\delta)\cong \mathbb{S}^2$ is orientable. Here by Proposition \ref{remark:1}, $\delta$ is the unique admissible extension of $\lambda$ whose characteristic matrix satisfies that the sum of every column is 1 mod 2:
%\begin{center}
%	$
%\begin{pmatrix}
%	1&0 &1\\
%	0&1 & 1\\
%\end{pmatrix}_{\lambda}\xrightarrow{{\rm the~unique~ orientable}~\mathbb{Z}_2^3{\rm -extending}}
%\begin{pmatrix}
%	1&0 &1\\
%	0&1 & 1\\
%	0&0&1 \\
%	\end{pmatrix}_{\delta}
%	$
%\end{center}
%\end{example}
Moreover, by Theorem \ref{theorem: ChoiP}, the Betti number of the orientable manifold recovered by the admissible extension, called as the  natural $\mathbb{Z}_2^4$-extension $\delta$, can be computed clearly. And we show this routine in Example \ref{example:6} as a demonstration.

\begin{example} \label{example:6} Calculation of Betti numbers of some $M(P,\delta)$:

Figure \ref{figure:colored P} (1) is the dodecahedron $P$ whose 12 facets has been ordered, Figure \ref{figure:colored P} (2) is the simplicial complex $K=(\partial P)^*$ with its 12 vertices being ordered correspondingly.

\begin{figure}[H]
	\scalebox{0.9}[0.9]{\includegraphics {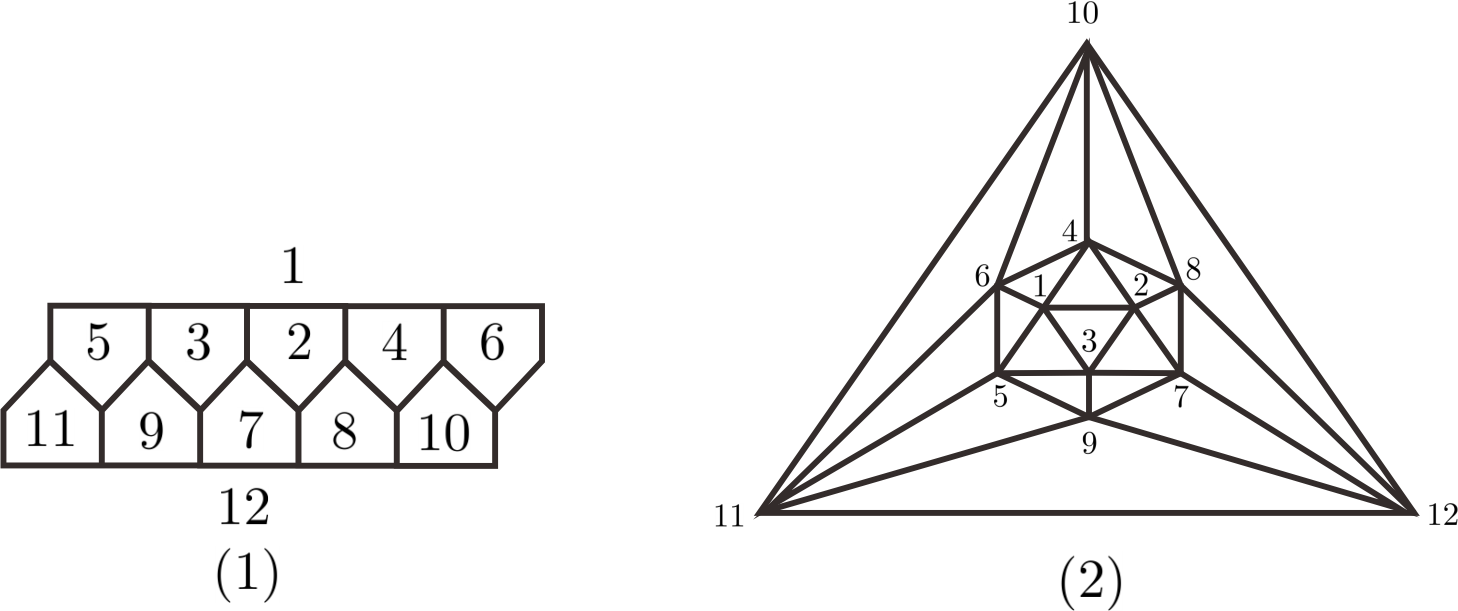}}
	\caption{Ordered facets of the polytope $P$ and the corresponding ordered vertices of $K=(\partial P)^*$.}
	\label{figure:colored P}
\end{figure}

Coloring the polytope $P$ with the characteristic vector  $v=(1,~ 2,~ 4,~ 5,~ 3,~ 7,~ 7,~ 3,~ 5,~ 4,~ 2,~ 1)$. We then have a $\mathbb{Z}_2^3$-coloring characteristic matrix:

	\begin{center}
		$\Lambda=
		\begin{pmatrix}
		\begin{smallmatrix}
		0&0 & 1&1&0&1&1&0&1&1&0&0\\
		0&1 & 0&0&1&1&1&1&0&0&1&0\\
		1&0 & 0&1&1&1&1&1&1&0&0&1\\
		\end{smallmatrix}
		\end{pmatrix}_{3\times 12}.$
	\end{center}
	
And by Proposition \ref{remark:1}, the $\mathbb{Z}_2^4$-coloring matrix of its admissible extension is
	\begin{center}
$\Delta=
		\begin{pmatrix}\begin{smallmatrix}
		&&&&&&&&&&&\\
		&&&&&&\Lambda&&&&&\\
		&&&&&&&&&&&\\
		0&0&0&1&1&0&0&1&1&0&0&0\\
		\end{smallmatrix}\end{pmatrix}_{4\times 12}.
		$
	\end{center}

Now the row space of $\Delta$ is $Row\Delta=\langle(0, 0, 1, 1, 0, 1, 1, 0, 1, 1, 0, 0), (0, 1, 0, 0, 1, 1, 1, 1, 0, 0, 1, 0), $

\noindent$(1, 0, 0, 1, 1, 1, 1, 1, 1, 0, 0, 1), (0, 0, 0, 1, 1, 0, 0, 1, 1, 0, 0, 0)\rangle$.
	
 For all $\omega_i\in Row\Lambda$, we calculated its contribution to the Betti numbers in Table \ref{table:1}.
	
	%\begin{table}[H]
		\begin{longtable}{|c|c|c|c|}
			\caption{ $\beta^1(M(P,\delta))$. }
			\label{table:1}\\
			\hline
			$\omega_1$=(0, 0, 1, 1, 0, 1, 1, 0, 1, 1, 0, 0)&  $K_{\omega_1}$&\scalebox{0.085}[0.085]{\includegraphics {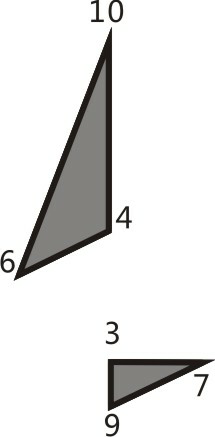}}  &  $\widetilde\beta^{0}(K_{\omega_1})=1$\\
			\hline
			$\omega_2$=(0, 1, 0, 0, 1, 1, 1, 1, 0, 0, 1, 0)& $K_{\omega_2}$&\scalebox{0.085}[0.085]{\includegraphics {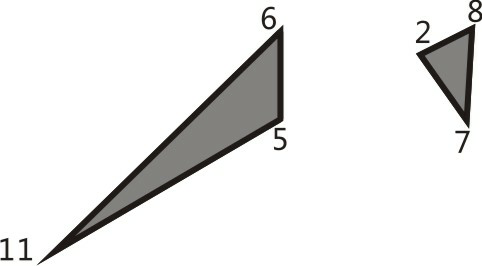}}  &$\widetilde\beta^{0}(K_{\omega_2})=1$\\
			\hline
			$\omega_3$=(1, 0, 0, 1, 1, 1, 1, 1, 1, 0, 0, 1)& $K_{\omega_3}$&\scalebox{0.085}[0.085]{\includegraphics {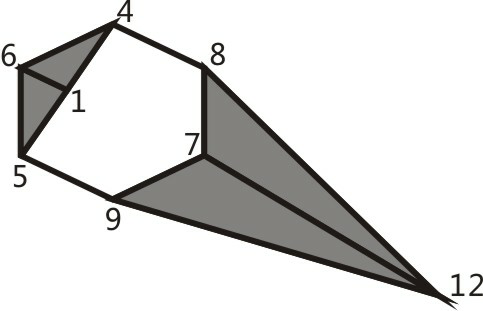}} & $\widetilde\beta^{0}(K_{\omega_3})=0$, $\beta^{1}(K_{\omega_3})=1$\\
			\hline
			$\omega_4$=(0, 0, 0, 1, 1, 0, 0, 1, 1, 0, 0, 0)& $K_{\omega_4}$&\scalebox{0.085}[0.085]{\includegraphics {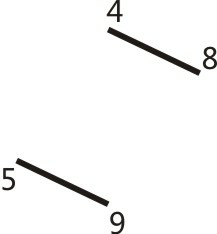}} &$\widetilde\beta^{0}(K_{\omega_4})=1$\\
			\hline
			$\omega_5$=(0, 1, 1, 1, 1, 0, 0, 1, 1, 1, 1, 0)& $K_{\omega_5}$&\scalebox{0.085}[0.085]{\includegraphics {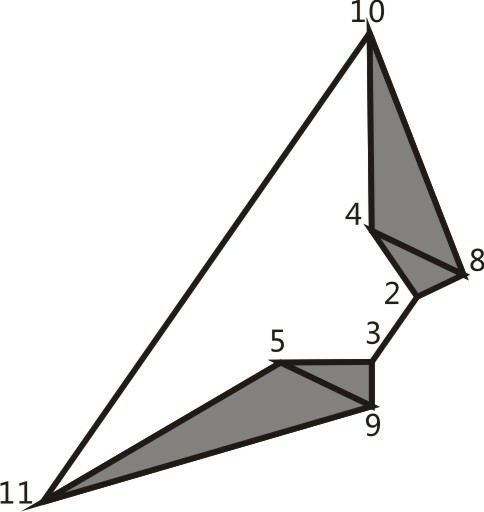}} &$\widetilde\beta^{0}(K_{\omega_5})=0$, $\beta^{1}(K_{\omega_5})=1$\\
			\hline
			$\omega_6$=(1, 0, 1, 0, 1, 0, 0, 1, 0, 1, 0, 1)& $K_{\omega_6}$&\scalebox{0.085}[0.085]{\includegraphics {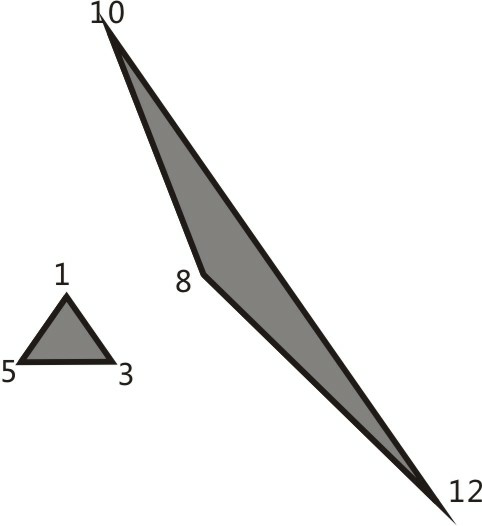}} & $\widetilde\beta^{0}(K_{\omega_6})=1$\\
			\hline
			$\omega_7$=(0, 0, 1, 0, 1, 1, 1, 1, 0, 1, 0, 0)& $K_{\omega_7}$&\scalebox{0.085}[0.085]{\includegraphics {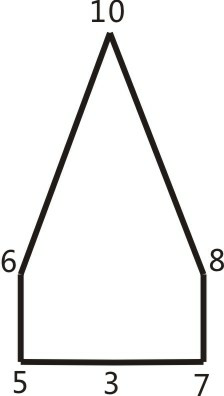}} &$\widetilde\beta^{0}(K_{\omega_7})=0$, $\beta^{1}(K_{\omega_7})=1$\\
			\hline
			$\omega_8$=(1, 1, 0, 1, 0, 0, 0, 0, 1, 0, 1, 1)& $K_{\omega_8}$&\scalebox{0.085}[0.085]{\includegraphics {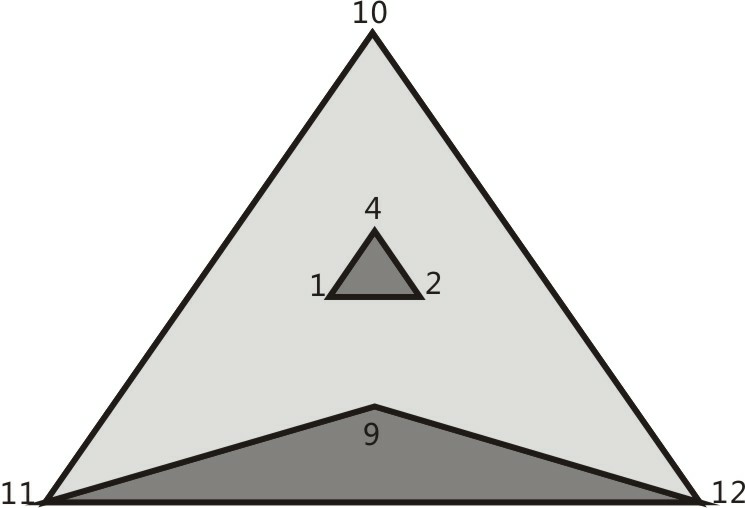}} &$\widetilde\beta^{0}(K_{\omega_8})=1$\\
			\hline
			$\omega_9$=(0, 1, 0, 1, 0, 1, 1, 0, 1, 0, 1, 0)& $K_{\omega_9}$&\scalebox{0.085}[0.085]{\includegraphics {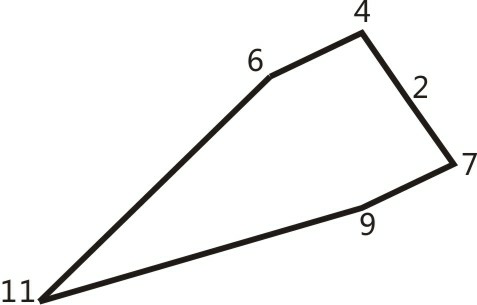}} &$\widetilde\beta^{0}(K_{\omega_9})=0$, $\beta^{1}(K_{\omega_9})=1$\\
			\hline
			$\omega_{10}$=(1, 0, 0, 0, 0, 1, 1, 0, 0, 0, 0, 1)& $K_{\omega_{10}}$&\scalebox{0.085}[0.085]{\includegraphics {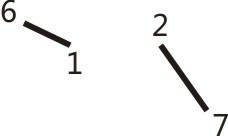}} &$\widetilde\beta^{0}(K_{\omega_{10}})=1$\\
			\hline
			$\omega_{11}$=(1, 1, 1, 0, 0, 1, 1, 0, 0, 1, 1, 1)& $K_{\omega_{11}}$&\scalebox{0.085}[0.085]{\includegraphics {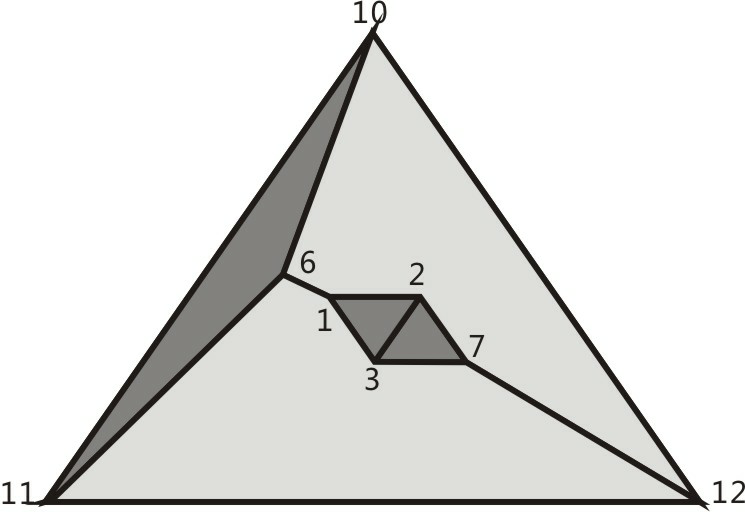}}& $\widetilde\beta^{0}(K_{\omega_{11}})=0$, $\beta^{1}(K_{\omega_{11}})=1$\\
			\hline
			$\omega_{12}$=(1, 1, 0, 0, 1, 0, 0, 1, 0, 0, 1, 1)& $K_{\omega_{12}}$&\scalebox{0.085}[0.085]{\includegraphics {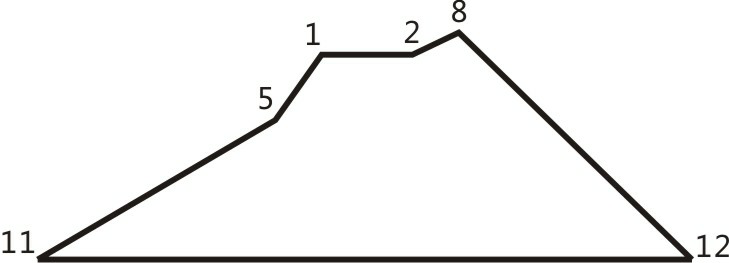}} &$\widetilde\beta^{0}(K_{\omega_{12}})=0$, $\beta^{1}(K_{\omega_{12}})=1$\\
			\hline
			$\omega_{13}$=(1, 0, 1, 1, 0, 0, 0, 0, 1, 1, 0, 1)& $K_{\omega_{13}}$&\scalebox{0.085}[0.085]{\includegraphics {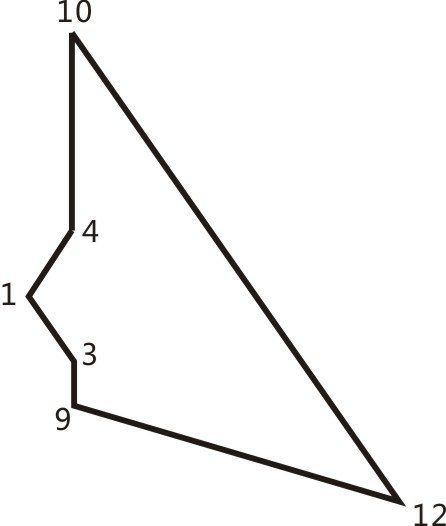}} &$\widetilde\beta^{0}(K_{\omega_{13}})=0$, $\beta^{1}(K_{\omega_{13}})=1$\\
			\hline
			$\omega_{14}$=(0, 1, 1, 0, 0, 0, 0, 0, 0, 1, 1, 0)& $K_{\omega_{14}}$&\scalebox{0.085}[0.085]{\includegraphics {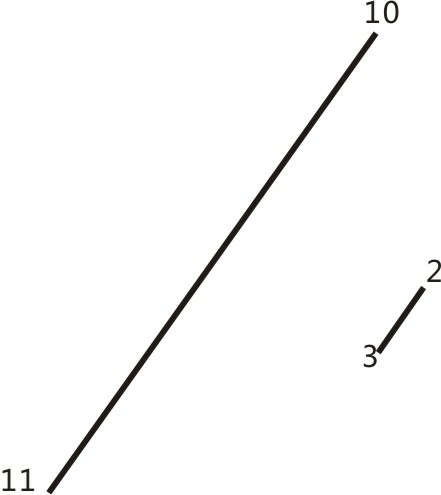}} &$\widetilde\beta^{0}(K_{\omega_{14}})=1$\\
			\hline
			$\omega_{15}$=(0, 0, 0, 0, 0, 0, 0, 0, 0, 0, 0, 0)&$K_{\omega_{15}}$ &$\emptyset$&no contribution to $\beta^1(M(P,\delta))$\\
			\hline
			$\omega_{16}$=(1, 1, 1, 1, 1, 1, 1, 1, 1, 1, 1, 1)& $K_{\omega_{16}}$&$\cong \mathbb{S}^2$ & $\widetilde\beta^{0}(K_{\omega_{16}})=\beta^{1}(K_{\omega_{16}})=0$, $\beta^{2}(K_{\omega_{16}})=1$ \\
			\hline
		\end{longtable}
		
	%\end{table}
	
From Table \ref{table:1} and Theorem \ref{theorem: ChoiP}, we have
\begin{center}
$\beta^1(M(P,\delta);Q)=\sum\limits_{i=1}^{16}{\widetilde\beta^0}(K_{\omega_i};\mathbb{Q})=\beta^2(M(P,\delta);\mathbb{Q})=\sum\limits_{i=1}^{16}{\widetilde\beta^1}(K_{\omega_i};\mathbb{Q})=7$.
\end{center}
\end{example}

For an orientable 3-manifold $M(nP,\delta)$, by the Poincar\'{e} Duality, we have $\beta^0(M(nP,\delta))=\beta^3(M(nP,\delta))=1$ and $\beta^1(M(nP,\delta))=\beta^2(M(nP,\delta))$. So $\beta^1$ is the only thing we need to calculate when detecting the free  part of $H^*(M(nP,\delta))$. By Theorem \ref{theorem: ChoiP}, $\beta^1(M(nP,\delta))$ equals to the sum of 16 reduced zero-th Betti numbers of 2-dimensional  subcomplexes $k_{\omega_i}$ of the simplicial complex $nK=(\partial(nP))^*$.

 Therefore the first Betti number $\beta^{1}(M(nP,\lambda))$ can be figured out by counting and summing up the numbers of connected components of all those subcomplexes. Here each subcomplex corresponds to a non-zero vector in the row space $row\Delta$ as shown in Example \ref{example:6}.

%%%Again from duality and orientability of these subcomplexes, we have $\beta^0(K_{\omega_i})$=$\beta_2(K_{\omega_i})$=$\beta_0(K_{\omega_i})$.

We perform a great number of calculations upon such process through the computer. By using \textbf{CT-tree} we are able to seek out a certain coloring family that change regularly and therefore construct a family of geometrically bounding 3-manifolds $M(nP,\delta)$. According to the one-to-one corresponding discussed in Section 2.3, constructing such a specific manifold actually means to work out the required characteristic function $\delta$ over the polytope $nP$.

%%%And yielding to a certain volume $vol(nP)$,

\begin{definition} A \emph{downward pentagon coloring brick} ($d$-brick for short) is a $\mathbb{Z}_2^3$-colored figure with five pentagons glued as in Figure \ref{figure:2} (1). An \emph{upward pentagon coloring brick} ($u$-brick for short) is a $\mathbb{Z}_2^3$-colored figure with five pentagons glued as in Figure \ref{figure:2} (2). A \emph{hexagon  coloring brick} ($h$-brick for short) is a $\mathbb{Z}_2^3$-colored figure with five hexagons glued as in Figure \ref{figure:2} (3).  All the bricks can also be represented by an 5-length integer vector through binary transformation. We adopt a uniform format $S=(a_1~a_2~a_3~a_4~a_5)$ to denote a coloring brick. Here $a_1$, $a_2$, $a_3$, $a_4$ and $a_5$ are the five colors in $\mathbb{Z}_2^3$. A colored $nP$ consists of one $d$-brick, ($n-1$) many $h$-bricks and one $u$-brick successively.
\end{definition}
\begin{figure}[H]
		\scalebox{0.6}[0.6]{\includegraphics {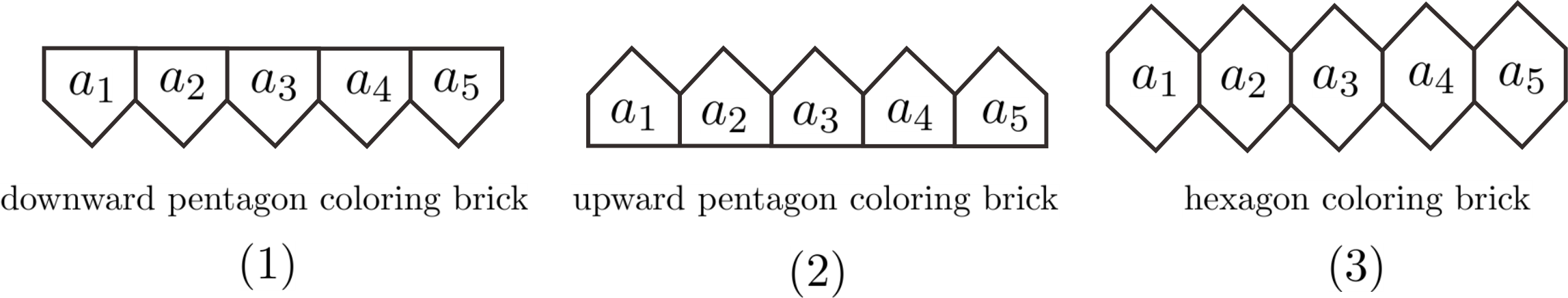}}
	 \caption{ Bricks.}\label{figure:2}
\end{figure}

In particular, in the case of a $\mathbb{Z}_2^3$-colored dodecahedron $P$,  there are only one downward pentagon coloring brick, one downward pentagon coloring brick and no hexagon coloring bricks. Moreover, we have shown before that there are 2155 $GL_3(\mathbb{Z}_2)$-equivalent classes of $\mathbb{Z}_2^3$-coloring on $P$ that can recovered as non-orientable manifolds. Then by adding the information of 120 permutation vectors that encode all the 120 elements information of dodecahedron $P$'s symmetry group $\mathbb{A}$, we further obtain DJ-equivalent classes through computer programming. We obtain exactly 24 non-orientable DJ-equivalent classes and one orientable DJ-equivalent class in  all the small covers over the  dodecahedron  $P$.   Such result also coincides with Corollary 3.4 of Garrison-Scott's paper \cite{Scott:02}.

\begin{definition}
A \emph{$\mathbb{Z}_2^3$-coloring} of the simple polytope $nP$ can be recorded by a coloring vector valued by facets' colors layer-by-layer and left-to-right. Here the writing format of a coloring vector is separating the first and last item and making items of every coloring brick written together.
\end{definition}

For example, $$c=[1~ 24246 ~37171 ~24246 ~1]$$ is a coloring vector of Figure \ref{figure:colororder} (1), whose characteristic vector $C$ is $$(1,~ 2,~ 4,~ 4,~ 2,~ 6,~ 1,~ 7,~ 7,~ 1,~ 3,~ 2,~ 4,~ 4,~ 2,~ 6,~ 1).$$
  Figure \ref{figure:colororder} (2) is the order required in Section 2.5. We can get the characteristic vector $C$ from a coloring vector simply by adjusting elements' positions according to the assigned order. For example, the 5-th facet is colored by (0, 1, 0), which is the binary form of 2, so the 5-th item of $C$ is 2.

\begin{figure}[H]
		\scalebox{0.7}[0.7]{\includegraphics {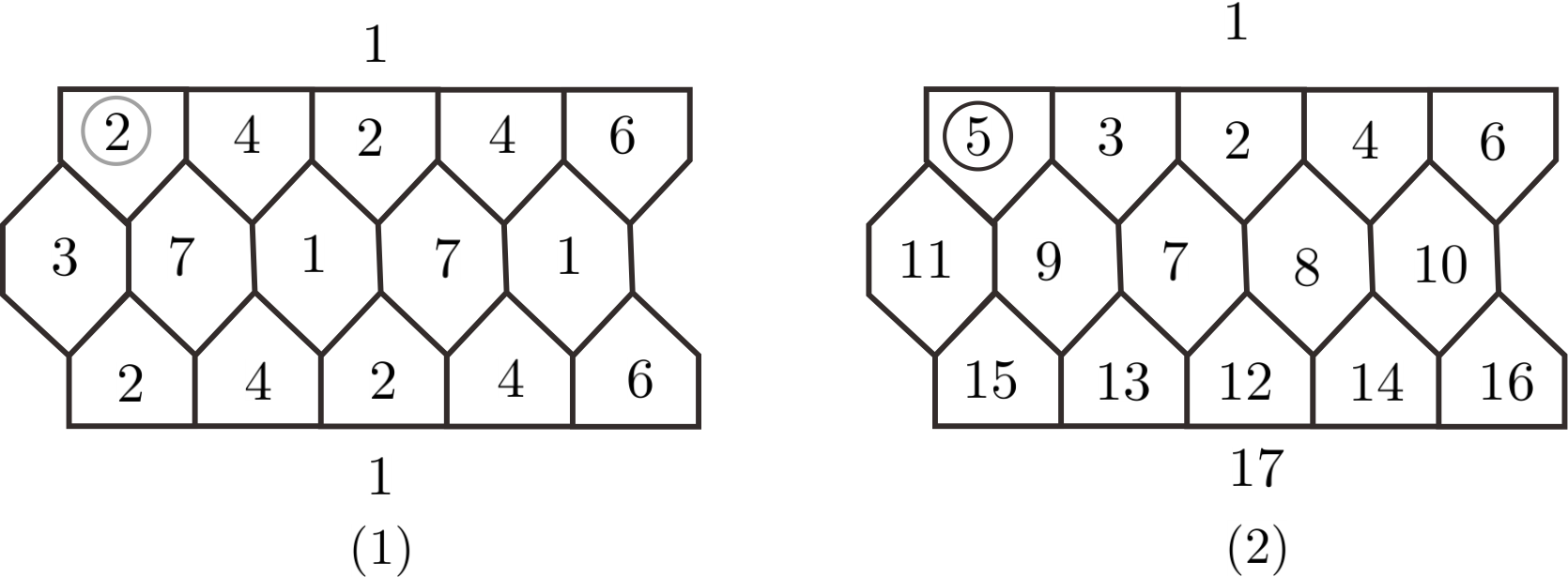}}
		\caption{The left is a coloring over the polytope $2P$ and the right is the facet order of $2P$.}
		\label{figure:colororder}
\end{figure}

\begin{definition}
Two $\mathbb{Z}_2^3$-coloring bricks are said to be \emph{compatible} if the \emph{non-singularity condition} is satisfied at all ten intersecting vertices. They form an \emph{compatible pair}. See Figure  \ref{figure:compatible} (1).
\end{definition}

\begin{definition}
A group of $\mathbb{Z}_2^3$-coloring bricks are \emph{pairwisely compatible} if every two of them are compatible after certain rotation. 
Rotation here means to rotate the colors of one brick for $\frac{2i\pi}{5}$ left, where $0\leq i\leq 4$. For example, $(44262)$ is obtain by rotating  $(24426)$  for $\frac{2\pi}{5}$ left.
\end{definition}

\begin{definition}
	The \emph{affix set} of a downward/upward pentagon $\mathbb{Z}_2^3$-coloring brick $S$, denoted by $\mathcal{A}(S)$, is the set of colorings in $\mathbb{Z}_2^3$ that are compatible with $S$. Namely the non-singularity condition holds at all five vertices as shown in (2) and (3) of Figure \ref{figure:compatible}.
\end{definition}

\begin{figure}[H]
	\scalebox{0.6}[0.6]{\includegraphics {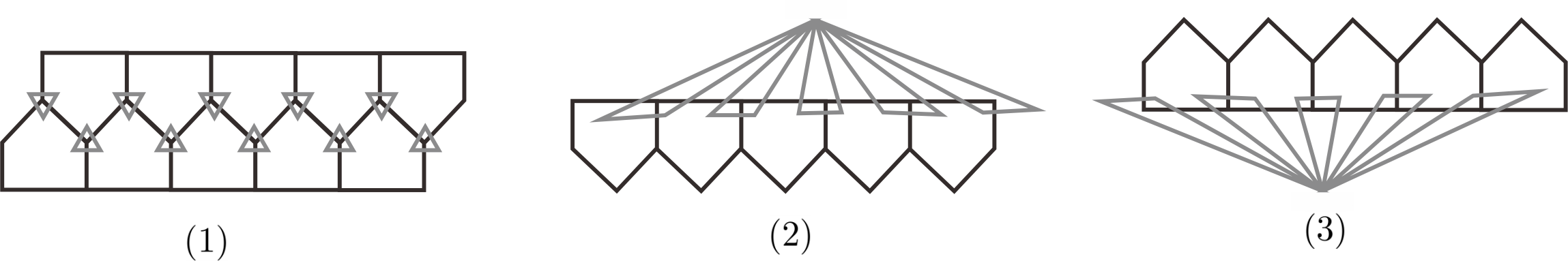}}
	\caption{Compatible pair.}
	\label{figure:compatible}
\end{figure}

%-------------------------------------------------------------------��һ���������������޹�������Ա�����������ޣ�

Now we can prove the key lemma of this section:
\begin{lemma}  \label{lemma: key}
Let $n\in \mathbb{Z}_{+}$ be an  even number, there is a non-orientable $\mathbb{Z}_2^3$-coloring $\lambda$ over the polytope $nP$, such that for its natural associated  $\mathbb{Z}_2^4$-coloring extension $\delta$, we have $\beta^1 (M(nP, \delta))= n+1$.
\end{lemma}

\begin{proof}
	
We start with the case $n=2$. The polytope $2P$ has five layers of facets: both the first and the fifth layer are pentagons; the second and the fourth layers consist of five pentagons respectively; and the third layer contains five hexagons.

Firstly, we construct three coloring bricks $S_1=(24247)$, $S_2=(65372)$, $S_3 =(35716)$. They are pairwisely compatible. Namely there are three compatible pairs:
\begin{gather}
S_1S_2=(24247 \  65372), \notag \\
S_1S_3=(24247 \ 35716), \notag \\
S_2S_3=(65372 \ 57163). \notag 
\end{gather}
These three coloring bricks appeal in all Lemma \ref{lemma: key}, Lemma \ref{lemma: second} and Lemma \ref{lemma: fifth}.
We still use $S_1$, $S_2$, $S_3$ to denote the coloring bricks even after being rotated. For example we keep $S_2S_3$ to represent $(65372 \ 57163)$ although the second brick $(57163)$ of this pair is actually what we rotate $S_3=(35716)$ by angle $\frac{2 \pi}{5}$ from left. But we will claim the elements clearly before adopting the abbreviation sign to avoid misunderstanding. The unified format of our claim is (65372 57163)=$S_2S_3\triangleq A$. Next we select out $a=1$, $b=1$ and $c=4$ as corresponding affix elements from the affix set of $S_1$, $S_2$ and  $S_3$ respectively. 

Next, we can use these coloring bricks and their affix elements to build up $\underbrace {\left( {\tbinom{3}{1}\cdot\tbinom{2}{1}\cdot\tbinom{2}{1} \cdots\tbinom{2}{1}} \right)}_{n+1}= 3\cdot 2^n$ many $\mathbb{Z}_2^3$-coloring functions over the polytope $nP$. The entries of coloring bricks may be adjusted by rotating in order to fit the facet structure of $nP$ as well as obey the adjacent relations required by the three compatible pairs. Denoting $$(35716 ~~24247)=S_3S_1\triangleq A_1.$$
 Now by [$a S_1 A_1a$] we mean the color polytope $2P$ as shown in Figure \ref{figure:4}.

\begin{figure}[H]
	\scalebox{0.73}[0.73]{\includegraphics {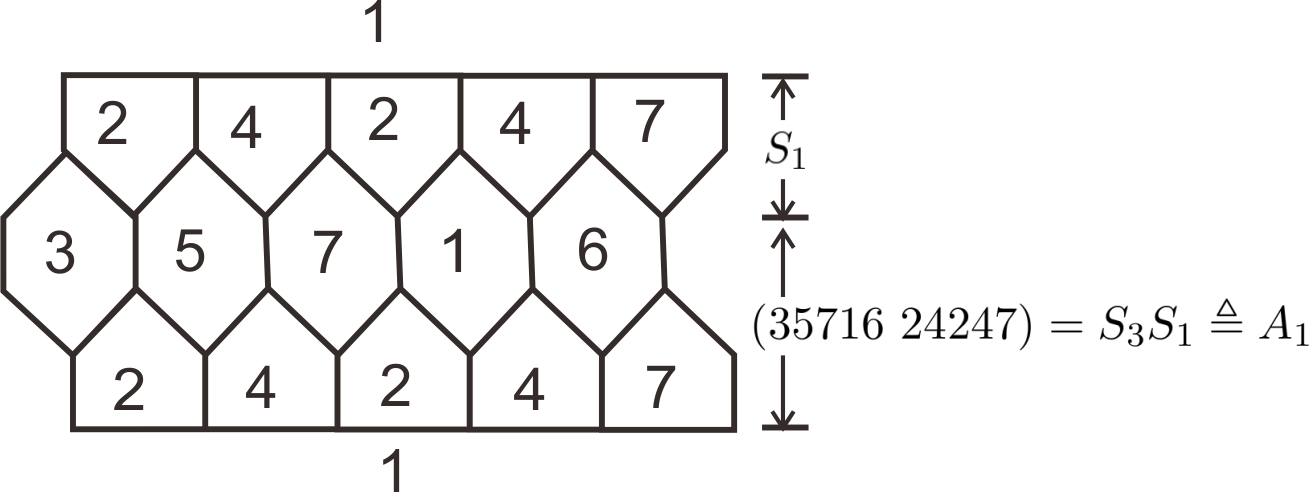}}
	\caption{Colored polytope $2P$.}
	\label{figure:4}
\end{figure}

Following the ordering manner required in Section 2.5, the adjacent matrix of $(\partial(2P))^*$ is:
\begin{center}$
	\begin{pmatrix}
	\begin{smallmatrix}
	1&\	1&\	1&\	1&\	1&\	1&\	0&\	0&\	0&\	0&\	0&\	0&\	0&\	0&\	0&\	0&\	0&\\
	1&\	1&\	1&\	1&\	0&\	0&\	1&\	1&\	0&\	0&\	0&\	0&\	0&\	0&\	0&\	0&\	0&\\
	1&\	1&\	1&\	0&\	1&\	0&\	1&\	0&\	1&\	0&\	0&\	0&\	0&\	0&\	0&\	0&\	0&\\
	1&\	1&\	0&\	1&\	0&\	1&\	0&\	1&\	0&\	1&\	0&\	0&\	0&\	0&\	0&\	0&\	0&\\
	1&\	0&\	1&\	0&\	1&\	1&\	0&\	0&\	1&\	0&\	1&\	0&\	0&\	0&\	0&\	0&\	0&\\
	1&\	0&\	0&\	1&\	1&\	1&\	0&\	0&\	0&\	1&\	1&\	0&\	0&\	0&\	0&\	0&\	0&\\
	0&\	1&\	1&\	0&\	0&\	0&\	1&\	1&\	1&\	0&\	0&\	1&\	1&\	0&\	0&\	0&\	0&\\
	0&\	1&\	0&\	1&\	0&\	0&\	1&\	1&\	0&\	1&\	0&\	1&\	0&\	1&\	0&\	0&\	0&\\
	0&\	0&\	1&\	0&\	1&\	0&\	1&\	0&\	1&\	0&\	1&\	0&\	1&\	0&\	1&\	0&\	0&\\
	0&\	0&\	0&\	1&\	0&\	1&\	0&\	1&\	0&\	1&\	1&\	0&\	0&\	1&\	0&\	1&\	0&\\
	0&\	0&\	0&\	0&\	1&\	1&\	0&\	0&\	1&\	1&\	1&\	0&\	0&\	0&\	1&\	1&\	0&\\
	0&\	0&\	0&\	0&\	0&\	0&\	1&\	1&\	0&\	0&\	0&\	1&\	1&\	1&\	0&\	0&\	1&\\
	0&\	0&\	0&\	0&\	0&\	0&\	1&\	0&\	1&\	0&\	0&\	1&\	1&\	0&\	1&\	0&\	1&\\
	0&\	0&\	0&\	0&\	0&\	0&\	0&\	1&\	0&\	1&\	0&\	1&\	0&\	1&\	0&\	1&\	1&\\
	0&\	0&\	0&\	0&\	0&\	0&\	0&\	0&\	1&\	0&\	1&\	0&\	1&\	0&\	1&\	1&\	1&\\
	0&\	0&\	0&\	0&\	0&\	0&\	0&\	0&\	0&\	1&\	1&\	0&\	0&\	1&\	1&\	1&\	1&\\
	0&\	0&\	0&\	0&\	0&\	0&\	0&\	0&\	0&\	0&\	0&\	1&\	1&\	1&\	1&\	1&\	1&\\
	\end{smallmatrix}
	\end{pmatrix}$
\end{center}

 Because the  labels of facets and the adjacent relations go with the rotation, parts of corresponding coloring vectors are the same if they yield to the same brick. And here the characteristic vector $C$ of the colored polytope  $2P$ is $$(1,~ 2,~ 4,~ 4,~ 2,~ 7,~ 7,~ 1,~ 5,~ 6,~ 3,~ 2,~ 4,~ 4,~ 2,~ 7,~ 1).$$

We write $\lambda$ the characteristic function of $C$. By Theorem \ref{theorem: NakayamaN} and Remark \ref{remark:2}, it's easy to prove that $\lambda$ is admissible and the characteristic matrix $\Delta$ of its uniquely orientable extension is:
\begin{equation}\label{matrix:1}
\begin{pmatrix}\begin{smallmatrix}
0&\	0&\	1&\	1&\	0&\	1&\	1&\	0&\	1&\	1&\	0&\	0&\	1&\	1&\	0&\	1&\	0&\\
0&\	1&\	0&\	0&\	1&\	1&\	1&\	0&\	0&\	1&\	1&\	1&\	0&\	0&\	1&\	1&\	0&\\
1&\	0&\	0&\	0&\	0&\	1&\	1&\	1&\	1&\	0&\	1&\	0&\	0&\	0&\	0&\	1&\	1&\\
0&\	0&\	0&\	0&\	0&\	0&\	0&\	0&\	1&\	1&\	1&\	0&\	0&\	0&\	0&\	0&\	0&\\
\end{smallmatrix}\end{pmatrix}. 
\end{equation}

That is,  $M(2P, \lambda)$ is non-orientable and  $M(2P, \delta)$ is orientable, where $\delta$ is the natural $Z_2^4$-extension of $\lambda$.

Then the row space $row\Delta$, spanned by the four row vectors above, consists of 15 non-zero  row vectors as below. In the following we always mean a $(2^4-1)\times(5n+7)$ matrix, where $n$ is from $nP$, when we refer to the row space of the characteristic matrix of a $\mathbb{Z}_2^4$-coloring $\delta$ over $nP$. Here Matrix (\ref{matrix:2}) is the row space of Matrix (\ref{matrix:1}).
	\begin{equation}\label{matrix:2}
	\begin{pmatrix}
	\begin{smallmatrix}
	0&\	0&\	1&\	1&\	0&\	1&\	1&\	0&\	1&\	1&\	0&\	0&\	1&\	1&\	0&\	1&\	0&\\
	0&\	1&\	0&\	0&\	1&\	1&\	1&\	0&\	0&\	1&\	1&\	1&\	0&\	0&\	1&\	1&\	0&\\
	1&\	0&\	0&\	0&\	0&\	1&\	1&\	1&\	1&\	0&\	1&\	0&\	0&\	0&\	0&\	1&\	1&\\
	0&\	0&\	0&\	0&\	0&\	0&\	0&\	0&\	1&\	1&\	1&\	0&\	0&\	0&\	0&\	0&\	0&\\
	0&\	1&\	1&\	1&\	1&\	0&\	0&\	0&\	1&\	0&\	1&\	1&\	1&\	1&\	1&\	0&\	0&\\
	1&\	0&\	1&\	1&\	0&\	0&\	0&\	1&\	0&\	1&\	1&\	0&\	1&\	1&\	0&\	0&\	1&\\
	0&\	0&\	1&\	1&\	0&\	1&\	1&\	0&\	0&\	0&\	1&\	0&\	1&\	1&\	0&\	1&\	0&\\
	1&\	1&\	0&\	0&\	1&\	0&\	0&\	1&\	1&\	1&\	0&\	1&\	0&\	0&\	1&\	0&\ 1&\\
	0&\	1&\	0&\	0&\	1&\	1&\	1&\	0&\	1&\	0&\	0&\	1&\	0&\	0&\	1&\	1&\	0&\\
	1&\	0&\	0&\	0&\	0&\	1&\	1&\	1&\	0&\	1&\	0&\	0&\	0&\	0&\	0&\	1&\	1&\\
	1&\	1&\	1&\	1&\	1&\	1&\	1&\	1&\	0&\	0&\	0&\	1&\	1&\	1&\	1&\	1&\	1&\\
	1&\	1&\	0&\	0&\	1&\	0&\	0&\	1&\	0&\	0&\	1&\	1&\	0&\	0&\	1&\	0&\	1&\\
	1&\	0&\	1&\	1&\	0&\	0&\	0&\	1&\	1&\	0&\	0&\	0&\	1&\	1&\	0&\	0&\	1&\\
	0&\	1&\	1&\	1&\	1&\	0&\	0&\	0&\	0&\	1&\	0&\	1&\	1&\	1&\	1&\	0&\	0&\\
	1&\	1&\	1&\	1&\	1&\	1&\	1&\	1&\	1&\	1&\	1&\	1&\	1&\	1&\	1&\	1&\	1&\\
	\end{smallmatrix}
	\end{pmatrix}
	\end{equation}

By Theorem \ref{theorem: ChoiP}, we can calculate out $\beta^1(M(2P,\delta))$ through the $\widetilde{\beta^0}$=$\widetilde{\beta_0}$ of its 15 non-empty full-subcomplices $K_{\omega}$. The reduced zero-th Betti number of  each $K_{\omega}$ actually equals to $K_{\omega}$'s connected component number minus one.  We define subset $\omega^{(*)}=\{1,~2,~...~,d\}\subset[m]$ to be a \emph{connected index component} if $V=\{v_1,~v_2,~,...,~v_d\}$ is the vertex set of a connected component of $K_{\omega}$. It can be observed that $\omega^{(*)}$ is a connected index component of $K_{\omega}$ if and only if, for all $i_1, i_2 \in \omega^{(*)}$, we can always find some elements $j_1, j_2, ..., j_s \in \omega^{(*)}$ such that $$a_{i_1 j_1}=a_{j_1 j_2}=...=a_{j_{s-1} j_s}=a_{j_s i_2}=1,$$
where $a_{ij} \in X({K_\omega}$) and $X({K_\omega})$ is the matrix as defined in Section 2.5.

In order to make things more concise, we use the language of CT-Tree to illustrate our proof. For every $i$-th row $\omega_i=(w_{i1},...,w_{ij},...,w_{im})$ of the row space $row\Delta$, where $m=5n+7$ is the number of facets of $nP$ and $1\le i\le 2^4-1$. We define the following indices:

$\omega^1_i:= \{j~ \arrowvert~  1\leqslant j\leqslant m ~ {\rm and} ~ c_{ij}=1, ~ {\rm where} ~ c_{ij}\in row \Delta \}$ with cardinality  $l_i$;

$\omega^0_i:= \{1,2,...,m\} - {\omega^1_i}$.

For the  adjacent matrix $X(nP)$ of $nP$, we use $X(nP, \omega_i)$ to denote the  $l_i\times l_i$ complement minor after excluding all $q$-th rows and $q$-th columns from $X(nP)$ where $q\in \omega^0_i$. We call $b(i):=\{j \arrowvert j>i ~ {\rm and} ~  a_{ij}=1 ,~ {\rm where} ~ a_{ij}\in X(nP, \omega_i)\}$ as \textbf{$i$-band} or \textbf{band of $i$}, and $b(i)$ is a subset of $[m]$.

If we consider the row space  $row\Lambda$ as shown in  (3.2). Then:
\begin{gather}
\omega_1=(0,~ 0,~ 1,~ 1,~ 0,~ 1,~ 1,~ 0,~ 1,~ 1,~ 0,~ 0,~ 1,~ 1,~ 0,~ 1,~ 0), \notag \\
\omega^1_1=\{3,~ 4,~ 6,~ 7,~ 9,~ 10,~ 13,~ 14,~ 16\}, \notag \\
\omega^0_1=\{1,~ 2,~ 5,~ 8,~ 11,~ 12,~ 15,~ 17\}. \notag 
\end{gather}
Now we have
\begin{center}
	$X(nP, \omega_1)=${\footnotesize$\bordermatrix[()]{%
			&3& 4 & 6 & 7 &9  & 10& 13&	14&	16 \cr
			&1&	0&	0&	1&	1&	0&	0&	0&	0\cr
			&0&	1&	1&	0&	0&	1&	0&	0&	0\cr
			&0&	1&	1&	0&	0&	1&	0&	0&	0\cr
			&1&	0&	0&	1&	1&	0&	1&	0&	0\cr
			&1&	0&	0&	1&	1&	0&	1&	0&	0\cr
			&0&	1&	1&	0&	0&	1&	0&	1&	1\cr
			&0&	0&	0&	1&	1&	0&	1&	0&	0\cr
			&0&	0&	0&	0&	0&	1&	0&	1&	1\cr
			&0&	0&	0&	0&	0&	1&	0&	1&	1\cr
		}$},
\end{center}
where the column-name row is exactly $\omega_1^1$.

In order to figure out the $\widetilde\beta^0$ of the full subcomplex that corresponds to $\omega_1$, we apply a \textbf{Connect-Trace Tree(CT-Tree) procedure} to $X(nP, \omega_1)$.

Assuming $\omega^1_1=\{i_1, i_2, ..., i_l\vert i_1<i_2<...<i_l\}$, then
$$
X(nP, \omega_1)=\begin{pmatrix}
a_{i_1{i_1}}& a_{i_1{i_2}} & \cdots & a_{i_1{i_l}} \\
a_{i_2{i_1}}& a_{i_2{i_2}} & \cdots & a_{i_2{i_l}}\\
\vdots & \vdots & \ddots & \vdots\\
a_{i_l{i_1}}& a_{i_l{i_2}} & \cdots & a_{i_l{i_l}}\\
\end{pmatrix}.
$$

Initially, we pick the first element $i_1$ in $\omega^1_1$ and arrange the information into Table \ref{table:2} as shown below. Then we get the 1st-branch of the CT-Tree of this round concerning $\omega_1$. And here we detail the elements of $\{i_j~\arrowvert~ i_j>i_1, \ a_{i_1{j_1}}=1, \ a_{i_1{j_1}}\in X(nP, \omega_1)\}$ by $\{i_{11}, i_{12}, ..., i_{1{l(i_1)}}\}$.

\begin{table}[H]

	\begin{tabular}{c|p{5cm}|c}
		\hline
		lead$(i_j)$ & $i_j$-band $b(i_j)$ & branch structure\\
		\hline
		$i_1$ & $i_{11}, i_{12}, ..., i_{1{l(i_1)}}$ & 1-st branch \\
		\hline
		\end{tabular}
	    \caption{The 1-th branch.}
	    \label{table:2}
\end{table}

Conducting the same operation to every $i$-th row of $X(nP, \omega_i)$, we generate the 2nd-branch of the CT-tree as shown in Table \ref{table:3}, which actually  means placing the $b(i_j)$ row by row where $i_j \in \{i_{11}, i_{12}, ..., i_{1{l(i_1)}}\}$.

\begin{table}[H]
	\begin{tabular}{c|p{5cm}|c|c}
		\hline
		lead$(i_j)$ &\multicolumn{2}{|c|}{$i_j$-band $b(i_j)$}& branch structure\\
		\hline
		$i_1$ & $i_{11},i_{12} ,...,i_{1{l(i_1)}}$ & $b(i_1)$&1st-branch\\
		\hline
		$i_{11}$ & $ i_{111},i_{112} ,...,i_{11{l(i_{11})}}$&$b(i_{11})$& 2nd- branch\\
		
		\vdots & \vdots & &\\
		$i_{1{l(i_1)}}$ & $ i_{1{l(i_1)}1} ,...,i_{1{l(i_1)}{l(i_{1{l{(i_1)}})})}}$&$b(i_{1l(i_i)})$&\\
		\hline
	\end{tabular}
	\caption{The 1st and 2nd branches.}
	\label{table:3}
\end{table}

For example, as to $\omega_1=(0,~ 0,~ 1,~ 1,~ 0,~ 1,~ 1,~ 0,~ 1,~ 1,~ 0,~ 0,~ 1,~ 1,~ 0,~ 1,~ 0),$ the first row in Matrix (\ref{matrix:2}), the first and the second branches are shown in Table  \ref{table:4}:

 \begin{table}[H]
 	\begin{tabular}{c|p{5cm}|c}
 		\hline
 		lead$(i_j)$ & $i_j$-band $b(i_j)$ & branch structure\\
 		\hline
 		3 & 7, 9 & 1st-branch \\
 		\hline
 		7&  9, 13& 2nd-branch \\
 		9& 13&\\
 		\hline
 	\end{tabular}
  	\caption{The 1st and 2nd branches of the first round CT-tree of $\omega_1$.} \label{table:4}
 \end{table}

The second branch is obtained by \textbf{branching a row}. There is only one single band $b(i_1)=\{i_{11}, i_{12},  ..., i_{1{l(i_1)}}\}$ in the first branch. And the  band part of the second branch consists of all the bands of elements of that row, that is ($b(i_{11})$, ... , $b(i_{1l(i_1)})$). We use band as a unit to depict the shelf structure within a branch. There are  two bands in the 2nd-branch of $\omega_1$ as shown in Table \ref{table:4}. Moreover, the operation of \textbf{u-branching a band} means cancelling the repeating elements when branching a band and leaving only the one in the top-left position. Then we can define the  operation of \textbf{u-branching a branch}, which resulting from u-branching every band of a branch and later place them row by row into the table. Thus we can get the $(n+1)$-th branch $(n>2)$ by u-branching the $n$-th branch. We generate the CT-tree by u-branching until we reach a branch whose band part is empty after removing the elements that are identical with those of the corresponding lead part. Such branch is called a \textbf{cadence branch}. Denoting the last lead as $i_q$. Then we have a lead vector of CT-tree $\omega^{(1)}_1:= (i_1,\ldots, i_q)$. It can be shown that $\omega_1^{(1)}$ is a connected index component. Namely we have the following proposition.

\begin{proposition}
The full-subcomplex corresponds to $\omega^{(1)}_1$ is exactly a connected component of $K_{\omega_1}$, where $K_{\omega_1}$ is the subcomplex of $K$ by restricting to $\omega_1^{(1)}\subset [m]$.
\end{proposition}
\begin{proof}

By the definition of the connected index component and the way of picking elements for $\omega^{(1)}_1$, we know that all the entries of $\omega^{(1)}_1$ are contained in the same connected index component. Ending with cadence branch ensuring the maximality.
\end{proof}

Setting  $\omega^1_{(1)}$= $\omega^1_1\diagdown \omega^{(1)}_1$ and applying the above CT-tree procedure to $\omega^1_{(1)}$, we then find the second connected component $\omega^{(2)}_1$ of $K_{\omega_1}$. Again setting  $\omega^1_{(2)}$= $\omega^1_{(1)}\diagdown \omega^{(2)}_1$ and applying the above CT-tree procedure to $\omega^1_{(2)}$.

Continuing this procedure until we reach  a $t \in \mathbb{Z}_{\geq 1}$, such that $\omega^1_{(t)}
$ is empty. Then we can conclude that there are exactly $t$ connected components in the subcomplex  $K_{\omega_1}$.
\begin{table}[H]
	\begin{floatrow}
		\capbtabbox{
			\begin{tabular}{c|c|c}
				\hline
				lead$(i_j)$ & $b(i_j)$ & branch structure\\
				\hline
				3 & 7, 9 & 1st-branch \\
				\hline
				7&  9, 13& 2nd-branch\\
				9& &\\
				\hline
				13& & 3th(cadence)-branch\\
				\hline	
			\end{tabular}
		}{
		\caption{Tracing the 1st connected component of $K_{\omega_1}$ corresponding to $\omega^{(1)}_1$.}
		\label{table:5}
	}
	\capbtabbox{
		\begin{tabular}{c|c|c}
			\hline
			lead$(i_j)$ & $b(i_j)$ & branch structure\\
			\hline
			4 & 6, 10 & 1st-branch\\
			6& &\\
			\hline
			10&14, 16&2nd-branch\\
			\hline
			14 & 16 & 3th(cadence)-branch\\
			16& &\\
			\hline
		\end{tabular}
	}{
	\caption{Tracing the 2nd connected component of $K_{\omega_1}$ corresponding to $\omega^{(2)}_1$.}
	\label{table:6}
}
\end{floatrow}
\end{table}
For example, as to $\omega_1=(0,~ 0,~ 1,~ 1,~ 0,~1,~1,~ 0,~ 1,~ 1,~ 0,~ 0,~ 1,~ 1,~ 0,~ 1,~ 0)$, the first row in Matrix (\ref{matrix:2}), $t$ equals to 2. That means there are exactly 2 connected components, where $\omega^{(1)}_1=(3,~ 7,~ 9,~ 13)$ and $\omega^{(2)}_1=(4,~ 6,~ 10,~ 14,~ 16)$, so $\widetilde\beta^0(K_{\omega_1})=1$. These  CT-trees are shown in Table \ref{table:5} and Table \ref{table:6}.

By this method, we can calculate out that $\beta^1(M(2P,\delta))$=3 as shown in Table \ref{table:7} and thus finish the proof of Lemma  \ref{lemma: key} in the case $n=2$.

 \begin{table}[H]
  	\begin{tabular}{|c|c|c|c|c|c|c|c|c|c|c|c|c|c|c|c|c|}
 		\hline
 		$i$-th subcomplex&1&2&3&4&5&6&7&8&9&10&11&12&13&14&15&$total~~\beta^1$\\
 		\hline
 		$\widetilde\beta^0(i)$&1&1&0&0&0&0&1&0&0&0&0&0&0&0&0&3\\
 		\hline
 	\end{tabular}
 	\caption{$\beta^1$ of $M(2P,\delta)$ }
 	\label{table:7}
 \end{table}
The adjacent matrix of $(\partial(nP))^*$ changes regularly as $n$ increases. Here  ``regularly" results from the same adjacency changing pattern upon the same parity. Three  ordered polytopes $2P$, $4P$ and $6P$ are shown in Figure \ref{figure:5}. The adjacent manner of the 17th facet of the polytope $2P$ is the only one that is different from the same position of the polytope $4P$, while the facets from 17th to 27th of the polytope $4P$ follow the same order manner of facets from 7th to 17th of the polytope $2P$ by simply plus 10. The same changing pattern makes sense when we compare the adjacency patterns between the polytopes $4P$ and $6P$. Which means the adjacent manner of the 27th facet of the polytope $4P$ is the only one that is different from the same position of the polytope $6P$, while the facets from 27th to 37th of the polytope $6P$ follow  the same order manner of facets from 17th to 27th of the polytope $4P$ by simply plus 10. Generally speaking, the adjacent manner of the $(5n+7)$th facet of the polytope $nP$ is the only one that different from the same position of the polytope  $(n+2)P$ while facets from $(5n+7)$th to $(5(n+2)+7)$th of the polytope $(n+2)P$ are following the order manner of facets from $(5(n-2)+7)$th to $(5n+7)$th of the polytope $nP$ by simply plus 10.
\begin{figure}[H]
	\scalebox{0.75}[0.75]{\includegraphics {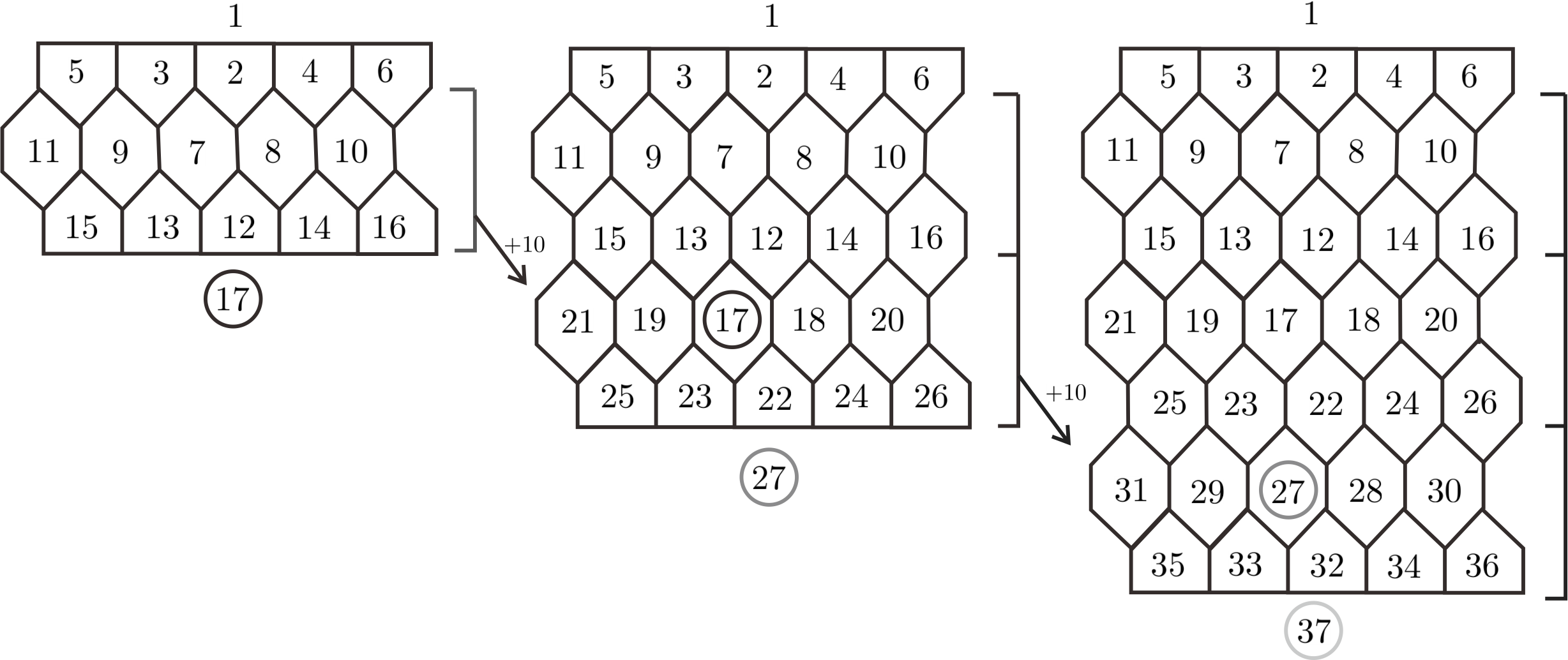}}
	\caption{Orderings of  polytopes $2P$, $4P$ and $6P$.}\label{figure:5}
\end{figure}
From the procedure of the  Betti number calculation, we can see the first Betti number  of $M(nP,\delta)$ increases steadily if $\widetilde\beta^0$ of each full-subcomplex increases in a certain way. The number of connected components of each full-subcomplex is determined by the selected vertices and their adjacent patterns, which are set by characteristic vector and adjacent matrix respectively. From the previous paragraph we can see the adjacent matrices change regularly from $nP$ to $(n+2)P$. Then if the last two coloring bricks of $(n+2)P$ are actually the copy of the last two bricks of $nP$, the 0-skeleton is deemed to increase by constant when changes between successive steps remain the same, like that from $2P$ to $4P$ and from $4P$ to $6P$. Such statement can be easily proved by CT-trees.

For example, we now have colored polytopes $2P$, $4P$ and $6P$ with coloring vectors
\begin{gather}
[aS_1S_3S_1a], \notag \\
[aS_1S_3S_1S_3S_1a]\triangleq[aS_1S_3S_1a]+\{S_3S_1\}, \notag \\
[aS_1S_3S_1S_3S_1S_3S_1a]\triangleq[aS_1S_3S_1a]+2\cdot\{S_3S_1\}. \notag 
\end{gather}
\noindent Their characteristic functions are named as $\lambda=\lambda^0$, $\lambda^1$ and $\lambda^2$ respectively, where the upper index represent how many times are the last two bricks $\{31\}$ of the  coloring $[a1\{31\}a]$ of $\lambda$ being repeated. The non-orientability of these $\mathbb{Z}_2^3$-colorings can be checked by Theorem \ref{theorem: NakayamaN}. Moreover, we can get their natural  $\mathbb{Z}_2^4$-extensions $\delta=\delta^0$, $\delta^1$ and $\delta^2$ respectively through Proposition \ref{remark:1}. That is $M(2P, \delta)$, $M(4P, \delta^1)$ and $M(6P, \delta^2)$ are the orientable double covers of the non-orientable manifolds $M(2P, \lambda)$, $M(4P, \lambda^1)$ and $M(6P, \lambda^2)$ respectively.

 By the previous methodology we can  calculate out that $\beta^1(M(4P,\delta^1))$=5 as shown in Table \ref{table:8}. And the $\beta^0$ of the first subcomplex is 2 by CT-tree process as displayed in Table \ref{table:9} and Table \ref{table:10}.

\begin{table}[H]
	\caption{$\beta^1$ of $M(4P,\delta^1)$.}
	\begin{tabular}{|c|c|c|c|c|c|c|c|c|c|c|c|c|c|c|c|}
		\hline
		1&2&3&4&5&6&7&8&9&10&11&12&13&14&15&$total ~~ \beta^1$\\
		\hline
		1&1&0&1&0&1&1&0&0&0&0&0&0&0&0&5\\
		\hline
	\end{tabular}
	\label{table:8}
\end{table}

\begin{table}[H]
	\begin{floatrow}
		\capbtabbox{
			\begin{tabular}{c|c|c}
				\hline
				lead$(i_j)$ & $b(i_j)$ & branch structure\\
				\hline
				3 & 7, 9 & 1st-branch \\
				\hline
				7&  9, 13& 2nd-branch\\
				9& &\\
				\hline
				13 & 17, 19 & 3th-branch \\
				\hline
				17&  19, 23& 4th-branch\\
				19& &\\
				\hline
				23& & 5th(cadence)-branch\\
				\hline	
			\end{tabular}
		}{
		\caption{Tracing the 1st connected component of $K_{\omega_1}$ with respect to $\delta^1$.}
		\label{table:9}
	}
	\capbtabbox{
		\begin{tabular}{c|c|c}
			\hline
			lead$(i_j)$ & $b(i_j)$ & branch structure\\
			\hline
			4 & 6, 10 & 1st-branch\\
			6& &\\
			\hline
			10&14, 16&2nd-branch\\
			\hline
			14 & 16, 20 & 3th-branch\\
			16& &\\
			\hline
			20&24, 26&4th-branch\\
			\hline
			24 & 26& 5th-(cadence)-branch\\
			26& &\\
			\hline
		\end{tabular}
	}{
	\caption{Tracing the 2nd connected component of $K_{\omega_1}$ with respect to $\delta^1$.}
	\label{table:10}
}
\end{floatrow}
\end{table}

Similarly we can calculate out that $\beta^1(M(6P,\delta^2))=7$ with details in Table \ref{table:11}. And the $\beta^0$ of the first subcomplex is 2 by CT-trees as shown in Table \ref{table:12} and Table \ref{table:13}.

\begin{table}[H]
	\begin{tabular}{|c|c|c|c|c|c|c|c|c|c|c|c|c|c|c|c|}
		\hline
		1&2&3&4&5&6&7&8&9&10&11&12&13&14&15&$total~~\beta^1$\\
		\hline
		1&1&0&2&0&2&1&0&0&0&0&0&0&0&0&7\\
		\hline
	\end{tabular}
	\caption{$\beta^1$ of $M(6P,\delta^2)$.}
	\label{table:11}
\end{table}

\begin{table}[H]
	\begin{floatrow}
		\capbtabbox{
			\begin{tabular}{c|c|c}
				\hline
				lead$(i_j)$ & $b(i_j)$ & branch structure\\
				\hline
				3 & 7, 9 & 1st-branch \\
				\hline
				7&  9, 13& 2nd-branch\\
				9& &\\
				\hline
				13& 17, 19& 3th-branch\\
				\hline
				17&  19, 23& 4th-branch\\
				19& &\\	
				\hline
				23& 27, 29& 5th-branch\\				
				\hline
				27& 29, 33& 6th-branch\\
				29& &\\	
				\hline
				33& & 7th(cadence)-branch \\	
				\hline	
			\end{tabular}
		}{
		\caption{Tracing the 1st connected component of $K_{\omega_1}$ with respect to $\delta^2$.}
		\label{table:12}
	}
	\capbtabbox{
		\begin{tabular}{c|c|c}
			\hline
			lead$(i_j)$ & $b(i_j)$ & branch structure\\
			\hline
			4 & 6, 10 & 1st-branch\\
			6& &\\
			\hline
			10&14, 16&2nd-branch\\
			\hline
			14 & 16, 20 & 3th-branch\\
			16& &\\
			\hline
			20&24, 26&4th-branch\\
			\hline
			24 & 26, 30 & 5th-branch\\
			26& &\\
			\hline
			30&34, 36&6th-branch\\
			\hline
			34 & 36 & 7th(cadence)-branch \\
			36& &\\
			\hline
		\end{tabular}
	}{\caption{Tracing the 2nd connected component of $K_{\omega_1}$ with respect to $\delta^2$.}
	\label{table:13}
}
\end{floatrow}
\end{table}

Thereupon we have
 $$\beta^1(M(4P, \delta^1))-\beta^1(M(2P, \delta))= \beta^1(M(6P,\delta^2))-\beta^1(M(4P, \delta^1)).$$
 Then it can be inferred that for every $j\in \mathbb{Z}_+$,
 $$\beta^1(M((2j)P, \delta^{j-1}))=\beta^1(M(2P,\delta))+(j-1)(\beta^1(M(4P, \delta^1))-\beta^1(M(2P, \delta^0))).$$
 It means that the Betti number sequence follows an arithmetic progression if the first three numbers lie in an arithmetic progression. We only need to determine the first item and the common difference of such a sequence. The above analysis can be concluded into the following proposition:

\begin{proposition} \label{prop:1}
 Given a $\mathbb{Z}_2^4$-coloring $\delta$ over the polytope  $nP$. Then for an arbitrary even number $s\geqslant n$, if
 $$\beta^1(M((n+2)P, \delta^{(1)}))-\beta^1(M(nP, \delta))= \beta^1(M((n+4)P,\delta^{(2)}))-\beta^1(M((n+2)P, \delta^{(1)})),$$
then we have
$$\beta^1(M(sP, \delta^{(\frac{s-n}{2})}))=\beta^1(M(nP,\delta))+(\frac{s-n}{2})(\beta^1(M(n+1)P,\delta^1)-\beta^1(M(nP, \delta))).$$
where $\delta^{(t)}$ represents a $\mathbb{Z}_2^4$-coloring over the polytope $(n+2t)P$. The coloring vector of $\delta^{(t)}$ is obtained by duplicating the last two bricks of $\delta$ for $t$ times.
\end{proposition}

Followed by Proposition \ref{prop:1} as well as the calculation result of $\beta^1(M(4P,\delta^1))$ and $\beta^1(M(6P,\delta^2))$, we can come up with Table \ref{table:14}.

\begin{table}[H]
	\begin{tabular}{|c|c|c|c|c|c|}
		\hline
		& $n=2$ & $n=4$ & $n=6$ & $\cdots$ & $n=2a$,  $a\in \mathbb{Z}_{> 0}$\\
		\hline
		1 & 1 & 1 & 1 & $\cdots$ &1\\
		\hline
		2 & 1 & 1 & 1 & $\cdots$ &1\\
		\hline
		3 & 0 & 0 & 0 & $\cdots$ &0\\
		\hline
		4 & 0 & 1 & 2 & $\cdots$ & $a-1$\\
		\hline
		5 & 0 & 0 & 0 & $\cdots$ &0\\
		\hline
		6 & 0 & 1 & 2 & $\cdots$ &$a-1$\\
		\hline
		7 & 1 & 1 & 1 & $\cdots$ &1\\
		\hline
		8 & 0 & 0 & 0 & $\cdots$ &0\\
		\hline
		9 & 0 & 0 & 0 & $\cdots$ &0\\
		\hline
		10 & 0 & 0 & 0 & $\cdots$ &0\\
		\hline
		11 & 0 & 0 & 0 & $\cdots$ &0\\
		\hline
		12 & 0 & 0 & 0 & $\cdots$ &0\\
		\hline
		13 & 0 & 0 & 0 & $\cdots$ &0\\
		\hline
		14 & 0 & 0 & 0 & $\cdots$ &0\\
		\hline
		15 & 0 & 0 & 0 & $\cdots$ &0\\
		\hline
		total $\beta^1$ & 3 & 5 & 7 & $\cdots$ & $2(a-1)+3=n+1$\\
		\hline
	\end{tabular}
	\caption{$\beta^1$ changes in  Lemma \ref{lemma: key}.}
	 \label{table:14}
\end{table}

By now we finally fulfill the proof of Lemma \ref{lemma: key}.
\end{proof}

\section{\textbf{Proof of Theorem \ref{theorem: bounding} for $n$ is even}}

In this section, we prove Theorem \ref{theorem: bounding} for an even $n \in \mathbb{Z}_{+}$. It is similar to  the proof of Lemma \ref{lemma: key}.
%-----------------------------------------------------------------------------�ڶ����������������޹ؿ��Ա������������
\begin{lemma}  \label{lemma: second}
	Let $n\in \mathbb{Z}_{+}$ be any  even number, there is a non-orientable $\mathbb{Z}_2^3$-coloring $\lambda$  over the polytope $nP$, such that for its natural associated  $\mathbb{Z}_2^4$-coloring $\delta$, we have   $\beta^1 (M(nP, \delta))= 5n-3$.
\end{lemma}

\begin{proof}
We still use the three coloring bricks $S_1=(24247)$, $S_2=(65372)$ and $S_3 =(35716)$ together with affix elements $a=1$, $b=1$ and $c=4$ as we adopted in Lemma \ref{lemma: key} and set up
$$(72424~~ 65372)=S_1 S_2\triangleq A_1.$$
 Following the same idea in Lemma \ref{lemma: key}, we firstly select out a suitable  non-orientable $\mathbb{Z}_2^3$-coloring $\lambda$ over the polytope  $2P$ as below $$(1,~3,~5,~7,~6,~2,~4,~2,~2,~4,~7,~3,~5,~7,~6,~2,~1),$$
such that $\beta^1(M(2P, \delta))=7$, where $\delta$ is the natural $\mathbb{Z}_2^4$-extension of $\lambda$.
Then we repeat the last two bricks for $t$ times in order to construct a coloring vector over the polytope $(2+2t)P$ and denote its characteristic function as $\lambda^t$. By Theorem   \ref{theorem: NakayamaN}  and Proposition \ref{remark:1},  we can get the unique admissible extension $\delta^t$ of $\lambda^t$. The sum of every column of the characteristic matrix of $\delta^t$ is 1 mod 2. That is,  $M((2+2t)P,\delta^t)$ is the orientable double cover of the non-orientable manifold $M((2+2t)P,\lambda^t)$. Finally we calculate out the increasing law of $\beta^1$ as shown in Table \ref{table:15}.

\begin{table}[H]
		\begin{tabular}{|c|c|c|c|c|c|}
			\hline
			& $n=2$ & $n=4$ & $n=6$ & $\cdots$ & $n=2+2t$,  $t\in\mathbb{Z}_{>0}$\\
			\hline
			1 & 0 & 0 & 0 & $\cdots$ &0\\
			\hline
			2 & 0 & 0 & 0 & $\cdots$ &0\\
			\hline
			3 & 2 & 4 & 6 & $\cdots$ &$2t+2$\\
			\hline
			4 & 1 & 2 & 3 & $\cdots$ &$t+1$\\
			\hline
			5 & 0 & 0 & 0 & $\cdots$ &0\\
			\hline
			6 & 1 & 3 & 5 & $\cdots$ &$2t+1$\\
			\hline
			7 & 0 & 0 & 0 & $\cdots$ &0\\
			\hline
			8 & 1& 2 & 3 & $\cdots$ &$t+1$\\
			\hline
			9 & 0 & 0 & 0 & $\cdots$ &0\\
			\hline
			10 & 0 & 1 & 2 & $\cdots$ &$t$\\
			\hline
			11 & 0 & 0 & 0 & $\cdots$ &0\\
			\hline
			12 & 1 & 3 & 5 & $\cdots$ &$2t+1$\\
			\hline
			13 & 1 & 2 & 3 & $\cdots$ &$t+1$\\
			\hline
			14 & 0 & 0 & 0 & $\cdots$ &0\\
			\hline
			15 & 0 & 0 & 0 & $\cdots$ &0\\
			\hline
			total $\beta^1$ & 7 & 17 & 27 & $\cdots$ & $10t+7=5n-3$\\
			\hline
		\end{tabular}
		\caption{$\beta^1$ changes for Lemma \ref{lemma: second}.}
		\label{table:15}
\end{table}

Thus we can always find out a non-orientable $\mathbb{Z}_2^3$-colorings $\lambda$ such that its natural $\mathbb{Z}_2^4$-extension $\delta$ satisfying $\beta^1(M(nP,\delta))=5n-3$.
\end{proof}
%-------------------------------------------------------------------------�����������ɣ��ѳ��������޹ر�����Ⱦɫ����г�
\begin{lemma}\label{lemma: third}
For any even integer $n\in \mathbb{Z}_{+}$ and any odd integer $k\in [5n-1,5n+3]$, there is a non-orientable $\mathbb{Z}_2^3$-coloring $\lambda$  over the polytope $nP$, such that for its natural associated  $\mathbb{Z}_2^4$-coloring $\delta$, we have $\beta^1 (M(nP, \delta))=k$.
\end{lemma}

\begin{proof}
We still start at $n=2$ and select out suitable characteristic functions to construct manifolds, whose first Betti numbers would add up by $10t$ when the last pair of their coloring bricks are repeated for $t$ times. Firstly we prepare some bricks and affixes for constructing the wanted coloring vectors as shown in Table \ref{table:16}.

\begin{table}[H]
	\begin{tabular}{|c|c|c|c|}
		\hline
		$a\in\mathcal{A}(S_1)$ & brick $S_1$& brick $S_2$& compatible pair that being repeated  \\
		\hline
		1 &34246 &26513&(26513 34246)=$S_2S_1\triangleq A_1$ \\	
		\hline
	\end{tabular}
		\caption{Bricks and affixes for $\lambda_1^t$ in Lemma \ref{lemma: third}.}
		\label{table:16}
\end{table}

Let $\lambda_1^0$, $\lambda_1^1$ and $\lambda_1^2$ be the three non-orientable $\mathbb{Z}_2^3$-coloring characteristic functions of coloring vectors 
\begin{gather}
[aS_1A_1a], \notag \\
[aS_1A_1A_1a]\triangleq[aS_1A_1a]+A_1, \notag \\
[aS_1A_1A_1A_1a]\triangleq[aS_1A_1a]+2\cdot A_1, \notag 
\end{gather}
which are over the polytopes  $2P$, $4P$ and $6P$ respectively. Their characteristic vectors are
\begin{gather}
(1,~2,~4,~4,~3,~6,~5,~1,~6,~3,~2,~2,~4,~4,~3,~6,~1), \notag \\
(1,~2,~4,~4,~3,~6,~5,~1,~6,~3,~2,~2,~4,~4,~3,~6,~5,~1,~6,~3,~2,~2,~4,~4,~3,~6,~1), \notag \\
(1,~2,~4,~4,~3,~6,~5,~1,~6,~3,~2,~2,~4,~4,~3,~6,~5,~1,~6,~3,~2,~2,~4,~4,~3,~6,~5,~1,~6,~3,~2,~2,~4,~4,~3,~6,~1). \notag 
\end{gather}
The natural associated $\mathbb{Z}_2^4$-extensions are denoted as $\delta_1^0$, $\delta_1^1$ and $\delta_1^2$. By calculation we have the first Betti numbers  of those recovered manifolds, namely    $\beta^1(M(2P,\delta_1^0))$,  $\beta^1(M(4P,\delta_1^1))$ and  $\beta^1(M(6P,\delta_1^2))$,  are 13, 23 and 33 respectively. Thus according to Proposition \ref{prop:1},
\begin{equation}\label{equation:01}
 \beta^1(M((2+2t)P,\delta_1^t))=13+10t,~~ t\in \mathbb{Z}_{+}.
\end{equation}

Similarly, we claim some bricks and affixes in Table \ref{table:17}.

\begin{table}[H]
	\begin{tabular}{|c|c|c|c|c|c|}
		\hline
		$a\in\mathcal{A}(S_1)$ &brick $S_1$& brick $S_2$& brick $S_3$ &	$b~or~c\in\mathcal{A}(S_3)$&compatible pair that being repeated\\
		\hline
		1 &24246 & 73153& 14245& 3~ or ~ 7&(73153 14245)=$S_2S_3\triangleq A_2$\\	
		\hline
	\end{tabular}
	\caption{Bricks and affixes for $\lambda_2^t$ and $\lambda_3^t$ of Lemma \ref{lemma: third}.}
	\label{table:17}
\end{table}

Here we take $\lambda_2^0$, $\lambda_2^1$ and $\lambda_2^2$ to be the three non-orientable $\mathbb{Z}_2^3$-coloring characteristic functions of coloring vectors
\begin{gather}
[aS_1A_2b], \notag \\
[aS_1A_2A_2b]\triangleq[aS_1A_2b]+A_2, \notag \\
[aS_1A_2b]\triangleq[aS_1A_2b]+2\cdot A_2, \notag 
\end{gather}
which are over the polytopes $2P$, $4P$ and $6P$ respectively. And denoting $\lambda_3^0$, $\lambda_3^1$ and $\lambda_3^2$ to be the $\mathbb{Z}_2^3$-coloring characteristic functions of coloring vectors
\begin{gather}
[aS_1A_2c], \notag \\
[aS_1A_2A_2c]\triangleq[aS_1A_2c]+A_2, \notag \\
[aS_1A_2A_2A_2c]\triangleq[aS_1A_2c]+2\cdot A_2, \notag 
\end{gather}
which are over the polytopes $2P$, $4P$ and $6P$ respectively. Moreover, their natural associated $\mathbb{Z}_2^4$-extensions are $\delta_2^0$, $\delta_2^1$, $\delta_2^2$ and $\delta_3^0$, $\delta_3^1$, $\delta_3^2$. By calculation we have the first Betti numbers of these recovered manifolds, namely  $\beta^1(M(2P, \delta_2^0))$,  $\beta^1(M(4P, \delta_2^1))$, $\beta^1(M(6P, \delta_2^2))$ and $\beta^1(M(2P, \delta_3^0))$,  $\beta^1(M(4P, \delta_3^1))$, $\beta^1(M(6P, \delta_3^2))$, are 15, 25, 35 and 17, 27, 37 respectively. Thus we have, for each $t\in \mathbb{Z}_{+}$,
\begin{equation}\label{equation:02}
\beta^1(M((2+2t)P, \delta_2^t))=15+10t,
\end{equation}
\begin{equation}\label{equation:03}
\beta^1(M((2+2t)P,\delta_3^t))=17+10t.
\end{equation}
By (\ref{equation:01}), (\ref{equation:02}) and (\ref{equation:03}), we fulfill the proof of Lemma \ref{lemma: third}.\end{proof}
%--------------------------------------------------------------------------------------------------ǰ�˲������仯
\begin{lemma} \label{lemma: fouth}
	For any even integer $n\in\mathbb{Z}_{+}$ and any odd integer $k \in[1, n-1]$, there is a non-orientable $\mathbb{Z}_2^3$-coloring $\lambda$  over the polytope $nP$, such that for its natural associated  $\mathbb{Z}_2^4$-coloring $\delta$, we have  $\beta^1 (M(nP,\delta)) = k$.
\end{lemma}

\begin{proof}
	Using the notations of $\{S_i\}_{i=1}^4$,  $\{A_i\}_{i=1}^2$ and the affix elements $a$, $b$ to represent elements in Table \ref{table:18}.
	
	\begin{table}[H]
		\begin{tabular}{|c|c|c|c|c|c|c|}
			\hline
			$a\in\mathcal{A}(S_1)$ &brick $S_1$& brick $S_2$& $b\in\mathcal{A}(S_2)$&brick $S_3$ &brick $S_4$&compatible pair that being repeated\\
			\hline
			1 &24247 &54241&3&67172&73172&(67172 54241)$\triangleq A_1$ \\	
			 & &&&&&(73172 54241)$\triangleq A_2$ \\	
			\hline
		\end{tabular}
		\caption{Bricks and affixes for $\lambda^{(t_1, t_2)}$ of Lemma \ref{lemma: fouth}.}
		\label{table:18}
	\end{table}
	
Firstly we find out a non-orientable $\mathbb{Z}_2^3$-coloring $\lambda$ over the polytope $2P$, whose coloring vector is $[aS_1 S_3 S_2 b]\triangleq A$. Its natural $\mathbb{Z}_2^4$-extension is denoted by $\delta$. Using $\lambda^{(t_1,t_2)}$ to represent the $\mathbb{Z}_2^3$-coloring characteristic function of coloring vector
$$[aS_1S_3S_2\underbrace{A_1, ..., A_1}_{t_1}\underbrace{A_2, ..., A_2}_{t_2}b] \triangleq A+t_1A_1+t_2A_2$$
over the polytope $2(t_1+t_2+1)P$. Moreover $\delta^{(t_1, t_2)}$ is the natural $\mathbb{Z}_2^4$-extension of $\lambda^{(t_1, t_2)}$, which is also defined on the polytope $2(t_1+t_2+1)P$. Specially  we have $\lambda^{(0, 0)}=\lambda$. The coloring vectors of  $\lambda^{(1, 0)}$ and $\lambda^{(0,1)}$ are $[aS_1S_3S_2A_1b]$ and $[aS_1S_3S_2A_2b]$ respectively. It is necessary to ensure the compatibility of $S_2$ and $S_4$ if we want to glue $A_2$ after $A_1$. And here everything has checked to be fine. Then we calculate out the Betti numbers and arrange them in Table \ref{table:19}.

	\begin{table}[H]
		\begin{tabular}{|c|c|c|c|}
			\hline
			$\beta^1(M(2P,\delta^{(0,0)}))=1$ & $\beta^1(M(4P,\delta^{(1,0)}))=1$&$\beta^1(M(6P, \delta^{(2,0)}))=1$&...\\
			    &$\beta^1(M(4P, \delta^{(0, 1)}))=3$ &$\beta^1(M(6P,\delta^{(1, 1)}))=3$&... \\	
			& &$\beta^1(M(6P, \delta^{(0, 2)}))=5$&...\\	
			\hline
		\end{tabular}
		\caption{ $\beta^1$ change of $\delta^{(t_1, t_2)}$ in Lemma \ref{lemma: fouth}.}
		\label{table:19}
	\end{table}

By former analysis in Lemma \ref{lemma: key}, the Betti number sequence follows certain pattern if the characteristic vectors and the adjacent matrices change regularly. So if we have the first three numbers lying in an arithmetic progression, then the Betti number sequence follows an arithmetic progression. From 
\begin{gather}
\beta^1(M(2P,\delta^{(0,0)}))=1, \notag \\
\beta^1(M(4P,\delta^{(1,0)}))=1, \notag \\
\beta^1(M(6P,\delta^{(2,0)}))=1, \notag 
\end{gather}
we have
\begin{equation}\label{equation:1}
\beta^1(M(2(t_1+t_2+1)P, \delta^{(t_1,t_2)})) = \beta^1(M(2(t_1 + t_2+2)P,\delta^{(t_1+1,t_2)})).
 \end{equation}
From
	\begin{gather}
	\beta^1(M(2P,\delta^{(0,0)}))=1,\notag\\
	\beta^1(M(4P,\delta^{(0,1)}))=3,\notag\\
	\beta^1(M(6P,\delta^{(0,2)}))=5,\notag
	\end{gather}
we have
\begin{equation}\label{equation:2}
 \beta^1(M(2(t_1 + t_2+1)P, \delta^{(t_1,t_2)}))+2 =\beta^1(M(2(t_1 + t_2+2)P, \delta^{(t_1,t_2+1)})).
\end{equation}
By (\ref{equation:1}) and (\ref{equation:2}) we obtain
\begin{equation}\label{equation:3}
 \beta^1(M(nP,\delta^{(t,\frac{n}{2}-1-t)}))= n-2t-1,
\end{equation}
where $n$ is even,  and $0\leqslant t\leqslant \frac{n}{2}-1$.

Thus we finish the proof of Lemma  \ref{lemma: fouth}.
\end{proof}

%--------------------------------------------------------------------------------------------------�в��������仯
\begin{lemma} \label{lemma: fifth}
	For any even integer $n\in\mathbb{Z}_{+}$ and any odd integer $k \in[n+3,5n-5]$, there is a non-orientable $\mathbb{Z}_2^3$-coloring $\lambda$  over the polytope $nP$, such that for its natural associated  $\mathbb{Z}_2^4$-coloring $\delta$, we have  $\beta^1 (M(nP,\delta)) = k$.
\end{lemma}

\begin{proof}
Keeping the notations of $S_1=(24247)$, $S_2=(65372)$, $S_3 =(35716)$ and their affix colorings $a=1$, $b=1$, $c=4$ that claimed in Lemma \ref{lemma: key}. Four compatible pairs useful to this proof are arranged in Table \ref{table:20}.

	\begin{table}[H]
		\begin{tabular}{|c|c|c|c|c|c|}
			\hline
			$a\in\mathcal{A}(S_1)$ &brick $S_1$& $b\in\mathcal{A}(S_2)$ &brick $S_2$& $c\in\mathcal{A}(S_3)$&brick $S_3$ \\
			\hline
			1 &24247 &1&65372&4&35716\\
			\hline
			\multicolumn{6}{|c|}{compatible pair that being repeated}\\
			\hline
			\multicolumn{6}{|c|}{(42472 71635)=$S_1S_3\triangleq A_1$}\\
			\multicolumn{6}{|c|}{(42472 37265)=$S_1S_2\triangleq A_2$}\\
			\multicolumn{6}{|c|}{(65372 24247)=$S_2S_1\triangleq A_3$}\\
			\multicolumn{6}{|c|}{(65372 71635)=$S_2S_3\triangleq A_4$}\\
	     	\hline
		\end{tabular}
		\caption{Bricks and affixes for $\lambda_i^{(t_1,t_2)}$ in Lemma \ref{lemma: fifth}.}
		\label{table:20}
	\end{table}
Firstly we seek out two non-orientable $\mathbb{Z}_2^3$-colorings $\lambda$ and $\widetilde{\lambda}$ on the polytope $2P$, whose colorings are $[aS_1 S_2 S_3c]$ and $[aS_1 S_2 S_1a]$. Their natural $\mathbb{Z}_2^4$-extensions are denoted as $\delta$ and $\widetilde{\delta}$. By calculation we have
\begin{equation}\label{equation:0}
\beta^1(M(2P,\delta))=\beta^1(M(2P,\widetilde{\delta}))=5.
\end{equation}

Now we use $\lambda_i^{(t_1,t_2)}$ to represent the $\mathbb{Z}_2^3$-coloring  characteristic function of coloring vector
$$[aS_1\underbrace{A_3,...,A_3}_{t_1}A_4 \underbrace{A_1,...,A_1}_{t_2}A_i*]\triangleq [aS_1]+t_1A_3+A_4+t_2A_1+A_i,$$
over the polytope $2(t_1+t_2+2)P$, here $*$ is the corresponding affix element, which means $*=c,b,a,c$ for $i=1,2,3,4$ respectively. Specially the coloring vector of $\lambda_i^{(0,0)}$ is $[aS_1A_4A_i*]$. Moreover $\delta_i^{(t_1,t_2)}$ is the natural associated $\mathbb{Z}_2^4$-extension of $\lambda_i^{(t_1,t_2)}$.

From 
	\begin{gather}
	\beta^1(M(4P,\delta_i^{(0,0)}))=5+2i,\notag\\
	\beta^1(M(6P,\delta_i^{(0,1)}))=7+2i,\notag\\
	\beta^1(M(8P,\delta_i^{(0,2)}))=9+2i,\notag
	\end{gather}

we have \begin{equation}\label{equation:4}
\beta^1(M(2(t_1+t_2+2)P,\delta_i^{(t_1,t_2)}))+2 = \beta^1(M(2(t_1 + t_2+3)P,\delta_i^{(t_1,t_2+1)})),~~ for~~ i=1,2,3,4.
\end{equation}

From 
\begin{gather}
\beta^1(M(4P,\delta_i^{(0,0)}))=5+2i,\notag\\
\beta^1(M(6P,\delta_i^{(1,0)}))=15+2i,\notag\\
\beta^1(M(8P,\delta_i^{(2,0)}))=25+2i.\notag
\end{gather}

Then for $i=1,2,3,4$,
 we have \begin{equation}\label{equation:5}
\beta^1(M(2(t_1 + t_2+2)P,\delta^{(t_1,t_2)}_i))+10 =\beta^1(M(2(t_1 + t_2+3)P,\delta^{(t_1+1,t_2)}_i)).\end{equation}

By (\ref{equation:4}) and (\ref{equation:5}), we have
\begin{equation}\label{equation:6}
\beta^1(M(nP,\delta^{(t,\frac{n}{2}-2-t)}_i))= n+8t+2i+3,
\end{equation}
for each $n$ is even,  and $0\leqslant t\leqslant \frac{n}{2}-2$.

Moreover, denoted by $\widetilde{\lambda^0}$, $\widetilde{\lambda^1}$ and $\widetilde{\lambda^2}$ the three $\mathbb{Z}_2^3$-coloring characteristic functions of
\begin{gather}
[aS_1A_3a],\notag\\
[aS_1A_3A_3a]\triangleq[aS_1A_3a]+A_3,\notag\\
[aS_1A_3A_3A_3a]\triangleq[aS_1A_3a]+2\cdot A_3,\notag
\end{gather}
\noindent which are the colorings of the polytopes $2P$, $4P$ and $6P$ respectively. Their characteristic vectors are
\begin{gather}
(1,~ 2,~ 4,~ 4,~ 2,~ 7,~ 3,~ 7,~ 5,~ 2,~ 6,~ 2,~ 4,~ 4,~ 2,~ 7,~ 1),\notag\\
(1,~ 2,~ 4,~ 4,~ 2,~ 7,~ 3,~ 7,~ 5,~ 2,~ 6,~ 2,~ 4,~ 4,~ 2,~ 7,~ 3,~ 7,~ 5,~ 2,~ 6,~ 2,~ 4,~ 4,~ 2,~ 7,~ 1),\notag\\
(1,~ 2,~ 4,~ 4,~ 2,~ 7,~ 3,~ 7,~ 5,~ 2,~ 6,~ 2,~ 4,~ 4,~ 2,~ 7,~ 3,~ 7,~ 5,~ 2,~ 6,~ 2,~ 4,~ 4,~ 2,~ 7,~ 3,~ 7,~ 5,~ 2,~ 6,~ 2,~ 4,~ 4,~ 2,~ 7,~ 1).\notag
\end{gather}
\noindent And the natural associated $\mathbb{Z}_2^4$-colorings are $\widetilde{\delta^0}$, $\widetilde{\delta^1}$ and $\widetilde{\delta^2}$. Specially $\widetilde{\lambda^0}=\widetilde{\lambda}$, which also means $\widetilde{\delta^0}=\widetilde{\delta}$. By calculation we obtain that the first Betti numbers of the recovered manifolds are 5, 15 and 25 respectively. Thus we have
 \begin{equation}\label{equation:7}
\beta^1(M((2+2t)P,\widetilde{\delta^t}))=5+10t,
\end{equation}
for each $t \in \mathbb{Z}_{+}$, where $t$ is the times of repeating the last two bricks of $\widetilde{\delta^0}$.

By (\ref{equation:0}), (\ref{equation:6}) and (\ref{equation:7}) we finish the proof of Lemma \ref{lemma: fifth}. And we arrange all the Betti numbers of Lemma \ref{lemma: fifth} together in Table  \ref{table:21}.

\begin{table}[H]\footnotesize
	\setlength{\tabcolsep}{2pt}
	\begin{tabular}{cccccc}
		
		\hline
		$\beta^1(M(2P,\delta))=5$&$\beta^1(M(4P,\delta_1^{(0,0)}))=7$& $\beta^1(M(6P,\delta_1^{(0,1)}))=9$   &$\beta^1(M(8P,\delta_1^{(0,2)}))=11$& $\beta^1(M(10P,\delta_1^{(0,3)}))=13$ &...\\
		&
		$\beta^1(M(4P,\delta_2^{(0,0)}))=9$& $\beta^1(M(6P,\delta_2^{(0,1)}))=11$  &$\beta^1(M(8P,\delta_2^{(0,2)}))=13$& $\beta^1(M(10P,\delta_2^{(0,3)}))=15$ &...\\	
		$\beta^1(M(2P,\widetilde{\delta}))=5$&	$\beta^1(M(4P,\delta_3^{(0,0)}))=11$& $\beta^1(M(6P,\delta_3^{(0,1)}))=13$ &$\beta^1(M(8P,\delta_3^{(0,2)}))=15$& $\beta^1(M(10P,\delta_3^{(0,3)}))=17$ &...\\	
		($\widetilde{\delta}=\widetilde{\delta^0}$)&$\beta^1(M(4P,\delta_4^{(0,0)}))=13$& $\beta^1(M(6P,\delta_4^{(0,1)}))=15$ &$\beta^1(M(8P,\delta_4^{(0,2)}))=17$& $\beta^1(M(10P,\delta_4^{(0,3)}))=19$ &...\\
		\hline
		&& $\beta^1(M(6P,\delta_1^{(1,0)}))=17$ &$\beta^1(M(8P,\delta_1^{(1,1)}))=19$& $\beta^1(M(10P,\delta_1^{(1,2)}))=21$ &...\\
		&$\beta^1(M(4P,\widetilde{\delta^1}))=15$& $\beta^1(M(6P,\delta_2^{(1,0)}))=19$ &$\beta^1(M(8P,\delta_2^{(1,1)}))=21$& $\beta^1(M(10P,\delta_2^{(1,2)}))=23$ &...\\	
		&& $\beta^1(M(6P,\delta_3^{(1,0)}))=21$ &$\beta^1(M(8P,\delta_3^{(1,1)}))=23$& $\beta^1(M(10P,\delta_3^{(1,2)}))=25$ &...\\	
		&& $\beta^1(M(6P,\delta_4^{(1,0)}))=23$ &$\beta^1(M(8P,\delta_4^{(1,1)}))=25$& $\beta^1(M(10P,\delta_4^{(1,2)}))=27$ &...\\
		\hline
		&& &$\beta^1(M(8P,\delta_1^{(2,0)}))=27$& $\beta^1(M(10P,\delta_1^{(2,1)}))=29$ &...\\
		&&$\beta^1(M(6P,\widetilde{\delta^2}))=25$ &$\beta^1(M(8P,\delta_2^{(2,0)}))=29$& $\beta^1(M(10P,\delta_2^{(2,1)}))=31$ &...\\	
		&& &$\beta^1(M(8P,\delta_3^{(2,0)}))=31$& $\beta^1(M(10P,\delta_3^{(2,1)}))=33$ &...\\	
		&& &$\beta^1(M(8P,\delta_4^{(2,0)}))=33$& $\beta^1(M(10P,\delta_4^{(2,1)}))=35$ &...\\
		\hline
		&& && $\beta^1(M(10P,\delta_1^{(3,0)}))=37$ &...\\
		&& &$\beta^1(M(8P,\widetilde{\delta^3}))=35$& $\beta^1(M(10P,\delta_2^{(3,0)}))=39$ &...\\	
		&& && $\beta^1(M(10P,\delta_3^{(3,0)}))=41$ &...\\	
		&& && $\beta^1(M(10P,\delta_4^{(3,0)}))=43$ &...\\
		\hline
		&&&&$\beta^1(M(10P,\widetilde{\delta^4}))=45$&...\\
		\hline
	\end{tabular}
	\caption{$\beta^1$ change for Lemma \ref{lemma: fifth}.}
	\label{table:21}
\end{table}
\end{proof}

Now, from Lemma  \ref{lemma: second} to  Lemma \ref{lemma: fifth}, we finish the proof of  Theorem \ref{theorem: bounding} for  an even $n$.

%---------------------------------------------------------------------------��wu����
\section{\textbf{Proof of Theorem \ref{theorem: bounding} for $n$ is odd}}
%ע��Ⱦɫ˳�� Ȼ������S1��Ⱦɫ����Ķ�Ӧ˵��

In this section, we prove Theorem \ref{theorem: bounding} for  an odd $n \in \mathbb{Z}_{\geq 1}$, which are similar to the arguments in Section 4.
%-------------------------------------------------------------------�ڶ�������һ��Ⱦɫ����綨�����½��ޣ���һ��

\begin{lemma}  \label{lemma: odd1}
	For any odd integer $n\in \mathbb{Z}_{>1}$, there is a non-orientable $\mathbb{Z}_2^3$-coloring $\lambda$  over the polytope $nP$, such that for it natural associated $\mathbb{Z}_2^4$-coloring $\delta$, we have   $\beta^1 (M(nP, \delta))= n$.
\end{lemma}

\begin{proof}
	
The  proof simply parallels that of Lemma \ref{lemma: key}. Here we firstly consider the case $n=3$. We keep the notations  $S_1=(24247)$, $S_2=(65372)$, $S_3=(35716)$, $a=1$, $b=1$, $c=4$ and three compatible pairs $S_1S_2=(24247 \  65372)$, $S_1S_3=(24247 \ 35716)$, $S_2S_3=(65372 \ 57163$ as we  set in Lemma \ref{lemma: key}.

 Following the same method  in Lemma \ref{lemma: key}, we firstly construct a non-orientable  $\mathbb{Z}_2^3$-coloring $\lambda$ over the polytope $3P$, whose coloring vector and characteristic vector are
 $$[aS_1S_3S_1S_3c]$$
and
$$(1,~ 2,~ 4,~ 4,~ 2,~ 7,~ 7,~ 1,~ 5,~ 6,~ 3,~ 2,~ 4,~ 4,~ 2,~ 7,~ 7,~ 1,~ 5,~ 6,~ 3,~ 4).$$
So we can figure out $\beta^1(M(3P,\delta))=3$ from Theorem  \ref{theorem: ChoiP}, where $\delta$ is the natural $\mathbb{Z}_2^4$-extension of $\lambda$. 

And then we repeat the last two bricks for $t$ times to construct the expected coloring vector over the polytope $(3+2t)P$. The corresponding characteristic function is denoted by $\lambda^t$. By Theorem  \ref{theorem: NakayamaN} and Proposition \ref{remark:1} we can get its unique admissible extension $\delta^t$, a $\mathbb{Z}_2^4$-coloring characteristic function. The sum of every column of characteristic matrix of $\delta^t$ is 1 mod 2. That is, $M((3+2t)P,\delta^t)$ is the orientable double cover of the non-orientable manifold $M((3+2t)P,\lambda^t)$. Finally we figure out the increasing law of $\beta^1(M((3+2t)P,\delta^t))$ as shown in Table \ref{table:22}.
	
	\begin{table}[H]
		\begin{tabular}{|c|c|c|c|c|c|}
			\hline
			& $n=3$ & $n=5$ & $n=7$ & $\cdots$ & $n=2a+1$,  $a\in \mathbb{Z}_{> 0}$\\
			\hline
			1 & 1 & 1 & 1 & $\cdots$ &1\\
			\hline
			2 & 1 & 1 & 1 & $\cdots$ &1\\
			\hline
			3 & 0 & 0 & 0 & $\cdots$ &0\\
			\hline
			4 & 0 & 1 & 2 & $\cdots$ & $a-1$\\
			\hline
			5 & 0 & 0 & 0 & $\cdots$ &0\\
			\hline
			6 & 0 & 1 & 2 & $\cdots$ &$a-1$\\
			\hline
			7 & 1 & 1 & 1 & $\cdots$ &1\\
			\hline
			8 & 0 & 0 & 0 & $\cdots$ &0\\
			\hline
			9 & 0 & 0 & 0 & $\cdots$ &0\\
			\hline
			10 & 0 & 0 & 0 & $\cdots$ &0\\
			\hline
			11 & 0 & 0 & 0 & $\cdots$ &0\\
			\hline
			12 & 0 & 0 & 0 & $\cdots$ &0\\
			\hline
			13 & 0 & 0 & 0 & $\cdots$ &0\\
			\hline
			14 & 0 & 0 & 0 & $\cdots$ &0\\
			\hline
			15 & 0 & 0 & 0 & $\cdots$ &0\\
			\hline
			total ~$\beta^1$ & 3 & 5 & 7 & $\cdots$ & $2(a-1)+3=n$\\
			\hline
		\end{tabular}
		\caption{$\beta^1(M(nP,\delta^t))$ for $n=3+2t$ in Lemma \ref{lemma: odd1}.}
		\label{table:22}
	\end{table}
	
	By now we finally fulfill the proof of Lemma \ref{lemma: odd1}.
\end{proof}

%-------------------------------------------------------------------�ڶ�����������Ⱦɫ����綨���ޣ���һ��
\begin{lemma}\label{lemma: odd2}
	For any odd number $n\in \mathbb{Z}_{>1}$ and any odd integer $k\in[5n-9,5n+3]$, there is a non-orientable $\mathbb{Z}_2^3$-coloring $\lambda$  over the polytope $nP$, such that for it natural associated  $\mathbb{Z}_2^4$-coloring $\delta$, we have  $\beta^1 (M(nP, \delta))=k$.
\end{lemma}

\begin{proof}
	Similar to Lemma \ref{lemma: third}, we start at $n=3$ and construct six suitable characteristic vectors whose corresponding manifolds' Betti numbers would add up by $10t$ when repeating the last pair of coloring bricks for $t$ times.
	
	First of all we claim some bricks and affixes in Table  \ref{table:23}. Those bricks are useful in constructing coloring vectors of the expected $\mathbb{Z}_2^3$-coloring characteristic function $\lambda_i^t$.
	
	\begin{table}[H]
		\begin{tabular}{|c|c|c|c|c|c|}
			\hline
			$a\in\mathcal{A}(S_1)$ &brick $S_1$& $b\in\mathcal{A}(S_2)$& brick $S_2$&brick $S_3$ &compatible pair that being repeated\\
			\hline
			1 &24247 &1&65372&35716&(53726 71635)=$S_2S_3\triangleq A_1$ \\
			&&&&&(24724 37265)=$S_1S_2\triangleq A_2$ \\
			&&&&&(53726 74242)=$S_2S_1\triangleq A_3$ \\	
			\hline
		\end{tabular}
		\caption{Bricks and affixes for Lemma \ref{lemma: odd2}.  (1) }
		\label{table:23}
	\end{table}
	
	For every $i=1, 2, 3$, let $\lambda_i^0$, $\lambda_i^1$ and $\lambda_i^2$ be the three $\mathbb{Z}_2^3$-coloring characteristic functions of the three colorings over the polytopes  $3P$, $5P$ and $7P$ respectively as shown in Table \ref{table:24}. Here $t$ represents how many times the last compatible pair of $\lambda_i^0$ has been repeated.

	\begin{table}[H]
		\begin{tabular}{|c|c|c|c|}
			\hline
			\diagbox{$i$}{$t$}& 0&1& 2\\
			\hline
			0 & $[aA_1A_2a]$ & $[aA_1A_2A_2a]$ &$[aA_1A_2A_2A_2a]$\\
			\hline
			1 & $[aA_1A_3a]$& $[aA_1A_3A_3a]$ &$[aA_1A_3A_3A_3a]$\\
			\hline
			2 & $[aA_1A_1a]$& $[aA_1A_1A_1a]$ &$[aA_1A_1A_1A_1a]$\\
			\hline
		\end{tabular}
		\caption{Coloring vectors of $\lambda_i^t$ in Lemma \ref{lemma: odd2}.}
		\label{table:24}
	\end{table}

 Let $\delta_{i}^t$ be the natural $\mathbb{Z}_2^4$-extensions of $\lambda_{i}^t$, for $i=1, 2, 3$.
 Then by calculation we have the first Betti number  of the recovered manifold of $M((3+2t)P, \delta_{i}^t)$ is $5+2i+10t$, for $t=0$, $1$  and $2$. Thus
\begin{equation}\label{equation:21}
\beta^1(M((3+2t)P,\delta_1^t))=7+10t,
\end{equation}
\begin{equation}\label{equation:22}
\beta^1(M((3+2t)P,\delta_2^t))=9+10t,
\end{equation}
\begin{equation}\label{equation:23}
\beta^1(M((3+2t)P,\delta_3^t))=11+10t,
\end{equation} 
for each $t\in \mathbb{Z}_{\geqslant0}$.

Again, we construct some bricks and affixes as shown in Table \ref{table:25}. Those items are useful in building up the  coloring vectors of the expected   characteristic function $\widetilde{\lambda_i^t}$.

	\begin{table}[H]
		\begin{tabular}{|c|c|c|c|c|c|c|c|c|}
			\hline
			brick $S_1$& brick $S_2$ & brick $S_3$ &brick $S_4$& $d\in\mathcal{A}(S_4)$ &brick $S_5$ & $e\in\mathcal{A}(S_5)$ &brick $S_6$ &$f\in\mathcal{A}(S_6)$\\
			\hline
			34246 &26513 &31245 &26416 &3&16416&3&46452&3 \\
			\hline
			\multicolumn{9}{|c|}{compatible pair that being repeated}\\
			\hline
			\multicolumn{9}{|c|}{(34246 26513)=$S_1S_2\triangleq A_0$}\\
			\multicolumn{9}{|c|}{(31245 26416)=$S_3S_4\triangleq A_1$}\\
			\multicolumn{9}{|c|}{(31245 16416)=$S_3S_5\triangleq A_2$}\\
			\multicolumn{9}{|c|}{(31245 46452)=$S_3S_6\triangleq A_3$}\\			
			\hline
		\end{tabular}
		\caption{Bricks and affixes for Lemma \ref{lemma: odd2}. (2)}
		\label{table:25}
	\end{table}

	For every $i=1, 2, 3$, let $\widetilde{\lambda_i^0}$, $\widetilde{\lambda_i^1}$ and $\widetilde{\lambda_i^2}$ be the three $\mathbb{Z}_2^3$-coloring characteristic functions of the three colorings over the polytopes $3P$, $5P$ and $7P$ respectively as shown in Table \ref{table:26}. Here $t$ represents how many times the last compatible pair of $\widetilde{\lambda_i^0}$ has been repeated.
	
	\begin{table}[H]
		\begin{tabular}{|c|c|c|c|}
			\hline
			\diagbox{$i$}{$t$}&0&1&2\\
			\hline
			0 & $[aA_0A_1d]$& $[aA_0A_1A_1d]$ &$[aA_0A_1A_1A_1d]$\\
			\hline
			1 & $[aA_0A_2e]$& $[aA_0A_2A_2e]$ &$[aA_0A_2A_2A_2e]$\\
			\hline
			2 & $[aA_0A_3f]$& $[aA_0A_3A_3f]$ &$[aA_0A_3A_3A_3f]$\\
			\hline
		\end{tabular}
		\caption{Coloring vectors of $\widetilde{\lambda_i^t}$  in Lemma \ref{lemma: odd2}.}\label{table:26}
	\end{table}

Let $\widetilde{\delta_{i}^t}$ be the natural $\mathbb{Z}_2^4$-extensions of $\widetilde{\lambda_{i}^t}$, for $i=1, 2, 3$.
 Then by calculation we have the first Betti number of the recovered manifold of $\widetilde{\delta_i^t}$ is $11+2i+10t$, for $t=0$, $1$  and $2$. Thus
 \begin{equation}\label{equation:24}
 \beta^1(M((3+2t)P,\widetilde{\delta_1^t}))=13+10t,
 \end{equation}
 \begin{equation}\label{equation:25}
 \beta^1(M((3+2t)P,\widetilde{\delta_2^t}))=15+10t,
 \end{equation}
 \begin{equation}\label{equation:26}
 \beta^1(M((3+2t)P,\widetilde{\delta_3^t}))=17+10t,
 \end{equation} 
for each  $t\in \mathbb{Z}_{+}$.
	
Thus by (\ref{equation:21}), (\ref{equation:22}), (\ref{equation:23}), (\ref{equation:24}), (\ref{equation:25}) and (\ref{equation:26}), namely gathering these six special families of characteristic functions, we fulfill the proof of Lemma \ref{lemma: odd2}.
	
\end{proof}

%--------------------------------------------------------------------------------------------------ǰ�˲������仯������
\begin{lemma} \label{lemma: odd3}
	For any odd integer $n\in\mathbb{Z}_{>1}$ and any odd integer $k \in[1,n-1]$, there is a non-orientable $\mathbb{Z}_2^3$-coloring $\lambda$  over the polytope $nP$, such that for its natural associated  $\mathbb{Z}_2^4$-coloring $\delta$, we have $\beta^1 (M(nP,\delta)) = k$.
\end{lemma}

\begin{proof}
	
	Using the notations of $\{S_i\}_{i=1}^4$, $\{A_i\}_{i=1}^3$ and affixes $a$, $b$ as  defined in Table \ref{table:27}.
	
	\begin{table}[H]
		\begin{tabular}{|c|c|c|c|c|c|c|}
			\hline
			$a\in\mathcal{A}(S_1)$ &brick $S_1$& brick $S_2$& $b\in\mathcal{A}(S_2)$&brick $S_3$ &brick $S_4$&compatible pair that being repeated\\
			\hline
			1 &24247 &17532&4&53176&53147&(24247 17532)=$S_1S_2\triangleq A_1$ \\	
			& &&&&&(53176 17532)=$S_3S_2\triangleq A_2$ \\	
			& &&&&&(53147 17532)=$S_4S_2\triangleq A_3$ \\	
			\hline
		\end{tabular}
		\caption{Bricks and affixes for $\lambda^{(t_1,t_2)}$ of Lemma \ref{lemma: odd3}.}
		\label{table:27}
	\end{table}
	
	Firstly we seek out a non-orientable $\mathbb{Z}_2^3$-coloring characteristic function $\lambda$ on the polytope $3P$, whose coloring vector is $[aA_1 A_2 b]\triangleq [aAb]$. Its natural $\mathbb{Z}_2^4$-extension  is $\delta$. Let $\lambda^{(t_1,t_2)}$ to represent the $\mathbb{Z}_2^3$-coloring  characteristic function of
$$[aA_1A_2\underbrace{A_2,...,A_2}_{t_1}\underbrace{A_3,...,A_3}_{t_2}b]\triangleq A+t_1A_2+t_2A_3,$$
which is over the polytope  $(2(t_1+t_2)+3))P$.
Moreover $\delta^{(t_1,t_2)}$ is the natural associated $\mathbb{Z}_2^4$-extension of $\lambda^{(t_1,t_2)}$ over the polytope  $(2(t_1+t_2)+3))P$. Specially  $\lambda^{(0,0)}=\lambda$. It is necessary to make sure the compatibility of $S_2$ and $S_3$ if we want to glue $A_2$ after $A_3$. And here everything has checked to be fine for the compatibility naturally lies in the pair $A_2$. Thus we have Betti numbers as shown in Table \ref{table:28}.
	
	\begin{table}[H]
		\begin{tabular}{|c|c|c|c|}
			\hline
			$\beta^1(M(3P,\delta^{(0,0)}))=1$ & $\beta^1(M(5P,\delta^{(1,0)}))=1$&$\beta^1(M(7P,\delta^{(2,0)}))=1$&...\\
			&$\beta^1(M(5P,\delta^{(0,1)}))=3$ &$\beta^1(M(7P,\delta^{(1,1)}))=3$&... \\	
			& &$\beta^1(M(7P,\delta^{(0,2)}))=5$&...\\	
			\hline
		\end{tabular}
		\caption{$\beta^1$ change of $\delta^{(t_1,t_2)}$ in Lemma \ref{lemma: odd3}.}
		\label{table:28}
	\end{table}
	
	Then by the analysis in Lemma \ref{lemma: key}, the Betti number sequence follows certain pattern as the characteristic vectors and adjacent matrices  change regularly. So if we have the first three numbers lying in an arithmetic progression, then  the Betti number sequence follows an arithmetic progression. 
	
	From $$\beta^1(M(3P,\delta^{(0,0)}))=1,$$
$$\beta^1(M(5P,\delta^{(1,0)}))=1,$$
$$\beta^1(M(7P,\delta^{(2,0)}))=1,$$
we have
\begin{equation}\label{equation:8}
\beta^1(M((2(t_1+t_2)+3)P,\delta^{(t_1,t_2)})) = \beta^1(M((2(t_1 + t_2)+5)P,\delta^{(t_1+1,t_2)})).
\end{equation}

From $$\beta^1(M(3P,\delta^{(0,0)}))=1,$$
$$\beta^1(M(5P,\delta^{(0,1)}))=3,$$
$$\beta^1(M(7P,\delta^{(0,2)}))=5,$$
 we have
\begin{equation}\label{equation:9}
\beta^1(M((2(t_1 + t_2)+3)P,\delta^{(t_1,t_2)}))+2 =\beta^1(M((2(t_1 + t_2)+5)P,\delta^{(t_1,t_2+1)})).
\end{equation}
	
By (\ref{equation:8}) and (\ref{equation:9}) we can obtain
 \begin{equation}\label{equation:10}
\beta^1(M(nP,\delta^{(t,\frac{n-3}{2}-t)}))= n-2-2t,
 \end{equation}
for each  $n$ is odd, $n\in\mathbb{Z}_{\geq 3}$ and $0\leqslant t\leqslant \frac{n-3}{2}$.

	Thus we finish the proof of Lemma \ref{lemma: odd3}.
\end{proof}

%--------------------------------------------------------------------------------------------------�в��������仯
\begin{lemma} \label{lemma: odd4}
	For any odd integer $n\in\mathbb{Z}_{>1}$ and any odd integer $k \in[n+1,5n-9]$, there is a non-orientable $\mathbb{Z}_2^3$-coloring $\lambda$ over the polytope $nP$, such that for the natural associated  $\mathbb{Z}_2^4$-coloring $\delta$, we have $\beta^1 (M(nP,\delta)) = k$.
\end{lemma}

\begin{proof}

	Keeping the notations of $S_1=(24247)$, $S_2=(65372)$, $S_3 =(35716)$ and affixes $a=1$, $b=1$,  $c=4$ that we used in Lemma \ref{lemma: key}. Four compatible pairs which are useful in this proof are arranged in Table \ref{table:29}.
	
	\begin{table}[H]
		\begin{tabular}{|c|c|c|c|c|c|}
			\hline
			$a\in\mathcal{A}(S_1)$ &brick $S_1$& $b\in\mathcal{A}(S_2)$ &brick $S_2$& $c\in\mathcal{A}(S_3)$&brick $S_3$ \\
			\hline
			1 &24247 &1&65372&4&35716\\
			\hline
			\multicolumn{6}{|c|}{compatible pair that being repeated}\\
			\hline
			\multicolumn{6}{|c|}{(42472 57163)=$S_1S_3\triangleq A_1$}\\
			\multicolumn{6}{|c|}{(42472 53726)=$S_1S_2\triangleq A_2$}\\
			\multicolumn{6}{|c|}{(65372 72424)=$S_2S_1\triangleq A_3$}\\
			\multicolumn{6}{|c|}{(65372 57163)=$S_2S_3\triangleq A_4$}\\
			\hline
		\end{tabular}
		\caption{Bricks and affixes for Lemma \ref{lemma: odd4}.}
		\label{table:29}
	\end{table}
	Firstly we construct a non-orientable $\mathbb{Z}_2^3$-colorings $\lambda$ over the polytope $3P$, whose coloring vector is $[bS_2 S_3 S_1 S_3c]$. Its natural associated $\mathbb{Z}_2^4$-extension is denoted by $\delta$. By calculation we have
	\begin{equation}\label{equation:11}
	\beta^1(M(3P,\delta))=5.
	\end{equation}
	
	By $\lambda_i^{t-1}$, for each $t\in \mathbb{Z}_{\geq 1}$ and $i=1, 2, 3$, or $4$, we refer to the non-orientable $\mathbb{Z}_2^3$-coloring characteristic function $\lambda$ on the polytope $(2t+3)P$ with coloring vector $$[bA_4\underbrace{A_1,...,A_1}_tA_i*].$$
Here $*$ is the corresponding affix element, which means $*=c, b, a, c$ for $i=1,2,3,4$ respectively. Specially  $\lambda_1^t$ is obtained by inserting $(t+1)$ times $A_1$ into the coloring vector of $\lambda$. Denoted by $\delta_i^{t-1}$ the natural $\mathbb{Z}_2^4$-extension of $\lambda_i^{t-1}$, from
$$\beta^1(M(5P,\delta_i^0))=5+2i,$$
$$\beta^1(M(7P,\delta_i^1))=7+2i,$$
 $$\beta^1(M(9P,\delta^2))=9+2i,$$
 we have
	\begin{equation}\label{equation:12}
	\beta^1(M((2t+3)P,\delta_i^{t-1}))+2 = \beta^1(M((2t+5)P,\delta_i^t))
	\end{equation}
for $i=1,2,3,4$.

	Next, we select three non-orientable $\mathbb{Z}_2^3$-colorings  $\widetilde{\lambda^0}$, $\widetilde{\lambda^1}$ and $\widetilde{\lambda^2}$ on the polytopes $3P$, $5P$ and $7P$, whose coloring vectors are
$$[aA_1A_3a],$$
$$[aA_1A_3A_3a]\triangleq[a1A_3a]+A_3,$$
$$[aA_1A_3A_3A_3a]\triangleq[a1A_3a]+2\cdot A_3$$ 
respectively. And the natural $\mathbb{Z}_2^4$-extensions are denoted as $\widetilde{\delta^0}$,  $\widetilde{\delta^1}$ and $\widetilde{\delta^2}$. By calculating we have $$\beta^1(M(3P,\widetilde{\delta^0}))=5,$$
 $$\beta^1(M(5P,\widetilde{\delta^1}))=15,$$
 $$\beta^1(M(7P,\widetilde{\delta^2}))=25.$$
Now for each $t\in \mathbb{Z}_{\geq 1}$, denoting $\widetilde{\lambda^{t-1}}$ as the  $\mathbb{Z}_2^3$-coloring characteristic function of $[aA_1\underbrace{A_3,...,A_3}_ta]$ over the polytope $(2t+1)P$ and $\widetilde{\delta^{t-1}}$ as its  natural $\mathbb{Z}_2^4$-extension. Then we have
	\begin{equation}
	\label{equation:13}
	\beta^1(M((2t+1)P,\widetilde{\delta^{t-1}}))=10t-5,
	\end{equation}
for each $t\in \mathbb{Z}_{\geq 1}$.

	Now using $\lambda_i^{(t_1-1, t_2)}$ to represent the $\mathbb{Z}_2^3$-coloring  characteristic function of the coloring vector $$[aA_1\underbrace{A_3,...,A_3}_{t_1}A_4\underbrace{A_1, ..., A_1}_{t_2}A_i*]\triangleq [aA_1]+t_1A_3+A_4+t_2A_1+[A_i*]$$
over the polytope $(2(t_1+t_2)+5)P$.
 In particular, the coloring vector of $\lambda_i^{(0, 0)}$ is $[aA_1A_3A_4A_i*]$. And $\delta_i^{(t_1-1, t_2)}$ is the natural $\mathbb{Z}_2^4$-extension of $\lambda_i^{(t_1-1, t_2)}$ over the polytope $(2(t_1+t_2)+5)P$.
	
From $$\beta^1(M(7P,\delta_i^{(0, 0)}))=5+2i,$$
 $$\beta^1(M(9P,\delta_i^{(0, 1)}))=7+2i,$$
$$\beta^1(M(11P, \delta_i^{(0, 2)}))=9+2i,$$
for $i=1, 2, 3, 4$,
we have
\begin{equation}
\label{equation:14}
\beta^1(M((2(t_1 + t_2)+5)P, \delta^{(t_1-1, t_2)}))+2 =\beta^1(M((2(t_1 + t_2)+7)P, \delta^{(t_1-1, t_2+1)})),
\end{equation}	
for each $t\in \mathbb{Z}_{\geq 1}$.

 From $$\beta^1(M(7P, \delta_i^{(0,0)}))=5+2i,$$
$$\beta^1(M(9P,\delta_i^{(1,0)}))=15+2i,$$
$$\beta^1(M(11P,\delta_i^{(2,0)}))=25+2i,$$
for $i=1, 2, 3, 4$,
we have
 \begin{equation}
 \label{equation:15}
 \beta^1(M((2(t_1 + t_2)+5)P, \delta^{(t_1-1, t_2)}_i))+10 =\beta^1(M((2(t_1 + t_2)+7)P, \delta^{(t_1, t_2)}_i)),
 \end{equation}
for each $t\in \mathbb{Z}_{\geq 1}$.

By (\ref{equation:14}) and (\ref{equation:15}),  we have
\begin{equation}
\label{equation:16}
 \beta^1(M(nP,\delta^{(t, \frac{n-1}{2}-3-t)}_i))=n+2i+8t,
\end{equation}
for each $n$ is odd, $n\in\mathbb{Z}_{\geq 7}$ and  $0\leqslant t\leqslant \frac{n-7}{2}$.
	
	Thus from (\ref{equation:11}), (\ref{equation:12}), (\ref{equation:13}) and (\ref{equation:16}) we finish the proof of Lemma \ref{lemma: odd4}. All the Betti numbers of Lemma \ref{lemma: odd4} have been arrange together in Table  \ref{table:30}.
	
	\begin{table}[H]
		\footnotesize
		\setlength{\tabcolsep}{7pt}
		\begin{tabular}{cccccc}
			\hline
			$\beta^1(M(3P,\widetilde{\delta}))=5$&
			$\beta^1(M(5P,\delta_1^0))=7$& $\beta^1(M(7P,\delta_1^1))=9$   &$\beta^1(M(9P,\delta_1^2))=11$& $\beta^1(M(11P,\delta_1^3))=13$ &...\\
			$\beta^1(M(3P,\widetilde{\delta^0}))=5$&
			$\beta^1(M(5P,\delta_2^0))=9$& $\beta^1(M(7P,\delta_2^1))=11$  &$\beta^1(M(9P,\delta_2^2))=13$& $\beta^1(M(11P,\delta_2^3))=15$ &...\\	
			&$\beta^1(M(5P,\delta_3^0))=11$& $\beta^1(M(7P,\delta_3^1))=13$ &$\beta^1(M(9P,\delta_3^2))=15$& $\beta^1(M(11P,\delta_3^3))=17$ &...\\	
			&$\beta^1(M(5P,\delta_4^0))=13$& $\beta^1(M(7P,\delta_4^1))=15$ &$\beta^1(M(9P,\delta_4^2))=17$& $\beta^1(M(11P,\delta_4^3))=19$ &...\\
			\hline
			&& $\beta^1(M(7P, \delta_1^{(0, 0)}))=17$ &$\beta^1(M(9P,\delta_1^{(0,1)}))=19$& $\beta^1(M(11P,\delta_1^{(0,2)}))=21$ &...\\
			&$\beta^1(M(5P, \widetilde{\delta^1}))=15$	& $\beta^1(M(7P,\delta_2^{(0,0)}))=19$ &$\beta^1(M(9P,\delta_2^{(0,1)}))=21$& $\beta^1(M(11P,\delta_2^{(0, 2)}))=23$ &...\\	
			&& $\beta^1(M(7P, \delta_3^{(0,0)}))=21$ &$\beta^1(M(9P,\delta_3^{(0,1)}))=23$& $\beta^1(M(11P,\delta_3^{(0,2)}))=25$ &...\\	
			&& $\beta^1(M(7P,\delta_4^{(0,0)}))=23$ &$\beta^1(M(9P,\delta_4^{(0,1)}))=25$& $\beta^1(M(11P,\delta_4^{(0,2)}))=27$ &...\\
			\hline
			&&		 &$\beta^1(M(9P,\delta_1^{(1,0)}))=27$& $\beta^1(M(11P,\delta_1^{(1,1)}))=29$ &...\\
			&&$\beta^1(M(7P,\widetilde{\delta^2}))=25$&$\beta^1(M(9P,\delta_2^{(1,0)}))=29$& $\beta^1(M(11P,\delta_2^{(1,1)}))=31$ &...\\	
			&& &$\beta^1(M(9P,\delta_3^{(1,0)}))=31$& $\beta^1(M(11P,\delta_3^{(1,1)}))=33$ &...\\	
			&& &$\beta^1(M(9P,\delta_4^{(1,0)}))=33$& $\beta^1(M(11P,\delta_4^{(1,1)}))=35$ &...\\
			\hline
			&& && $\beta^1(M(11P,\delta_1^{(2,0)}))=37$ &...\\
			&& &$\beta^1(M(9P, \widetilde{\delta^3}))=35$& $\beta^1(M(11P,\delta_2^{(2,0)}))=39$ &...\\	
			&& && $\beta^1(M(11P, \delta_3^{(2,0)}))=41$ &...\\	
			&& && $\beta^1(M(11P, \delta_4^{(2,0)}))=43$ &...\\
			\hline
			&&&&&\\
			&& &&$\beta^1(M(11P,\widetilde{\delta^4}))=45$&...\\
			\hline
		\end{tabular}
		\caption{$\beta^1$ change for Lemma \ref{lemma: odd4}.}
		\label{table:30}
	\end{table}
\end{proof}

\begin{lemma}\label{lemma:odd5}
	For $k$ an odd integer and $k\in[1, 7]$, there is a non-orientable $\mathbb{Z}_2^3$-coloring over the dodecahedron $P$, such that for its natural associated $\mathbb{Z}_2^4$-coloring $\delta$,  we have $\beta^1(M(P,\delta))=k$.
\end{lemma}

\begin{proof}
	We only need to give out the satisfied characteristic functions to accomplish this lemma, which is given in Table \ref{table:31}.
	\begin{table}[H]
		\small
		\begin{tabular}{|c|c|c|}
			\hline
			& $\lambda$ & $\beta^1(M(P,\delta))$\\
			\hline
			1 & (1, 2, 4, 4, 2, 7, 1, 7, 7, 5, 6, 4) & 1 \\
			\hline
			2 & (1, 2, 4, 4, 2, 7, 7, 3, 1, 5, 4, 2) & 3 \\
			\hline
			3 & (1, 2, 4, 4, 2, 7, 3, 5, 5, 6, 3, 1) & 5 \\
			\hline
			4 & (1, 2, 4, 5, 2, 6, 3, 6, 5, 4, 3, 1) & 7\\
			\hline
		\end{tabular}
		\caption{$\mathbb{Z}_{2}^{3}$-colorings and $\beta^1$ of their natural  $\mathbb{Z}_{2}^{4}$-extensions of Lemma \ref{lemma:odd5}.}
		\label{table:31}
	\end{table}
\end{proof}

Now, from Lemma  \ref{lemma: odd1} to  Lemma \ref{lemma:odd5}, we finish the proof of  Theorem \ref{theorem: bounding} for  an odd $n$, thus together with Section 4, we finish the proof of  Theorem \ref{theorem: bounding}.
\bibliographystyle{amsplain}

\end{document}